\documentclass[a4paper, 11pt]{amsart}
\usepackage{amsmath,amssymb,amscd}
\usepackage{amsfonts}
\usepackage[english]{babel}
\usepackage[leqno]{amsmath}
\usepackage{amssymb,amsthm}
\usepackage{mathrsfs} 
\usepackage{amscd}
\usepackage{enumerate}
\usepackage{epsfig}
\usepackage{relsize}
\usepackage{layout}
\usepackage{fullpage}
\usepackage[usenames,dvipsnames]{xcolor}
\usepackage[backref=page]{hyperref}
\usepackage{tikz}
\hypersetup{
 colorlinks,
 citecolor=Green,
 linkcolor=Red,
 urlcolor=Blue}
\usepackage[matrix,arrow,tips,curve]{xy}
\input{xy}
\xyoption{all}
\definecolor{darkgreen}{rgb}{0.0, 0.7, 0.0}
\definecolor{purple}{rgb}{0.5, 0.0, 0.5}
\definecolor{red}{rgb}{0.8, 0.2, 0.0}

\newenvironment{sis}{\left\{\begin{aligned}}{\end{aligned}\right.}
\newtheorem{thm}{Theorem}[section]
\newtheorem{lemma}[thm]{Lemma}
\newtheorem{sublemma}[thm]{Sublemma}

\newtheorem{prop}[thm]{Proposition}
\newtheorem{cor}[thm]{Corollary}

\newtheorem{fact}[thm]{Fact}
\newtheorem{theoremalpha}{Theorem}

\numberwithin{equation}{section}
\theoremstyle{definition}
\newtheorem{defi}[thm]{Definition}

\newtheorem{nota}[thm]{}
\newtheorem{question}[thm]{Question}
\newtheorem*{question*}{Question}

\theoremstyle{remark}
\newtheorem{remark}[thm]{Remark}
\newtheorem{remarks}[thm]{Remarks}
\newtheorem{example}[thm]{Example}
\newcommand{\Z}{\mathbb{Z}}
\newcommand{\Q}{\mathbb{Q}}
\newcommand{\R}{\mathbb{R}}
\newcommand{\NN}{\mathbb{N}}
\newcommand{\Pic}{\operatorname{Pic}}

\newcommand{\un}{\underline}
\newcommand{\ov}{\overline}
\newcommand{\kbar}{\bar{k}}
\newcommand{\Gal}{\mathrm{Gal}}
\newcommand{\sep}{\mathrm{sep}}
\newcommand{\Ql}{\mathbb{Q}_\ell}
\newcommand{\isomarrow}{\xrightarrow{\sim}}
\newcommand{\ev}{\mathrm{ev}}
\newcommand{\Zl}{\mathbb{Z}_\ell}
\newcommand{\CSp}{\mathrm{CSp}}
\newcommand{\mG}{\mathbb{G}}
\newcommand{\mult}{\mathrm{m}}
\newcommand{\Sp}{\mathrm{Sp}}
\newcommand{\Kbar}{\bar{K}}
\newcommand{\Ksep}{K^{\mathrm{sep}}}
\newcommand{\mono}{\mathrm{mon}}
\newcommand{\Fbar}{\bar{F}}
\newcommand{\mZ}{\mathbb{Z}}
\newcommand{\mC}{\mathbb{C}}
\newcommand{\wt}{\widetilde}
\newcommand{\wh}{\widehat}

\DeclareMathOperator{\Spec}{Spec \:}

\def \Im{{\rm Im}}

\DeclareMathOperator{\ri}{ri}

\DeclareMathOperator{\id}{id}

\DeclareMathOperator{\uguale}{=}
\def \PP{\mathbb{P}}

\def \P{\mathcal P}
\def \F{\mathcal F}

\def \L{\mathcal L}
\def \E{\mathcal E}
\def\O{\mathcal O}

\def \D{\mathcal D}

\def\M0{\mathcal M^0}
\def\t{\bf t}

\newcommand{\Aut}{{\operatorname{Aut}}}

\newcommand{\End}{{\operatorname{End}}}

\DeclareMathOperator{\codim}{{codim}}
\DeclareMathOperator{\lin}{{lin}}
\DeclareMathOperator{\ldim}{{ldim}}

\DeclareMathOperator{\Ker}{{Ker}}
\DeclareMathOperator{\Proj}{{Proj}}
\DeclareMathOperator{\Sym}{{Sym}}

\DeclareMathOperator{\CH}{{CH}}

\newcommand{\bbR}{{\mathbb R}}
\newcommand{\bbN}{{\mathbb N}}
\newcommand{\bbC}{{\mathbb C}}
\newcommand{\bbQ}{{\mathbb Q}}

\newcommand{\bS}{{\mathbf S}}

\newcommand{\gon}{\operatorname{gon}}
\newcommand{\Eff}{\operatorname{Eff}}
\newcommand{\Pseff}{\operatorname{Pseff}}
\newcommand{\Nef}{\operatorname{Nef}}

\newcommand{\Efft}{{^{\rm t} \hskip -.05cm \operatorname{Eff}}}
\newcommand{\Psefft}{{^{\rm t} \hskip -.05cm \operatorname{Pseff}}}
\newcommand{\Neft}{{^{\rm t} \hskip -.05cm \operatorname{Nef}}}

\newcommand{\cone}{\operatorname{cone}}
\newcommand{\aff}{\operatorname{aff}}
\newcommand{\conv}{\operatorname{conv}}
\renewcommand{\a}{\alpha}
\renewcommand{\r}{\rho}
\renewcommand{\t}{\tau}

\title{%
Effective cycles on the symmetric product of a curve, I: the diagonal cone. \\ 
\MakeLowercase{(with an appendix by} B\MakeLowercase{en} M\MakeLowercase{oonen)}}

\author{Francesco Bastianelli, Alexis Kouvidakis, Angelo Felice Lopez and Filippo Viviani}

\thanks{Research partially supported by INdAM (GNSAGA) and by the MIUR national projects``Geometria delle variet\`a algebriche" PRIN 2010-2011 and ``Spazi di moduli e applicazioni" FIRB 2012.}

\address{\hskip -.43cm Dipartimento di Matematica, Universit\`a degli Studi di Bari Aldo Moro, Via Edoardo Orabona, 4, 70125 Bari (Italy). e-mail {\tt francesco.bastianelli@uniba.it}}

\address{\hskip -.43cm Department of Mathematics and Applied Mathematics, University of Crete, GR-70013 Heraklion, Greece. e-mail {\tt kouvid@math.uoc.gr}}

\address{\hskip -.43cm Dipartimento di Matematica e Fisica, Universit\`a di Roma Tre, Largo San Leonardo Murialdo 1, 00146 Roma, Italy. e-mail {\tt lopez@mat.uniroma3.it, filippo.viviani@gmail.com}}

\address{\hskip -.43cm Radboud University, IMAPP, PO Box 9010, 6500GL Nijmegen, The Netherlands. e-mail {\tt b.moonen@science.ru.nl}}

\begin{document}


\begin{abstract}
In this paper we investigate the cone $\Pseff_n(C_d)$ of pseudoeffective $n$-cycles in the symmetric product $C_d$ of a smooth curve $C$. We study the convex-geometric properties of the cone $\D_n(C_d)$ generated by the $n$-dimensional diagonal cycles. In particular we determine its extremal rays and we prove that $\D_n(C_d)$ is a perfect face of $\Pseff_n(C_d)$ along which $\Pseff_n(C_d)$ is locally finitely generated. 
\end{abstract}

\maketitle

\section{Introduction}
\label{intro}

A large amount of the beautiful classical results about the geometry of a smooth projective irreducible curve $C$ of genus $g$, such as for example the Brill-Noether theory \cite{GAC1}, can be stated and proved by studying, for some $d \ge 2$, the $d$-fold 
symmetric product $C_d$ of $C$, that is the smooth projective variety parameterizing unordered $d$-tuples of points of $C$,
and the Abel-Jacobi morphism 
$$\alpha_d\colon C_d \to \Pic^d(C)$$ 
which sends an effective  divisor $D\in C_d$ into its associated line bundle $\O_C(D)\in \Pic^d(C)$.


The aim of this paper and its sequel \cite{BKLV} is to study the cone of effective cycles (and its closure) on $C_d$. More precisely, for any $0\leq n\leq d$, we study the (convex) cone $\Eff_n(C_d)$ generated by \emph{effective} cycles inside the finite-dimensional $\bbR$-vector space $N_n(C_d)=N^{d-n}(C_d)$ of $n$-dimensional (or $(d-n)$-codimensional) cycles up to numerical equivalence, its closure $\Pseff_n(C_d)=\ov{\Eff_n(C_d)}\subset N_n(C_d)$, which is called the \emph{pseudoeffective} cone or cone of pseudoeffective (or simply pseff) cycles of dimension $n$, and its dual $\Nef^n(C_d)=\Pseff_n(C)^\vee=\Eff_n(C_d)^{\vee}\subset N^{n}(C_d)$, which is called the \emph{nef} cone or cone of nef cycles of codimension $n$.

While the structure of the full numerical ring of cycles $N^*(C_d)=\oplus_m N^m(C_d)$ 
is difficult to describe and it depends on the given curve $C$, there is a subring which is easy to describe
and on which we will focus most of our attention. In order to define this subring, 
consider the following two  divisor classes, well-defined up to numerical equivalence,
$$x = [\{D \in C_d : D = p_0+D', D' \in C_{d-1}\}]  \text{ and }  \theta=\alpha_d^*([\Theta])$$
where $p_0$ is a fixed point of $C$ and $\Theta$ is any theta divisor on $\Pic^d(C)$. It is well-known that $x$ is ample and $\theta$ is nef. 
These two classes generate a graded subring $R^*(C_d)=\oplus_m R^m(C_d)$ of 
$N^*(C_d)$, which 
we call \color{black}the \emph{tautological} ring of cycles, whose structure is well-understood (and independent of the given curve $C$), and which coincides with the full ring $N^*(C_d)$ if $C$ is a very general curve; see \S \ref{SS:tautring}. The cone $\Efft_n(C_d)
\subset R_n(C_d)\color{black}$ generated by tautological effective cycles of dimension $n$ is called the \emph{tautological effective} cone of dimension $n$, its closure $\Psefft_n(C_d)=\ov{\Efft_n(C_d)}\subset R_n(C_d)$ is called the \emph{tautological pseudoeffective} cone or cone of tautological pseudoeffective cycles of dimension $n$.

\vspace{0.1cm}

The problem of studying $\Psefft_n(C_d)$ and 
its dual for large values of $d$ was raised explicitly in \cite[\S 6]{DELV}. In this breakthrough paper, the authors have answered negatively an old question of Grothendieck by finding examples of smooth projective varieties (indeed special abelian varieties) possessing nef cycles that are not pseudoeffective \cite[Cor. 2.2, Prop. 4.4]{DELV} and nef cycles of complementary dimensions whose intersection is negative \cite[Cor. 4.6]{DELV}. Section 6 of loc. cit. contains lots of theoretical questions and quests for computations of nef and pseff cycles on special varieties (among which 
are the symmetric \color{black} products of curves). This has generated a fairly big amount of recent  literature on cycles of intermediate dimension, both on theoretical aspects (\cite{fl2}, \cite{leh1}, \cite{fl4}, \cite{fl1}, \cite{Ott1}, \cite{Ott3}, \cite{lx1}, \cite{lx2}) and on explicit computations (\cite{Ful}, \cite{CC}, \cite{Li}, \cite{Ott2}, \cite{sch}, \cite{CLO}). 

\subsection{Previous results: the case of divisors and curves}\label{S:oldresults}



The situation is well-understood for the cone of tautological effective divisors and the cone of tautological effective curves, at least if $d$ is large with respect to the genus $g$. Assuming from now on that $g\geq 1$ and $d\geq 2$ (in order to avoid the easy cases $C_1=C$ and $(\PP^1)_d=\PP^d$), 
the spaces $R^1(C_d)$ and $R_1(C_d)$ have dimension two, hence $\Psefft^1(C_d)$ and $\Psefft_1(C_d)$ (and their duals) 
have two extremal rays. Let us summarize what is known about the two extremal rays in each of the two cases. 

One extremal ray of $\Psefft^1(C_d)$ (which is also an extremal ray of $\Eff^1(C_d)$) is generated  for any $d\geq 2$ by the class $[\Delta]=2(d-1)![(d+g-1)x-\theta]$ of the big diagonal
$$\Delta=\{D\in C_d\: : \: D \text{ is non reduced} \},$$  
as proved by the second author in \cite[Thm. 3]{Kou}. 

The other extremal ray of $\Psefft^1(C_d)$ is generated by $\theta$ if and only if $d\geq g+1$: indeed, $\theta$ is always pseudoeffective (being nef) and it is not in the interior of the pseudoeffective cone, i.e. it is not big, if and only if $\alpha_d$ is not birational into its image, which happens exactly when $d \geq g+1$. 
By duality, if $d\geq g+1$ the tautological nef cone of curves has the following two extremal rays: 
$\bbR_{\geq 0} (dgx^{d-1}-(d-1)\theta x^{d-2})=(\bbR_{\geq 0} [\Delta])^\perp$ 
and 
$\bbR_{\geq 0} (-(g-1)x^{d-1}-x^{d-2}\theta)=(\bbR_{\geq 0} \theta)^{\perp}$.
For smaller values of $d$, we can define the pseudoeffective slope as 
$$\mu_d(C) = \max\left\{m\geq 0\left|\theta-mx\in \Psefft^1(C_d)\right.\right\}$$
so that  
$\bbR_{\geq 0} (\theta-\mu_{d}(C)x)$
is the extremal ray of $\Psefft^1(C_d)$ different from the one generated by the class of the big diagonal.

The previous result can be rephrased by saying that $\mu_{d}(C)>0$ if and only if $d\leq g$. Using the pull-back map $i_{p}^*:R^1(C_{d+1})\to R^1(C_d)$ with respect to any inclusion $i_{p}:C_d\to C_{d+1}$ obtained by sending $D$ into $D+p$ for a fixed point  $p\in C$, one easily gets 
that $\mu_{d}(C)$ is non-increasing in $d$. 
The  pseudoeffective slope $\mu_{d}(C)$ in the range $2\leq d\leq g$ has been subject to an extensive investigation; we refer to Remark \ref{R:oldresults} and to the references therein for detailed results.

A similar picture does hold for $\Psefft_1(C_d)$. One extremal ray of $\Psefft_1(C_d)$ (which is also an extremal ray of $\Eff_1(C_d)$) is generated  for any $d\geq 2$ by the class $[\delta]=d((1+g(d-1))x^{d-1}-(d-1)x^{d-2}\theta)$ of the small diagonal
$$\delta=\{d p\: : \: p\in C\},$$
as proved by Pacienza in \cite[Lemma 2.2]{Pac}. By duality, one of the extremal rays of 
$\Psefft_n(C)^\vee$ is 
$\bbR_{\geq 0} (dgx-\theta)=(\bbR_{\geq 0} [\delta])^\perp$. The other extremal ray 
is generated by $\theta$ if and only if $d\geq \gon(C)$ where $\gon(C)$ is the gonality of $C$: indeed, $\theta$ is always nef and it is not 
ample, if and only if $\alpha_d$ is not a finite morphism, which happens exactly when $d \geq \gon(C)$. Applying again the duality, we find that the other extremal ray of $\Psefft_1(C_d)$ is equal to 
$\bbR_{\geq 0} (-(g-1)x^{d-1}+\theta x^{d-2})=(\bbR_{\geq 0} \theta)^{\perp}$. For smaller values of $d$, we can define the nef slope $\nu_{d}(C)$ as  
$$\nu_d(C) = \max\left\{m\geq 0\left|\theta-mx\in \Psefft_n(C)^\vee\right.\right\},$$ 
so that  
$\bbR_{\geq 0} (\theta-\nu_{d}(C)x)$ is the extremal ray of 
$\Psefft_n(C)^\vee$ different from the one generated by  $dgx-\theta$.
Note that $\nu_d(C)\leq \mu_d(C)$ since 
$\Psefft_n(C)^\vee \subseteq \Psefft^1(C_d)$.
 The previous result can be rephrased by saying that $\nu_{d}(C)>0$ if and only if $d<\gon(C)$. Using the push-forward map $(i_{p})_*:R_1(C_{d+1})\to R_1(C_d)$ as before, one easily gets that $\nu_{d+1}(C)\leq \nu_{d}(C)$, i.e. that $\nu_{d}(C)$ is non-increasing in $d$. Observe that if $d=2$, then the pseudoeffective slope determines the nef slope (and vice versa) via the formula $\nu_2(C)=g\frac{\mu_2(C)-g+1}{\mu_2(C)-g}$. On the other hand, if $d>2$ then not much is known on  the nef slope $\nu_d(C)$ except for the following nice result:
if $C$ is a general curve of even genus $g$ then $\nu_{g/2}(C)=2$ by \cite[Thm. 1.1]{Pac}.




\subsection{Our results}

The aim of this paper and its sequel \cite{BKLV} is to extend 
some of \color{black}the above mentioned results to cycles of intermediate dimension.  


\vspace{0.1cm}

In the present paper, we generalize the extremality properties of the big and small diagonals to the intermediate diagonals, which are defined as it follows.
For any $1\leq n \leq d$, let $C^n$ be the
$n$-fold ordinary product
and, for any partition $\un a=(a_1,\ldots, a_n)\in \bbN^{n}$ of $d$ with $a_1\geq a_2\geq \ldots\geq a_n\geq 1$ and $\sum_{i=1}^n a_i=d$, consider 
the diagonal map
\begin{equation}\label{E:diag}
\begin{aligned}
\phi_{\un a}: C^n & \longrightarrow C_d \\
(p_1,\ldots, p_n) & \mapsto \sum_{i=1}^n a_i p_i.
\end{aligned}
\end{equation}
 

The cycle theoretic image of $\phi_{\un a}$ will be denoted by $\Delta_{\un a}$ and its class $[\Delta_{\un a}]$ belongs to the tautological ring $R^*(C_d)$; see Proposition \ref{diagclass} for an explicit formula which simplifies the classical formula in \cite[Chap. VIII, Prop. 5.1]{GAC1}. 

Our first result concerns the structure of the cone $\D_n(C_d)\subset \Efft_n(C_d)$ generated by the $n$-dimensional diagonals, which is called the \emph{$n$-dimensional diagonal cone}. To this aim, we introduce the balanced diagonals, 
defined as follows. For any $1\leq j \leq s(n):=\min\{n,d-n\}$, we denote by $\un \lambda^j$ the most balanced partition of $d-n$ in exactly  $j$ parts (see \eqref{balpart} for a more precise definition)  and we consider it as an element of $\bbN^n$ by adding zeros at the end. 
The element $\un \lambda^j=(\un \lambda^j_1,\cdots, \un \lambda^j_n)\in \bbN^n$ gives rise to a partition of $d$ in exactly $n$ parts $\un{\lambda^j+1}:=(\un \lambda^j_1+1,\cdots, \un \lambda^j_n+1)$, whose associated $n$-dimensional diagonal $\Delta^{\rm bal}_{j, n}:=\Delta_{\un{\lambda^j+1}}$ is called the \emph{$j$-th balanced diagonal}.


\begin{theoremalpha}\label{T:A}(= Theorem \ref{T:diagcone}) 
For $1\leq n\leq d-1$ and $g\geq 1$, the $n$-dimensional diagonal cone $\D_n(C_d)$  is a  rational polyhedral cone of dimension $r(n)
:=\min\{g,n,d-n\}\color{black}$ (hence it has codimension one in $\Efft_n(C_d)$ and in $\Psefft_n(C_d)$), whose extremal rays are generated by: 

\begin{itemize}
\item $\Delta^{\rm bal}_{1,n}, \ldots,  \Delta^{\rm bal}_{s(n), n}$ if $r(n)\geq 4$;
\item $\Delta^{\rm bal}_{1,n}, \{ \Delta^{\rm bal}_{j,n}\: :\: j\:  \text{is a $(d-n)$-break} \}, \Delta^{\rm bal}_{s(n), n}$ if $r(n)=3$, where a number $2\leq j \leq s(n)-1$ is called a 
$(d-n)$\emph{-break} \color{black}if 
$\left\lceil \frac{d-n}{j-1}\right\rceil >\left\lfloor \frac{d-n}{j+1}\right\rfloor+1$;
\item $ \Delta^{\rm bal}_{1,n}$ and $\Delta^{\rm bal}_{s(n), n}$ if $r(n)=2$.
\end{itemize}

In particular,  if $s(n)=r(n)$ (which happens if and only if $n\leq g$ or $d-n\leq g$) then $\D_n(C_d)$ is simplicial and each balanced diagonal generates an extremal ray of $\D_n(C_d)$. 
\end{theoremalpha}

Our second main result 
is the following theorem, which says that $\D_n(C_d)$ is a face of the cone of (tautological or not) pseudoeffective cycles and it collects the convex-geometric properties of this face (see the Appendix \ref{B} for the notion of perfect 
face \color{black}and locally finitely generated 
cone\color{black}).

\begin{theoremalpha}\label{T:B}
For any integers $g, n, d$ such that $1\leq n \leq d-1$ and $g\geq 1$, we have: 
\begin{enumerate}
\item \label{T:B1}
The $n$-dimensional diagonal cone $\D_n(C_d)$ is a rational polyhedral perfect face of $\Pseff_n(C_d)$ (and hence also of $\Eff_n(C_d)$), the faces of which are still perfect in 
$\Pseff_n(C_d)$.
\item \label{T:B2} 
$\Pseff_n(C_d)$ is locally finitely generated at every non-zero element of $\D_n(C_d)$.
\end{enumerate}
The same conclusions hold for $\Psefft_n(C_d)$ and $\Efft_n(C_d)$, and moreover $\D_n(C_d)$ is a facet (i.e. face of codimension one) of $\Psefft_n(C_d)$. 
\end{theoremalpha}


Remark that the above theorem holds true for any curve $C$, regardless of the dimension of $\Pseff_n(C_d)$, which instead varies according  to the Hodge theoretic speciality of the curve $C$ (namely the dimension of the numerical space of cycles of its Jacobian, see Theorem \ref{T:tautring}). 

Theorem \ref{T:B} follows from Corollary \ref{diag-poly}, whose proof is based on Theorem \ref{T:eta}. In this theorem, we consider 
the tautological class 
$$\eta_{n,d} =\frac{dg}{n} x^n-\theta x^{n-1}\in R^n(C_d)$$ 
and we prove that 
\begin{itemize}
\item  if $S$ is a $n$-dimensional diagonal in $C_d$, then $[S]\cdot \eta_{n,d}=0$;
\item  there exists a constant $c_{n,d}\geq 0$, which is positive for $g\geq 2$, such that for any  $n$-dimensional irreducible subvariety $S$ of $C_d$ which is  not a diagonal we have that 
\begin{equation}\label{eq:eta'}
(\eta_{n,d}-c_{n,d}x^n) \cdot [S] \geq 0.
\end{equation}
\end{itemize}
Using the above two facts, the proof of Theorem \ref{T:B} follows if $g\geq 2$ (the case of $g=1$ being easy since all the cones are two dimensional, see  Example \ref{genus1}) from a general result in convex geometry that we prove in Proposition \ref{edges1} of the Appendix \ref{B}. 
The geometrical picture corresponding to Theorem \ref{T:A} and Theorem \ref{T:B} is represented in Figure \ref{uno}.

\begin{figure}
\begin{tikzpicture}[scale=1]                                    

\begin{scope}
\clip (3,1) -- (6,4) -- (8.25,4.25) .. controls (10.5,4.5) and (11.8,3.5) .. (12,3) -- (9,0.5) .. controls (7.5,0) and (6,0) .. (4.5,0.5) -- cycle;
\fill[black!10!white] (0,-0.5) rectangle (15.5,5);
\end{scope}

\begin{scope}
\clip (1.5,-0.5) -- (7,5) -- (9,5) -- (3.5,-0.5) -- cycle;
\fill[black!25!white] (3,1) -- (6,4) -- (8.25,4.25) .. controls (10.5,4.5) and (11.8,3.5) .. (12,3) -- (9,0.7) -- (9,0.5) .. controls (7.5,0) and (6,0) .. (4.5,0.5) -- cycle;
\end{scope}

\draw[very thin] (0,-0.5) -- (10,-0.5);	
\draw[very thin] (5.5,5) -- (15.5,5);	
\draw[very thin] (0,-0.5) -- (5.5,5);	
\draw[very thin] (10,-0.5) -- (15.5,5);	

\draw[very thin] (1.5,-0.5) -- (7,5) node[label=80:$\eta_{n,d}^{\perp}$] {} ;	
\draw[very thin] (3.5,-0.5) -- (9,5) node[label=80:$(\eta_{n,d}-c_{n,d}x^n)^{\perp}$] {} ;	

\draw[very thick] (3,1) node [fill, color=black, shape=circle, scale=0.4] (BD1) {} -- node[midway,sloped,above] {diagonal cone} (6,4) node [fill, color=black, shape=circle, scale=0.4] (BD3) {}  ;	
\path (3.5,1.5) node [shape=circle,scale=0.4,draw,fill=white] {} ;
\path (4.5,2.5) node [shape=circle,scale=0.4,draw,fill=white] {} ;
\path (5.25,3.25) node [shape=circle,scale=0.4,draw,fill=white] {} ;

\draw[very thick] (9,0.5) -- node[midway,sloped,below] {Abel-Jacobi faces} (12,3); 

\draw[very thick, densely dotted] (3,1) -- (4.5,0.5);
\draw[very thick, densely dotted] (6,4) -- (8.25,4.25);
\draw[very thick, densely dotted]  (8.25,4.25) .. controls (10.5,4.5) and (11.8,3.5) .. (12,3);
\draw[very thick, densely dotted]  (4.5,0.5) .. controls (6,0) and (7.5,0) .. (9,0.5);

\path (6,-5) coordinate (V) ;
\draw[dashed] (6,4) -- (V);
\draw[dashed] (8.25,4.25) -- (V);
\begin{scope}
\clip (0,-0.5) -- (13,-0.5) -- (13,5) -- (0,5) -- cycle;
\draw[dashed] (3,1) -- (V);
\draw[dashed] (9,0.5) -- (V);
\draw[dashed] (12,3) -- (V);
\draw[dashed] (4.5,0.5) -- (V);
\end{scope}
\begin{scope}
\clip (0,-0.5) -- (10,-0.5) -- (10,-5) -- (0,-5) -- cycle;
\draw (3,1) -- (V);
\draw (9,0.5) -- (V);
\draw (12,3) -- (V);
\draw (4.5,0.5) -- (V);
\end{scope}

\draw (7,2) node[label=$\mathrm{Pseff}_n(C_d)$] {};

\end{tikzpicture}
\caption{The colored area in the picture represents a hyperplane section of the cone $\mathrm{Pseff}_n(C_d)$.  
The dark grey area lying between the hyperplanes $\eta_{n,d}^{\perp}$ and $(\eta_{n,d}-c_{n,d}x^n)^{\perp}$ does not contain irreducible cycles, according to \eqref{eq:eta'}.
The dotted lines denote the boundaries we cannot describe.
The bold line on the left represents the diagonal cone. In particular, the black dots denote rays spanned by balanced diagonals, which are all extremal as soon as $r(n)\geq 4$. 
The white dots represent rays spanned by ``non-balanced'' diagonals, which are never extremal.
The bold line on the right denotes the Abel-Jacobi faces, 
described in \cite{BKLV}.}
\label{uno}
\end{figure}
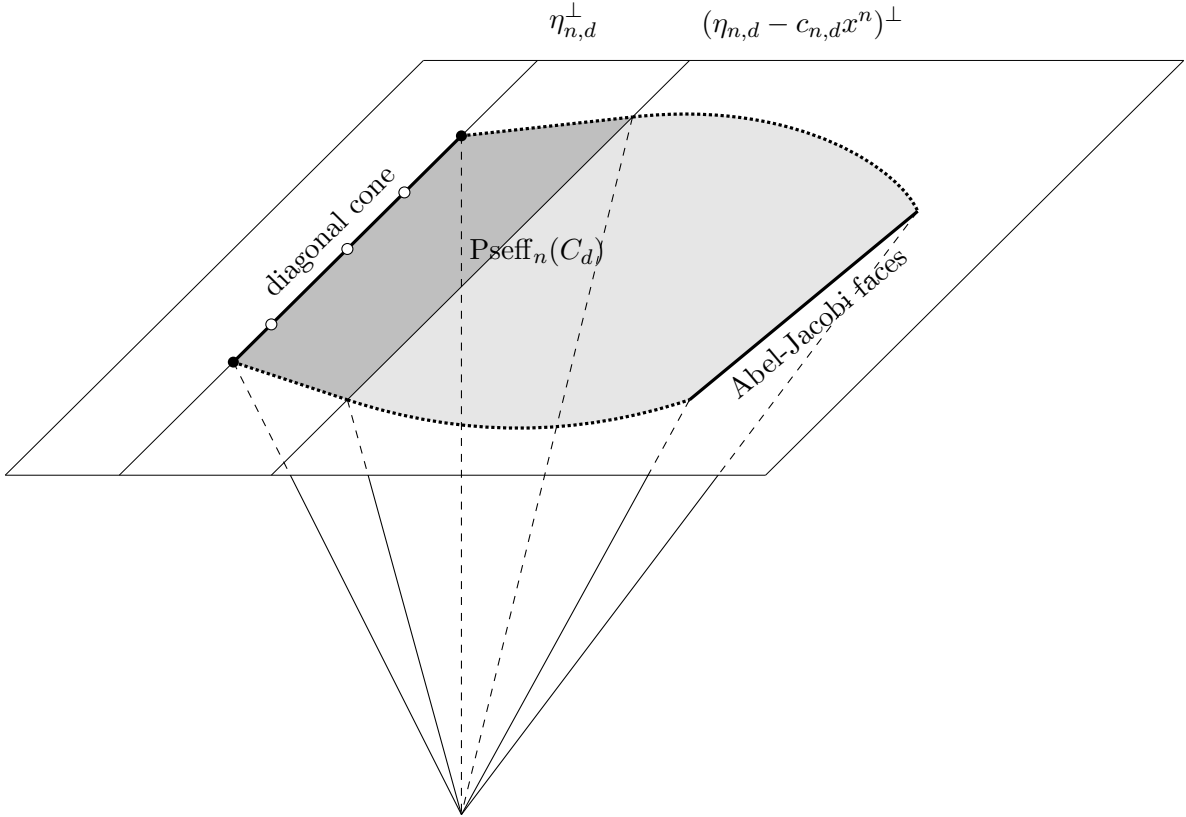

Another interesting corollary of Theorem \ref{T:eta} is the following \emph{numerical rigidity} of the diagonal cone (see Corollary \ref{diag-rig}): if $S$ and $T$ are two effective $n$-cycles such that $[S]=[T]\in N_n(C_d)$ and $S$ is a positive linear combinations of diagonal cycles then 
so is $T$.

\vspace{0.1cm}

The present work leaves open some natural questions.


\begin{question}
If $4\leq r(n)<s(n)$ (which is equivalent to $4\leq g<\min\{n, d-n\}$), what is the face lattice of the $n$-dimensional diagonal cone $\D_n(C_d)$? 
\end{question}
If $d-n$ is a multiple of every integer $2\leq j\leq n$, then $\D_n(C_d)$ is a cone over the cyclic polytope (see Example \ref{exa:rbigs}). However this is not true in general: Lemma \ref{L:face-r3} implies that for $r(n)\geq 4$ a bounded section of $\D_n(C_d)$ can be a non-simplicial polytope (in particular, it is not the cyclic polytope).


\begin{question}
Is $\Psefft_n(C_d)$ or $\Pseff_n(C_d)$ locally polyhedral at any non-zero point of $\D_n(C_d)$? 
\end{question}
Note that a necessary condition for $\Pseff_n(C_d)$ to be locally polyhedral at any non-zero point of $\D_n(C_d)$ is to be locally finitely generated in such points and that $\D_n(C_d)$ and all its faces are perfect faces of $\Pseff_n(C_d)$ (and similarly for $\Psefft_n(C_d))$, and both these necessary conditions are satisfied according to Theorem \ref{T:B}.  





\section{Preliminaries}

\subsection{Notations and conventions}

Unless otherwise specified, we work throughout the paper over an algebraically closed field $k$ 
of arbitrary characteristic.


For any natural number $n\in \bbN$ and any real number $r\in \R$, we set 
$$\binom{r}{n}=
\begin{cases}
\frac{r(r-1)\ldots (r-n+1)}{n!} & \text{if }n>0, \\
1 & \text{if } n=0.
\end{cases}
$$

\subsection{Symmetric product}

Let $C$ be a smooth projective irreducible curve of genus $g\geq 1$. For any integer $d\geq 1$, we denote by $C^d$ the 
$d$-fold ordinary product of $C$ and by $C_d$ the $d$-fold symmetric product of $C$, which parametrizes effective divisors on $C$ of degree $d$.  

The symmetric product $C_d$ is related to the Jacobian of $C$ by the 
Abel-Jacobi morphism
$\alpha_d : C_d \to \Pic^d(C)$ sending $D\in C_d$ into $\O_C(D)\in \Pic^d(C)$. 
If $L\in \Pic^d(C)$ then $\alpha_d^{-1}(L) = |L|$. 

Given $p_0\in C$, there is an inclusion $i=i_{p_0}: C_{d-1} \hookrightarrow C_d$, obtained by sending $D$ into $D+p_0$. We 
set $X_{p_0} = i_{p_0}(C_{d-1})$. 
Note that $i_{p_0}$ is compatible with 
$\alpha_d$ in the sense that  $\alpha_d\circ i_{p_0}= t_{p_0}\circ \alpha_{d-1}$, where $t_{p_0}:\Pic^{d-1}(C)\to \Pic^d(C)$ is the translation by $p_0$ which 
sends $L$ into $L(p_0)$.  

If $d\geq 2g-1$ then $\alpha_d$ is a projective bundle which we now describe. Consider the 
diagram 
$$\xymatrix{
\Delta \ar@{^{(}->}[r]& C\times C_d \ar[r]^{\hskip -.5cm \id \times \alpha_d} \ar[d]_{\pi} & C\times \Pic^d(C) \ar[r]^{\hskip .3cm \nu} & \Pic^d(C) \\
& C_d}
$$
where 
$\pi$ and $\nu$ are 
projections and $\Delta$ is the universal divisor of $C_d$. The line bundle $\L_{p_0}:=(\id \times \alpha_d)_*\O(\Delta-\pi^*X_{p_0})$ on $C\times \Pic^d(C)$
is a \emph{Poincar\'e} line bundle normalized along $p_0$, i.e. it is the unique line bundle on $C\times \Pic^d(C)$ such that   $(\L_{p_0})_{|C\times \{L\}}\cong L$ and $(\L_{p_0})_{|\{p_0\}\times \Pic^d(C)}\cong \O_{\Pic^d(C)}$.
If $d\geq 2g-1$ then 
$E_d:=\nu_*(\L_{p_0})$ is a locally free sheaf of rank $d+1-g$ (
by Riemann vanishing), which is called the \emph{Picard} bundle of degree $d$ (with respect to 
$p_0$). By construction and Riemann vanishing, we have that $(E_d)_{|L}=H^0(C, L)$ for any $L\in \Pic^d(C)$.  


\begin{fact}\label{projbund} Assume that $d\geq 2g-1$. 
\begin{enumerate}[(i)]
\item \label{projbund1} There is an isomorphism $C_d\cong \PP(E_d^*):=\Proj \left(\oplus_n \Sym^n E_d^*\right)$ over $\Pic^d(C)$ under which the line bundle $\O_{C_d}(X_{p_0})$ corresponds to $\O_{\PP(E_d^*)}(1)$.
\item \label{projbund3}The vector bundle $E_d^*$ is ample, i.e. $\O_{\PP(E_d^*)}(1)$ is ample on $\PP(E_d^*)$.
\item \label{projbund4} The vector bundle $E_d^*$ is stable with respect to the polarization $[\Theta]$ on $\Pic^d(C)$.
\end{enumerate} 
\end{fact}
\begin{proof}
For a proof of \eqref{projbund1}, we refer the reader to 
\cite{Mat1} or to \cite[Chap. VII, Prop. (2.1)]{GAC1} 
(but be aware that we use Grothendieck's notation for projective bundles, contrary to the notation used in \cite{GAC1}).
For 
\eqref{projbund3} see  \cite[Chap. VII, Prop. (2.2)]{GAC1} (note that the result is stated in loc. cit. for $k=\bbC$ but the proof 
works over any 
algebraically closed field $k$). 
Part \eqref{projbund4} is due to Kempf \cite{Kem1} for $d=2g-1$ and to Ein-Lazarsfeld \cite{EL} for $d\geq 2g$. 
\end{proof}

\subsection{Tautological ring}\label{SS:tautring}

For any $0 \le n, m \le d$, we will denote by $N_n(C_d)$ (resp. $N^m(C_d)$) the $\R$-vector space of $n$-dimensional (resp. $m$-codimensional) cycles on $C_d$ modulo numerical equivalence. The intersection product 
induces a perfect duality $N^m(C_d)\times N_{d-m}(C_d)\stackrel{\cdot}{\longrightarrow} \R$. The vector space $N^*(C_d)=\oplus_{m=0}^d N^m(C_d)$ is a graded $\R$-algebra with respect to the intersection product. 

The \emph{tautological ring} $R^*(C_d)$ is the graded $\R$-subalgebra of $N^*(C_d)$ generated by the codimension one classes $\theta=\alpha_d^*([\Theta])$ (where $\Theta$ is any theta divisor on $J(C)$) and $x=[X_{p_0}]$ 
for some (equivalently any) 
point $p_0$. 
Observe that $\theta$ is a semiample class (because it is the pull-back of an ample line bundle via a regular morphism) and it is ample if and only if $\alpha_d$ is a finite morphism if and only if $d < \gon(C)$. 
The class $x$ is ample 
by \cite[Chap. VII, Prop. (2.2)]{GAC1} (note that the result is stated in loc. cit. for $k=\bbC$ but the proof, based on the Nakai-Moishezon criterion for ampleness, works over any algebraically closed field $k$).

The intersections among powers of $x$ and $\theta$ is governed by the following formulas.


\begin{lemma}\label{L:inters}
We have that $\theta^{g+1}=0$ and, for any integer $0\leq s\leq  d$, it holds that 
$$\theta^s \cdot x^{d-s}=
\begin{cases}
g(g-1) \cdots (g-s+1) & \text{if } s>0, \\
1 & \text{if } s=0.
\end{cases}
$$
\end{lemma}
\begin{proof}
The fact that $\theta^{g+1}=0$ follows, via  pull-back along $\alpha_d$, from the fact that $[\Theta]^{g+1}=0$ because $\Pic^d(C)$ has dimension $g$. 
In order to prove the second relation if $0\leq s\leq \min\{{d, g\}}$, observe that $x^{d-s}$ is the class of the subvariety $C_{s}$, embedded in $C_d$ by sending $D$ into $D+(d-s)p_0$. The pushforward of $[C_{s}]$ via $\alpha_d$ is the 
class of the subvariety 
$$\widetilde{W}_s:=\{L\in \Pic^d(C): \: L(-(d-s) p_0) \: \text{ is effective}\},$$  
which is equal to $\displaystyle \frac{[\Theta^{g-s}]}{(g-s)!}$ by the Poincar\'e formula (see \cite[Appendix]{Mats} or \cite[Chap. I, \S 5]{GAC1} if $k=\bbC$). 
Therefore, if $0\leq s\leq \min\{{d, g\}}$, using the projection formula, we compute
$$\theta^s \cdot x^{d-s}= \alpha_d^*([\Theta]^s) \cdot x^{d-s}=[\Theta]^s \cdot (\alpha_d)_*(x^{d-s})=[\Theta]^s\cdot \frac{[\Theta^{g-s}]}{(g-s)!}=\frac{[\Theta]^g}{(g-s)!}=\frac{g!}{(g-s)!},
$$
where we have used that $[\Theta]^g=g!$ again by the Poincar\'e formula. 
\end{proof}


We now describe the structure of  the tautological ring. We begin by giving an explicit presentation of $N^*(C_d)$ in terms of $N^*(J(C))$, which is 
analogous 
to the presentation of the Chow ring $\CH^*(C_d)$ 
in terms of 
$\CH^*(J(C))$ 
given by Collino \cite{Col}.



\begin{thm}\label{T:tautring}
The  ring $N^*(C_d)$ admits the following presentation 
$$0\to I_d\to N^*(J(C))[T]\xrightarrow[]{\phi_d} N^*(C_d)\to 0, $$
where $\phi_d$ is the ring homomorphism 
coinciding on $N^*(J(C))$ with 
$\alpha_d^*:N^*(J(C))\to N^*(C_d)$ 
and such that $\phi_d(T)=x$, and the ideal $I_d$ is defined by 
\begin{equation}\label{E:idealI}
I_d:=
\begin{cases}
((\beta)\colon T^{2g-1-d}) & \text{if } d\leq 2g-2, \\ 
(\beta T^{d-2g+1}) & \text{if } d\geq 2g-1,
\end{cases}
\end{equation}

where 
$\beta:=\sum_{i=0}^g (-1)^i \frac{[\Theta]^i}{i!} T^{g-i}\in  N^*(J(C))[T]$.
\end{thm}
\begin{proof}
From the presentation of $\CH^*(C_d)$ in terms of $\CH^*(J(C))$ given in \cite[Thm. 3]{Col} and using that $N^*(C_d)$ (resp. $N^*(J(C))$) is a quotient of  $\CH^*(C_d)\otimes {\bbR}$ (resp. of $\CH^*(J(C))\otimes {\bbR}$), it follows that the morphism $\phi_d$ is surjective.  

It remains to show that the kernel of $\phi_d$ is $I_d$. For $d\geq 2g-1$, this follows by applying the projective bundle formula to $\alpha_d:C_d \cong \PP(E_d^*) \to J(C)$ (see Fact \ref{projbund}\eqref{projbund1}), using that, for any $i\geq 0$, $\displaystyle c_i(E_d^*)=[\Theta]^i / i! \in N^*(J(C))$. The latter follows by combining \cite[\S 6]{Mat2} and the Poincar\'e formula (see the proof of Lemma \ref{L:inters}). 
In order to prove that $\Ker \phi_d=I_d$ for $d\leq 2g-2$, 
consider the inclusion $i_{m,n}:C_m\hookrightarrow C_n$ (for any $m<n$) which 
sends $D\in C_m$ into $D+(n-m)p_0\in C_n$.  


\un{Claim 1:} $i_{m,n}^*:N^*(C_n)\to N^*(C_m)$ is surjective.

This follows from the fact that 
$i_{m,n}^*:\CH^*(C_n)\to \CH^*(C_m)$ 
is surjective on Chow rings by \cite[Thm. 2]{Col} 
and the fact that $N^*(C_n)$ is a quotient of  $\CH^*(C_n)\otimes {\bbR}$ (and similarly for $C_m$). 

\un{Claim 2:} $(i_{m,n})_*:N^*(C_m)\to N^*(C_n)$ is injective. 

Indeed, suppose that $S$ is a cycle on $C_m$ of dimension $k$ such that $(i_{m,n})_*([S])=0\in N_k(C_n)$. Using this and the projection formula, we get that 
$$[S]\cdot i_{m,n}^*([T]) = (i_{m,n})_*([S])\cdot [T]=0 \text{ for any  } [T]\in N^k(C_n). $$
Since $i_{m,n}^*$ is surjective by Claim 1, we deduce that $[S]\cdot [W]=0$ for any $[W]\in N^k(C_m)$, which implies that $[S]=0$ since the intersection product is non-degenerate on $N^*(C_m)$ by definition of 
the ring $N^*(C_m)$. 

We are now ready to prove the equality $\Ker \phi_d=I_d$ for $d\leq 2g-2$. Indeed, we will deduce such an equality from the  equality $\Ker \phi_{2g-1}=(\beta)$ using Claim 2. 
First of all, observe that 
\begin{equation}\label{E:2phi}
\phi_d=i_{d,2g-1}^*\circ \phi_{2g-1},
\end{equation}
 since $\alpha_d^*=i_{d,2g-1}^*\circ \alpha_{2g-1}^*$ (which follows from the equality of morphisms $\alpha_d=\alpha_{2g-1}\circ i_{d, 2g-1}$) 
and $i_{d,2g-1}^*$ sends, with a little abuse of notation, the class $x=\phi_{2g-1}(T)\in N^*(C_{2g-1})$ into $x=\phi_d(T)\in N^*(C_d)$. 
Now if $P\in N^*(J(C))[T]$ we have 
\begin{align*}
& P\in \Ker \phi_d  \Longleftrightarrow  0=\phi_d(P)=i_{d,2g-1}^*(\phi_{2g-1}(P)) & \quad \text{ by } \eqref{E:2phi}, \\
& \Longleftrightarrow 0=(i_{d,2g-1})_*(i_{d,2g-1}^*(\phi_{2g-1}(P))) &  \quad \text{ by Claim 2,} \\
& \Longleftrightarrow 0=x^{2g-1-d}\phi_{2g-1}(P)=\phi_{2g-1}(T^{2g-1-d})\phi_{2g-1}(P) & \quad \text{ by the projection formula and  } \\
& & (i_{d,2g-1})_*([1])=x^{2g-1-d}, \\ 
& \Longleftrightarrow T^{2g-1-d}P\in \Ker \phi_{2g-1}=(\beta)  \Longleftrightarrow P \in I_d.  \\ 
\end{align*}
This shows that $\Ker \phi_d=I_d$ for $d\leq 2g-2$. 
\end{proof}

Using the previous result, we can now give a standard basis of $R^m(C_d)$ for every $m$, and show that the tautological ring satisfies Poincar\'e duality.

\begin{prop}\label{basetaut}
\noindent 
\begin{enumerate}[(i)]
\item \label{basetaut1} For any $0\leq m\leq d$, set $r(m):=\min\{m,d-m, g\}$. Then
$R^m(C_d)$ has dimension $r(m)+1$ and it is freely generated by any subset of $r(m)+1$ monomials belonging to $\{x^m,x^{m-1}\theta, \ldots, x^{m-\min\{m,g\}} \theta^{\min\{m,g\}}\}$.

In particular, the monomials $\{x^m,\ldots, x^{m-r(m)}\theta^{r(m)}\}$ form a basis of $R^m(C_d)$, which is called the \emph{standard basis}. 
\item \label{basetaut2} The intersection product $R^m(C_d)\times R^{d-m}(C_d)\to \R$ is non-degenerate. 
\end{enumerate} 
\end{prop}
\begin{proof}
Set $r:=r(m)$. 
We will start by showing the following

\un{Claim:} Given two increasing sequences $i_\bullet=\{0\leq i_0<i_1<\ldots<i_r\leq \min\{m,g\}\}$ and $j_\bullet=\{0\leq j_0<j_1<\ldots<j_r\leq \min\{d-m,g\}\}$, the intersection matrix 
$$A(i_\bullet, j_\bullet)=\left(x^{m-i_k}\theta^{i_k}\cdot x^{d-m-j_h}\theta^{j_h}\right)_{0\leq k, h\leq r} \text{ is  invertible}.$$  
Indeed, according to Lemma \ref{L:inters}, we have that 
\begin{equation}\label{inters-mat}
A(i_\bullet,j_\bullet)=\left( g(g-1)\ldots (g-i_k-j_h+1)\right)_{0\leq k,h \leq r}.
\end{equation}
By definition of $r$, we 
have that either 
$r=\min\{m,g\}$ or $r=\min\{d-m,g\}$
(or both). 
In the former case we 
have that $i_\bullet=\{0=i_0<1=i_1<\ldots<i_r=r=\min\{m,g\}\}$ while in the latter case we 
have that $j_\bullet=\{0=j_0<1=j_1<\ldots<j_r=r=\min\{d-m,g\}\}$.
Assume that we are in the latter case (the proof being analogous in the former case). 
Then 
$A(i_\bullet, j_\bullet)$ becomes
$$A(i_\bullet, j_\bullet)=B(i_\bullet)=\left( g(g-1)\ldots (g-i_k-h+1)\right)_{0\leq k,h \leq r}
$$
To conclude the proof of the Claim, 
we show that 
$\det B(i_\bullet) \neq 0$.
Since the $k$-th row of $B(i_\bullet)$ is a multiple of  $g(g-1)\ldots (g-i_k+1)$, we get 
$$\det B(i_\bullet)  = \left(\prod_{0\leq k\leq r} g(g-1)\ldots (g-i_k+1) \right) \det C(i_\bullet),$$
where the matrix $C(i_\bullet)$ has the $(k,h)$-entry equal to  
$$C(i_\bullet)_{k,h} := 
\begin{cases} 
1 & \mbox{if} \ h = 0,  \\ 
(g-i_k)\ldots (g-i_k-h+1) & \mbox{if} \ h > 0.
 \end{cases}
$$
Moreover, since $i_k\leq \min\{m,g\}\leq g$ for every $0\leq k\leq r$, $\det B(i_\bullet)  \neq 0$ if and only if $\det C(i_\bullet)  \neq 0$.
A straightforward Vandermonde-type argument, shows that 
$$\det C(i_\bullet) =\prod_{0\leq k<l\leq r} \left[(g-i_l)-(g-i_k) \right]=\prod_{0\leq k<l\leq r} (i_k-i_l)\neq 0.$$ 
Hence $\det C(i_\bullet)  \neq 0$ and we are done.   

\vspace{0.1cm}


The Claim shows that $\dim R^m(C_d)\geq r(m)+1$ and that any  subset of $r(m)+1$ monomials belonging to $\{x^m,x^{m-1}\theta, \ldots, x^{m-\min\{m,g\}} \theta^{\min\{m,g\}} \}$ is linearly independent. In order to conclude the proof of part \eqref{basetaut1}, 
it remains to show that $\dim R^m(C_d)\leq r(m)+1$. 
%

With this aim, let $\R([\Theta])$ be the subring of $N^*(J(C))$ generated by $[\Theta]$. Then Theorem \ref{T:tautring} gives an exact sequence
$$0 \to I_d \cap \R([\Theta])[T] \to  \R([\Theta])[T] \to R^*(C_d) \to 0.$$
Since $\R([\Theta])[T] \cong \bbR[X,Y]/(X^{g+1})$, it follows that $R^*(C_d)$ has the following presentation
\begin{equation}\label{E:pres-R}
R^*(C_d)=\frac{\bbR[\theta,x]/(\theta^{g+1})}{J_d},
\end{equation}
where $J_d$ is the image of $I_d \cap \R([\Theta])[T]$ in $\bbR[\theta,x]/(\theta^{g+1})$. 
Moreover, in $\bbR[\theta,x]/(\theta^{g+1})$, we have:

\begin{equation}\label{E:ideal-K}
J_d\supseteq K_d:=
\begin{cases}
((\alpha)\colon x^{2g-1-d}) & \text{if } d\leq 2g-2, \\ 
(\alpha x^{d-2g+1}) & \text{if } d\geq 2g-1,
\end{cases}
\end{equation}
where $\alpha=\sum_{i=0}^g (-1)^i \frac{\theta^i}{i!} x^{g-i}$ is now considered in the 
$\bbR[\theta,x]/(\theta^{g+1})$. 

We will distinguish the two cases $d\geq 2g-1$ and $d\leq 2g-2$. 

If $d\geq 2g-1$ then  \eqref{E:pres-R} and  \eqref{E:ideal-K} imply that:
\begin{itemize}
\item if $m\leq d-g$ then $R^m(C_d)$ is generated by the elements $\{x^m, x^{m-1}\theta, \ldots, x^{m-\min\{m,g\}}\theta^{\min\{m,g\}} \}$, hence $\dim R^m(C_d)\leq \min\{m,g\}+1=r(m)+1$;
\item if $d-g<m$ (which implies that $g\leq m$) then $R^m(C_d)$ is generated by $\{x^{d-g}\theta^{m-d+g}, x^{d-g-1}\theta^{m-d+g+1}, \ldots, x^{m-g}\theta^{g} \}$, hence $\dim R^m(C_d)\leq d-m+1=r(m)+1$. 
\end{itemize}

If $d\leq 2g-2$ then we will distinguish three further cases.
\begin{itemize}
\item If $m\leq d-m$ (which implies that $m\leq g-1$) then $R^m(C_d)$ is  generated by $\{x^m,x^{m-1}\theta, \ldots, x\theta^{m-1},\theta^m\}$, hence $\dim R^m(C_d)\leq m+1=r(m)+1$.
\item If $d-m<g\leq m$  then  $R^m(C_d)$ is  generated by $\{x^m,x^{m-1}\theta, \ldots, x^{m-g+1}\theta^{g-1},x^{m-g}\theta^g\}$ and 
\eqref{E:pres-R}, \eqref{E:ideal-K} provide the following independent set of relations 
$$\frac{\theta^kx^{m-d+g-1-k}\cdot \alpha}{x^{2g-1-d}}=\sum_{i=0}^{g-k} \frac{(-1)^i}{i!}\theta^{i+k}x^{m-k-i}\in K_d\subseteq J_d \quad \text{ for } 0\leq k \leq m-d+g-1.
$$
Hence $\dim R^m(C_d)\leq g+1-(m-d+g)=d-m+1=r(m)+1$.
\item If $d-m<m<g$ then $R^m(C_d)$ is generated by $\{x^m,x^{m-1}\theta, \ldots, x\theta^{m-1},\theta^m\}$ and we will use 
\eqref{E:pres-R} and \eqref{E:ideal-K} to find $m-(d-m)$ independent relations, which will then show that $\dim R^m(C_d)\leq m+1-[m-(d-m)]=d-m+1=r(m)+1.$ 

With that in mind, consider a polynomial $Q(x,\theta)=\sum_{j=0}^{g-1-(d-m)} \beta_j \theta^jx^{g-1-(d-m)-j}\in \bbR[x,\theta]/(\theta^{g+1})$ of degree $g-1-(d-m)$ in $x$ and $\theta$. The polynomial
$$Q(x,\theta)\cdot \alpha=\sum_{\substack{0\leq i \leq g \\ 0\leq j \leq g-1-(d-m) \\ i+j\leq g}} \frac{(-1)^i \beta_j}{i!} \theta^{i+j}x^{2g-1-(d-m)-i-j}\in \bbR[x,\theta]/(\theta^{g+1})$$
is divisible  by $x^{2g-1-d}$ if and only if 
$$\sum_{\substack{0\leq i \leq g \\ 0\leq j \leq g-1-(d-m) \\ i+j=k}} \frac{(-1)^i \beta_j}{i!}= 0$$
that is if and only if 
\begin{equation}\label{E:eqcoeff}
\sum_{j=0}^{\min\{k,g-1-(d-m)\}} \frac{(-1)^{k-j} \beta_j}{(k-j)!} = 0 \ \text{ for any } \: m+1\leq k \leq g.
\end{equation}
It can be shown that the conditions \eqref{E:eqcoeff} on the coefficients $\beta_j$ of $Q(x,\theta)$ are linearly independent; 
hence the linear space $W$ of all the polynomials $Q(x,\theta)\in \bbR[x,\theta]/(\theta^{g+1})$ of degree $g-1-(d-m)$ in $x$ and $\theta$ whose coefficients $\beta_j$ 
satisfy the conditions \eqref{E:eqcoeff} has dimension $\dim W=g-(d-m)-(g-m)=2m-d=m-(d-m)$. 


The inclusion \eqref{E:ideal-K} implies that the ideal $J_d$ contains the subspace  
$$\left\{\frac{Q(x,\theta)\cdot \alpha}{x^{2g-1-d}} \: : Q(x,\theta) \in W\right\}\subseteq K_d,$$
which has dimension $m-(d-m)$. This concludes our proof. 
\end{itemize}

Finally, part \eqref{basetaut2} follows from part \eqref{basetaut1} and the above Claim. 
\end{proof}



\begin{remark}
It follows from the proof of 
Proposition \ref{basetaut} that the inclusion \eqref{E:ideal-K} is actually an equality. Hence, using \eqref{E:pres-R}, we get the following presentation of the tautological ring 
of $C_d$ (which is the analogue of the presentation of the full ring $N^*(C_d)$ given in Theorem \ref{T:tautring})
\begin{equation}\label{E:newpres-R}
R^*(C_d)=\frac{\bbR[\theta,x]/(\theta^{g+1})}{J_d}
\: \text{ with }  
J_d:=
\begin{cases}
((\alpha)\colon x^{2g-1-d}) & \text{if } d\leq 2g-2, \\ 
(\alpha\cdot x^{d-2g+1}) & \text{if } d\geq 2g-1,
\end{cases}
\end{equation}
where $\alpha=\sum_{i=0}^g (-1)^i \frac{\theta^i}{i!} x^{g-i}\in (\bbR[\theta,x]/(\theta^{g+1}))_g$ and the above ideal-theoretic operation are performed in $\bbR[\theta,x]/(\theta^{g+1})$. 

An explicit set of generators of the ideal $J_d$ in the case $d\leq 2g-2$ is given (over $\bbC$) in \cite[Chap. VII, Ex. B]{GAC1}.
\end{remark}


We end this subsection with the following result that, although it is never used in what follows, it helps in understanding the relevance of the tautological ring. 


\begin{fact}\label{verygen}
If $C$ is a very general  curve over an algebraically closed field $k$, then $R^*(C_d)=N^*(C_d)$ for every $d\geq 1$. 
\end{fact}
Note that the statement becomes empty if $k$ is countable. 

\begin{proof}
Theorem \ref{T:tautring} implies that  $N^*(C_d)$ is a quotient of $N^*(\Pic^d(C))[x]$ (for any curve $C$ over any field $k=\ov{k}$).  Therefore, the equality $R^*(C_d)=N^*(C_d)$ holds if and only if $N^*(J(C))$ is generated by  the class $[\Theta]$ of the theta divisor. 
This last property is known to be true for a very general  curve: a proof over the complex numbers is given in  \cite[Thm. 17.5.1]{BL} and a proof  over an arbitrary field (even non algebraically closed) follows by the Appendix \ref{A}.
\end{proof}

\subsection{Push-pull operators}\label{S:pushpull}

The aim of this subsection is to recall the definition of the push and pull operators (following \cite[Chap. VIII, Exercise D]{GAC1}), which allows one to relate the tautological ring of $C_d$ with the one of $C_{d+1}$.
\begin{defi}\label{D:pushpull}
For any $d\geq 1$, consider the diagram 
$$\xymatrix{
& C\times C_d \ar[dl]_{\pi_2}\ar[dr]^{\mu} & \\
C_d & & C_{d+1}
}
$$
where $\pi_2$ is the projection 
and $\mu$ is the addition map sending 
$(p,D)$ into $p+D$.  
\begin{enumerate}
\item The \emph{push operator} $A$ is the following linear map (for any $0\leq m \leq d$) 
$$\begin{aligned}
A: N^m(C_d) & \longrightarrow N^m(C_{d+1}), \\
z & \mapsto \mu_*\pi_2^*(z).
\end{aligned}$$
\item The \emph{pull operator} $B$ is the following linear map (for any $0\leq n \leq d+1$) 
$$\begin{aligned}
B: N_n(C_{d+1}) & \longrightarrow N_n(C_d), \\
w & \mapsto (\pi_{2})_*\mu^*(w).
\end{aligned}$$
\end{enumerate}
\end{defi} 
Note that the above definitions make sense since $\pi_2$ is proper and smooth of relative dimension one, and $\mu$ is finite and flat. 

\begin{remark}\label{R:eff-pp}
A more geometric description of the maps $A$ and $B$ at the level of effective cycles is given as follows. 
\noindent 
\begin{enumerate}[(i)]
\item  \label{R:eff-ppA} For any $Z\subseteq C_d$ integral subvariety of codimension $m$, $A([Z])$ is the class of an effective codimension $m$ cycle $A(Z)$ whose support is the integral codimension $m$ subvariety 
$$\mu(\pi_2^{-1}(Z))=\{D+p\: : D\in Z, p\in C \}=\{E\in C_{d+1}\: : E\geq D\geq 0 \: \text{ for some }D\in Z\}\subseteq C_{d+1},$$
and whose multiplicity is the cardinality of the finite set $\{D\in Z\: : D\leq E \}$ for a generic point $E\in \mu(\pi_2^{-1}(Z))$. 
\item \label{R:eff-ppB} For any $W\subseteq C_{d+1}$ integral subvariety of dimension $n$, $B([W])$ is the class of a dimension $n$ effective cycle $B(W)$ of $C_d$   whose support is the union of the  
irreducible components of dimension $n$ of the reduced subvariety 
$$\pi_2(\mu^{-1}(W))=\{E-p\geq 0 \: : E\in W, p\in C \}=\{D\in C_d\: : D\leq E \: \text{ for some }E\in W\}\subseteq C_d,$$
and such that the multiplicity of a dimension $n$ irreducible component $\pi_2(\mu^{-1}(W))_i$ of $\pi_2(\mu^{-1}(W))$ is the cardinality of the finite set $\{E\in W\: : E\geq D_i\}$ for a generic point $D_i$ of 
$\pi_2(\mu^{-1}(W))_i$.
\end{enumerate}
\end{remark}

The basic properties of the push and pull operators are recalled in the following 

\begin{fact}(\cite[Chap. VIII, Exercise D]{GAC1})\label{F:pushpull}
\noindent 
\begin{enumerate}[(i)]
\item \label{F:pushpull1} For any $z\in N^m(C_d)$ and any $w\in N_m(C_{d+1})$, we have that 
$$A(z)\cdot w=z\cdot B(w).$$
In other words, if we identify $N^m(C_d)^{\vee}$ with $N_m(C_d)$ via the intersection product (and similarly for $C_{d+1}$), then $A$ and $B$ are dual maps.
\item  \label{F:pushpull2} 
The operators $A$ and $B$ preserve the tautological rings and we have that 
$$
\begin{sis}
& A(x^{\alpha}\theta^{\beta})=(d+1-\alpha-2\beta) x^{\alpha} \theta^{\beta} +\beta(g-\beta+1)x^{\alpha+1}\theta^{\beta-1}, \\
& B( x^{\alpha}\theta^{\beta})= \alpha x^{\alpha-1} \theta^{\beta}+\beta(g-\beta+1) x^{\alpha}\theta^{\beta-1}. \\
\end{sis}
$$
\end{enumerate}
\end{fact}

\begin{remark}
Even though the proof of the above fact in \cite{GAC1} is over the complex numbers, it is not so difficult to extend it to a curve over any algebraically closed field. Actually, in this paper we only use \eqref{F:pushpull1} and the cases $\beta = 0, 1$ of \eqref{F:pushpull2}.  
\end{remark}
\subsection{Cones of cycles}\label{cones}

Let us introduce the cones of cycles we will be working with. Inside the real vector space $N^m(C_d), 0 \le m \le d$, consider the (convex) cone $\Eff^m(C_d)$ generated by effective codimension $m$ cycles (called the cone of \emph{effective cycles}) and its closure $\Pseff^m(C_d)$ (called the cone of \emph{pseudoeffective cycles}). These cones are salient by \cite[Prop. 1.3]{BFJ}, \cite[Thm. 1.4(i)]{fl1}. The intersection $\Efft^m(C_d):=\Eff^m(C_d)\cap R^m(C_d)$ is called the \emph{tautological effective} cone and its closure  $\Psefft^m(C_d):=\ov{\Efft^m(C_d)}$ is called the \emph{tautological pseudoeffective} cone. Note that there is an inclusion $\Psefft^m(C_d)\subseteq \Pseff^m(C_d)\cap R^m(C_d)$, which a priori could be strict. 

For $0 \le n \le d$ we set $\Eff_n(C_d):=\Eff^{d-n}(C_d)$ and similarly for the other cones. 

The dual of $\Pseff^{d-m}(C_d)$ (respectively of $\Psefft^{d-m}(C_d)$) cone is  the \emph{nef} cone $\Nef^m(C_d) \subset N^m(C_d)$ (resp. the \emph{tautological nef} cone $\Neft^m(C_d)\subset R^m(C_d)$). Note that there is an inclusion $\Nef^m(C_d)\cap R^m(C_d)\subseteq \Neft^m(C_d)$,  which a priori could be strict. 

Note that, by Fact \ref{verygen}, if $C$ is very general then $\Efft^m(C_d)=\Eff^m(C_d)$,  $\Psefft^m(C_d)= \Pseff^m(C_d)$ and $\Nef^m(C_d)=\Neft^m(C_d)$.

A  case where we know a complete description of 
these cones is the case of curves of genus one\footnote{If $g=0$, i.e. $C=\PP^1$ then $C_d=\PP^d$ and all the cones in question become one-dimensional, hence trivial.}, where such a description can be deduced from the work of Fulger \cite{Ful}. 


\begin{example}[Genus $1$]\label{genus1}
If the curve $C$ has genus $1$, then for any $1\leq m \leq d-1$ we have that $N^m(C_d)=R^m(C_d)$ and 
\begin{equation}\label{E:cones-g1}
\Pseff^m(C_d)=\Nef^m(C_d)=\cone\left(x^{m-1}\theta, x^m-\frac{m}{d}x^{m-1}\theta\right) \subset N^m(C_d)_{\R} \cong \R^2.
\end{equation}

Indeed, if $g=1$ then Fact \ref{projbund} implies that $C_d$ (for any $d>0$) is isomorphic to  the projective bundle $\PP(E_d^*)\to \Pic^d C\cong C$, with $E_d^*$ a stable vector bundle over $\Pic^d C$ of degree $1$ and rank $d$. This implies that  
$$N^*(C_d)=R^*(C_d)=\frac{\R[x,\theta]}{(\theta^2=0, x^d- 
x^{d-1} \theta\color{black})},$$
from which we get that $\dim N^m(C_d) = 2$ and it is generated by $\{x^m, x^{m-1}\theta\}$. Now \eqref{E:cones-g1} follows from  \cite[Prop. 1.5 and Lemma 2.2]{Ful}, using that $E_d^*$ is semistable of slope $\mu(E_d^*)=1/d$. 

We can give a new proof of this fact by combining the results of 
this paper and of 
\cite{BKLV}. Indeed,  
$x^m-\frac{m}{d}x^{m-1}\theta$ is proportional to the class of any codimension-$m$ diagonal  by Example \ref{E:r1} and therefore it is an extremal ray of $\Pseff^m(C_d)$ (and hence also of $\Eff^m(C_d)$) by Theorem \ref{T:B}.
On the other hand, 
$x^{m-1}\theta$ is an extremal ray of $\Pseff^m(C_d)$ (and hence also of $\Eff^m(C_d)$) by \cite[Theorem B(1)]{BKLV}. This gives the  description  of $\Pseff^m(C_d)$ and of $\Eff^m(C_d)$. 
$\Nef^m(C_d)$ is easily computed by taking the dual. 
We also get that $\Eff^m(C_d)=\Pseff^m(C_d)$.
\end{example}

\begin{remark}[Tautological pseudoeffective classes of divisors with $2\leq d\leq g$]\label{R:oldresults}
%
%
As we noticed in \S \ref{S:oldresults}, one of the boundary rays of the two-dimensional cone $\Pseff^1(C_d)$ is spanned by the class of the big diagonal $[\Delta]=2(d-1)![(d+g-1)x-\theta]$, and if in addition $2\leq d\leq g$, the other boundary ray is generated by 
$\theta-\mu_d(C)x$, where the pseudoeffective slope $\mu_d(C):= \max\left\{m\geq 0\left|\theta-mx\in \Psefft^1(C_d)\right.\right\}$ is positive.
The 
latter has been extensively studied 
and, although the problem of determining 
it is in general widely open, various interesting results have been obtained:\footnote{For simplicity we state these results for a very general curve over the complex numbers, but many are true, often with the same proof, for a general curve over an arbitrary algebraically closed field.}   

$\bullet$ $\mu_{g}(C)=1$ if $C$ is a very general curve of genus $g$ by \cite[Rmk. 1 after Thm. 5]{Kou}. Indeed,
\null \ \ \ \ \ \ in \cite[Thm. B(3)]{BKLV} we will prove that the same result is true for any curve $C$.

$\bullet$ $\mu_{g-1}(C)=1+\frac{1}{2g-3}$ if $C$ is a very general curve of genus $g$ by \cite[Prop. B]{Mus1}. 

$\bullet$ $\mu_{g-2}(C)=1+\frac{2}{g-2}$ if $C$ is a very general curve of genus $g$ by \cite[Cor. IV]{Mus2}.

$\bullet$ 
$\mu_d(C)\leq g-d+1$ by \cite[\S 4]{Kou}, and equality holds if $C$ is an hyperelliptic curve by

\noindent \null \hskip .7cm \null \cite[Thm. G(i) and Prop. H]{Mus1}.  

$\bullet$ If $C$ is very general then $\mu_d(C)\leq g-\lfloor\sqrt{d-1}\sqrt{g}\rfloor$ by \cite[Thm. 5]{Kou}.

$\bullet$ For a very general curve $C$ (and hence also for an arbitrary curve by \cite[Prop. 3]{Ott2}), and

\noindent \null \hskip .7cm $d\leq g-3$ we have: if $d\leq \sqrt{g}$ then $\mu_d(C)\geq g-(d-1)\sqrt{g}$ by \cite[Thm. 5(2)]{Kou}; if $d\geq \sqrt{g}$ 

\noindent \null \hskip .7cm then $\mu_d(C)\geq \frac{g}{d}$ by \cite[Thm. 1]{Tar}; $\mu_d(C)\geq 1+\frac{g-d}{g^2-dg+(d-2)}$ by \cite[Thm. A]{Mus1}. 
    
$\bullet$ If $d=2$ then $g-\sqrt{g}\leq \mu_2(C)$ for an arbitrary curve and $\mu_2(C)\leq g-\left\lfloor\sqrt{g}\right\rfloor$ for a very 

\noindent \null \hskip .7cm  general curve by \cite[Thm. 2]{Kou}. In particular $\mu_2(C)=g-\sqrt{g}$ if $g$ is a perfect square 

\noindent \null \hskip .7cm  and $C$ is a very general curve (see also \cite[Prop. 3.1]{CK}). 

$\bullet$ Let $C$ be a very general curve genus $g$. Nagata's conjecture on plane curves implies that 

\noindent \null \hskip .7cm  $\mu_2(C)=g-\sqrt{g}$ if $g\geq 10$ by \cite[Prop. 3.2]{CK} (see also \cite[Cor. 1.7]{Ros}). For small $g$ 

\noindent \null \hskip .7cm  it is known that $\mu_2(C)=\frac{4}{3}>3-\sqrt{3}$ if $g=3$ by \cite[Thm. 2(3)]{Kou}, while some bounds 

\noindent \null \hskip .7cm  on $\mu_2(C)$ for  $5\leq g \leq 8$ have been obtained by the first author in \cite[Thm. 1.1]{Bas1} and 

\noindent \null \hskip .7cm  \cite[Thm. 1.7]{Bas2}. For results on $\mu_2(C)$ of a non general curve, see \cite{Cha}.
\end{remark}

\section{Diagonal cone}\label{S:diagcone}

The aim of this section is to study the cone generated by the diagonal $n$-cycles for $1\leq n \leq d$. 

Recall that for any partition $\un a=(a_1,\ldots, a_n)\in \bbN^{n}$ of $d$ in 
$n$ parts, i.e. $a_1\geq a_2\geq \ldots\geq a_n\geq 1$ and $\sum_{i=1}^n a_i=d$, we have the $n$-dimensional diagonal cycle $\Delta_{\un a}$ defined in \eqref{E:diag}.
We start by computing its class $[\Delta_{\un a}]\in R_n(C_d)$ in terms of the standard basis of $R_n(C_d)$ (see Proposition \ref{basetaut}) via the following formula, which simplifies the 
one of \cite[Chap. VIII, Prop. 5.1]{GAC1}.


\begin{prop}\label{diagclass}
For any partition $\un a=(a_1,\ldots, a_n)\in \bbN_{>0}^n$ of  $d$ in  $n$ parts, we have that 
$$[\Delta_{\un a}]=\prod_{i=1}^n a_i \cdot \sum_{\alpha=0}^{r(n)}x^{d-n-\alpha}\theta^{\alpha}\frac{(-1)^{\alpha}}{\alpha!}\sum_{\beta=0}^{\alpha}(-1)^{\beta}\binom{\alpha}{\beta} \sum_{s=0}^{r(n)} Q_{s}^{\un a}(\beta),$$
where
$$
Q_{s}^{\un a}(\beta)=(n-s)!s! \binom{g-\beta}{s}\sigma_s(\un{a-1}),
$$
with $\sigma_s$ is the $s$-th elementary symmetric function (with the convention that $\sigma_0=1$), $\un{a-1}=(a_1-1,\ldots,a_n-1)$ and $r(n):=\min\{n,d-n,g\}$. 
\end{prop}
\begin{proof}
Consider the polynomials $\displaystyle P_{\underline{a}}(t):=1+ \sum_{i=1}^n a_i t_i$ and $\displaystyle R_{\underline{a}}(t):=1+ \sum_{i=1}^n a_i^2 t_i$.
Following  \cite[Chap. VIII, Prop. 5.1]{GAC1}, it can be shown that
\begin{equation}\label{diagGAC}
\left[\Delta_{\underline{a}}\right]=\sum_{\alpha=0}^{\min\{d-n, g\}} x^{d-n-\alpha}\theta^{\alpha}\frac{(-1)^{\alpha}}{\alpha!}\sum_{\beta=0}^{\alpha}(-1)^{\beta}\binom{\alpha}{\beta}
\left[\frac{R_{\underline{a}}(t)^{g-\beta}}{P_{\underline{a}}(t)^{g-\beta-n}}\right]_{t_1t_2\cdots t_n}\,,
\end{equation}
where $[\,\cdot\,]_{t_1t_2\cdots t_n}$ denotes the coefficient of the monomial ${t_1t_2\cdots t_n}$.



To be more precise, the above formula is shown to be true in the singular cohomology ring $H^*(C_d,\bbC)$ in \cite[Chap. VIII, Prop. 5.1]{GAC1} first for a very general complex curve $C$ using that the algebraic part of $H^*(C_d,\bbC)$  is generated by $x$ and $\theta$ (compare with Fact \ref{verygen}), and then it is deduced to be true in $H^*(C_d,\bbC)$ for any complex curve $C$ by a specialization argument. We then get the formula in $N^*(C_d)$ since homological equivalence implies numerical equivalence. We have checked that the same formula is true in the $l$-adic cohomology $H_{\rm et}^*(C_d,\bbQ_l)$ over an arbitrary field $k=\ov k$ (for any prime $l$ different from the characteristic of $k$), which again will imply the result in $N^*(C_d)$.  
To see that, we argue as follows. First of all, the formula \cite[(15.2)]{Mac} of 
Macdonald for the cohomological class  of $\Delta_{\un a}$ holds true in $H^*_{\rm et}(C_d,\bbQ_l)$  because it descends quite formally from the presentation of the cohomology of $C_d$ given in \cite[(6.3)]{Mac} and this presentation holds  true for $H^*_{\rm et}(C_d,\bbQ_l)$). 
Macdonald's formula shows that the class 
$[\Delta_{\un a}]$ in $H^*_{\rm et}(C_d,\bbQ_l)$ is a polynomial in $x$ and $\theta$ (
namely $\eta$ and $\sum_{i=1}^g \sigma_i$ in Macdonald's notation).
Knowing this, we can repeat verbatim the proof of \cite[Chap. VIII, Prop. 5.1]{GAC1} replacing 
$H^*(C_d,\bbC)$ with 
$H^*_{\rm et}(C_d,\bbQ_l)$ and get the same formula. 

\vspace{0.1cm}

We now manipulate the above formula \eqref{diagGAC} a little bit. We set $\displaystyle F_{\underline{a}}(t):=R_{\underline{a}}(t)-P_{\underline{a}}(t)=\sum_{i=1}^n a_i(a_i-1) t_i$ and we compute 
\begin{equation}\label{exp-F}
\begin{aligned}
   \frac{R_{\underline{a}}(t)^{g-\beta}}{P_{\underline{a}}(t)^{g-\beta-n}}  & = \frac{\left(P_{\underline{a}}(t)+F_{\underline{a}}(t)\right)^{g-\beta}}{P_{\underline{a}}(t)^{g-\beta-n}}
    = \frac{1}{P_{\underline{a}}(t)^{g-\beta-n}} \sum_{s=0}^{g-\beta} {{g-\beta}\choose s}\left(P_{\underline{a}}(t)^{g-\beta-s}\cdot F_{\underline{a}}(t)^s\right)= \\
   & = \sum_{s\geq 0} {{g-\beta}\choose s}P_{\underline{a}}(t)^{n-s}F_{\underline{a}}(t)^s.
\end{aligned}
\end{equation}
We now expand each summand in the last equation in power series
\begin{equation}\label{exp-series}
P_{\underline{a}}(t)^{n-s}F_{\underline{a}}(t)^s= \sum_{l\geq 0}\binom{n-s}{l}\left[\sum_{i=1}^na_it_i \right]^{l}\left[\sum_{i=1}^na_i(a_i-1)t_i \right]^s.
\end{equation}
From 
\eqref{exp-series} we see that 
$t_1\ldots t_n$ 
occurs in $P_{\underline{a}}(t)^{n-s}F_{\underline{a}}(t)^s$ if and only if $0\leq l=n-s$ and that 
\begin{equation}\label{coef-t}
\left[P_{\underline{a}}(t)^{n-s}F_{\underline{a}}(t)^s\right]_{t_1t_2\cdots t_n}=
\begin{cases}
(n-s)!\,s!\cdot \prod_{i=1}^n a_i \cdot \sigma_s(\un{a-1}) & \text{if }s\leq n,\\
0 & \text{if }s>n.
\end{cases}
\end{equation}
Combining \eqref{exp-F}, \eqref{exp-series} and \eqref{coef-t}, we get that 
\begin{equation}\label{pol-beta}
\left[ \frac{R_{\underline{a}}(t)^{g-\beta}}{P_{\underline{a}}(t)^{g-\beta-n}} \right]_{t_1\ldots t_n}= \sum_{s=0}^{n} \binom{g-\beta}{s}(n-s)!s!\prod_{i=1}^n a_i  \cdot \sigma_s(\un{a-1})=
\prod_{i=1}^n a_i \cdot \sum_{s=0}^n Q_s^{\un a}(\beta).
\end{equation}
Now, substituting \eqref{pol-beta} into \eqref{diagGAC}, we get the desired formula except for the fact that the summation over $\alpha$ goes until $\min\{d-n, g\}$ and the summation over $s$ goes until $n$. 

But $Q_s^{\un a}(\beta)=0$ for $s>g$ 
since then $\binom{g-\beta}{s}=0$. Moreover,  $Q_s^{\un a}(\beta)=0$ for $s>d-n$ because $\un{a-1}$, being  a partition of  $d-n$,  can have at most $d-n$ non-zero terms, which implies that  $\sigma_s(\un{a-1})=0$ if $s>d-n$. Hence, we can restrict the summation over $s$ until $r(n)$. 

Consequently, $\sum_{s=0}^{r(n)} Q_s^{\un a}(\beta)$ is a polynomial in $\beta$ of degree at most $r(n)$, hence $\sum_{\beta=0}^{\alpha}(-1)^{\beta}\binom{\alpha}{\beta} \sum_{s=0}^{r(n)} Q_{s}^{\un a}(\beta) = 0$ for $r(n)<\alpha$ by Lemma \ref{binom} below. Therefore, we can restrict the summation over $\alpha$ until $r(n)$, and the formula is now proved.
\end{proof}

\begin{lemma}\label{binom}
Let $P(x)=\sum_{n=0}^{d} c_n \binom{x}{n}\in \Z[x]$ be a polynomial (in the Newton expansion). Then  
$$\sum_{j=0}^n (-1)^j\binom{n}{j} P(j)=(-1)^n c_n.$$
In particular, if $Q(x)=b_0+\ldots+ b_n x^n \in \Z[x]$ a polynomial of degree at most $n$ then 
$$\sum_{j=0}^n (-1)^j\binom{n}{j} Q(j)=(-1)^n n! b_n.$$
\end{lemma}
\begin{proof}
See \cite[p.190]{GKP}.
\end{proof}


We now introduce the main object of interest of this section.

\begin{defi}\label{D:diagcone}
For any $1 \le n \le d-1$, the \emph{$n$-dimensional diagonal cone}, which we will denote by $\D_n(C_d)$, is the cone inside $R_n(C_d)$ spanned by the diagonal cycles of dimension $n$. 
\end{defi}
Note that we have an inclusion of cones $\D_n(C_d)\subset \Efft_n(C_d) \subset \Psefft_n(C_d)$ and  similarly for the non-tautological cones of effective and pseudoeffective cycles. 

In the remaining of this section, we want to study the structure of the diagonal cone $\D_n(C_d)$. 
To this aim, let us introduce some notation.
We will denote by $\P_{\leq n}(d-n)$ the partitions of $d-n$ with at most $n$ parts, i.e.
$$\P_{\leq n}(d-n)=\{\un{\lambda}=(\lambda_1, \ldots, \lambda_n)\in \bbN^n\: : \lambda_1\geq \ldots\geq \lambda_n\geq 0, \: \sum_{i=1}^n \lambda_i=d-n\}.$$ 
Note that an element $\un \lambda\in \P_{\leq n}(d-n)$ can have at most $d-n$ non-zero parts; hence, if we set 
$$s(n):=\min\{n,d-n\},$$ 
we have a natural identification $\P_{\leq n}(d-n)\cong \P_{\leq s(n)}(d-n)$ obtained by forgetting the last $n-s(n)$ entries, which are necessarily zero. 
Given a partition $\un \lambda\in \P_{\leq n}(d-n)$, we get a partition of $d$ with exactly $n$ parts by setting $\un{\lambda+1}:=(\lambda_1+1,\ldots, \lambda_n+1)$; hence we have an associated $n$-dimensional diagonal cycle $\Delta_{\un{\lambda+1}}$. 

We now introduce the balanced partitions, (a subset of) which will give rise to extremal rays of $\D_n(C_d)$. For any $1\leq j \leq s(n)$, we can write $d-n=\alpha(j)\cdot j +\rho(j)$ for some unique integers $\alpha(j), \rho(j)\in \bbN$ with the property that  $0\leq \rho(j)<j$ and  $\alpha(j)\geq 1$ (using that $j\leq s(n)\leq d-n$). For any such $j$, define the $j$-th \emph{balanced partition} $\un \lambda^j=(\lambda^j_1,\ldots, \lambda^j_{s(n)})\in \P_{\leq s(n)}(d-n)$ as


\begin{equation}\label{balpart}
\un \lambda^j:=(\underbrace{\alpha(j)+1, \cdots, \alpha(j)+1}_{\rho(j)}, \underbrace{\alpha(j), \cdots, \alpha(j)}_{j-\rho(j)}, 0, \cdots, 0) \in \P_{\leq s(n)}(d-n).
\end{equation}
In other words, $\un \lambda^j$ has only $j$ non-zero entries and they are in the most possible ``balanced"  configuration. The associated $n$-dimensional diagonal will be denoted by $\Delta^{\rm bal}_{j, n}:=\Delta_{\un{\lambda^j+1}}$ and it will be called the  \emph{$j$-th balanced $n$-dimensional diagonal}.   

After 
this notation, we are ready to describe the diagonal cone $\D_n(C_d)$ inside $R_n(C_d)$. Recall from Proposition \ref{basetaut} that $R_n(C_d)$ has dimension $r(n)+1$, where $$r(n):=\min\{n, d-n, g\}$$ 
and note that $r(n)\leq s(n)$.

\begin{thm}\label{T:diagcone}
For $1\leq n\leq d-1$ and $g\geq 1$\footnote{Note that the cases $n=0$, $n=d$ or $g=0$ are trivial since in each of these cases $R^n(C_d)$ has dimension one.}, the $n$-dimensional diagonal cone $\D_n(C_d)$  is a  rational polyhedral cone of dimension $r(n)$ (hence it has codimension one in $\Efft_n(C_d)$ and in $\Psefft_n(C_d)$), whose extremal rays are generated by:
\begin{itemize}
\item $\Delta^{\rm bal}_{1,n}, \ldots,  \Delta^{\rm bal}_{s(n), n}$ if $r(n)\geq 4$; 
\item $\Delta^{\rm bal}_{1,n}, \{ \Delta^{\rm bal}_{j,n}\: :\: j\:  \text{is a $(d-n)$-break} \}, \Delta^{\rm bal}_{s(n), n}$ if $r(n)=3$, where a number $2\leq j \leq s(n)-1$ is called a 
$(d-n)$\emph{-break} \color{black}if 
$$\left\lceil \frac{d-n}{j-1}\right\rceil > \left\lfloor \frac{d-n}{j+1}\right\rfloor+1 \:\left(\text{or equivalently if} \left\lceil \frac{d-n}{j-1}\right\rceil 
\neq\color{black}\left\lfloor \frac{d-n}{j+1}\right\rfloor+1\right)
$$
\item $\Delta^{\rm bal}_{1,n}$ and $\Delta^{\rm bal}_{s(n), n}$ if $r(n)=2$.
\end{itemize}

In particular,  if $s(n)=r(n)$ (which happens if and only if $n\leq g$ or $d-n\leq g$) then $\D_n(C_d)$ is simplicial and each balanced diagonal generates an extremal ray of $\D_n(C_d)$. 
\end{thm}


In the case $r(n)=1$,  the above theorem says that $\D_n(C_d)$ is one dimensional, which is equivalent to say that all the diagonals have proportional classes. 
This can be easily checked. 

\begin{example}\label{E:r1}(Diagonal cone for $r(n)=1$)
If $r(n)=1$ (which happens if and only $g=1$ and $1\leq n\leq d-1$ or $g\geq 2$ and $n=1$ or $n=d-1$),
the class of any $n$-dimensional diagonal $\Delta_{\un a}$ is equal to (using Proposition \ref{diagclass}) 
$$[\Delta_{\un a}]=\prod\limits_{i=1}^n a_i (n-1)! \left\{[n+g(d-n)]x^{d-n}-(d-n)x^{d-n-1}\theta\right\}. 
$$ 
In particular, the big diagonal (i.e. the unique diagonal of codimension one) has class
$$[\Delta_{(2,1,\ldots,1)}]=2(d-2)!  \left\{[d-1+g]x-\theta \right\},$$
while the class of the small diagonal (i.e. the unique diagonal of dimension one) is 
$$[\Delta_{(d)}]=d\left\{[1+g(d-1)]x^{d-1}-(d-1)x^{d-2}\theta\right\}.
$$ 
In genus $g=1$, all the diagonals have proportional classes: explicitly, we have 
$$[\Delta_{\un a}]=\prod\limits_{i=1}^n a_i (n-1)! \left[ d x^{d-n}-(d-n) x^{d-n-1}\theta\right].$$
\end{example}
To prove Theorem \ref{T:diagcone} 
when $r(n)\geq 2$, we will  introduce another basis of 
$R_n(C_d)$, with respect to which the classes of the diagonals of dimension $n$ can be written in a very simple form. 

\begin{lemma}\label{newbasis}
For any $0\leq n\leq d$, the elements
$$w_s:=s!(n-s)! \sum_{\alpha=0}^{s}x^{d-n-\alpha}\theta^{\alpha} \frac{(-1)^{\alpha}}{\alpha!}\sum_{\beta=0}^{\alpha} (-1)^{\beta}\binom{\alpha}{\beta}  \binom{g-\beta}{s}\in R_n(C_d) \quad \text{ for any } \: 0\leq s\leq r(n),$$
form a basis of $R_{n}(C_d)$. For any partition  $\un \lambda\in \P_{\leq n}(d-n)$, the  class of the associated diagonal cycle $\Delta_{\un{\lambda+1}}$, suitably normalized, can be written as  
$$\wt \delta_{\un{\lambda+1}}:=\frac{[\Delta_{\un{\lambda+1}}]}{\prod_{i=1}^n (\lambda_i+1)}=\sum_{s=0}^{r(n)} \sigma_s(\un{\lambda})w_s=w_0+(d-n)w_1+\sum_{s=2}^{r(n)} \sigma_s(\un{\lambda})w_s.
$$
\end{lemma}
\begin{proof}
Using Lemma \ref{binom}, we get that the coefficient of $x^{d-n-s}\theta^s$ in $w_s$ is equal to 
$$s!(n-s)!\frac{(-1)^s}{s!}\sum_{\beta=0}^s(-1)^{\beta} \binom{s}{\beta}\binom{g-\beta}{s}=(-1)^s(n-s)!\neq 0.$$
Hence the matrix that gives the coordinates of the $w_s$'s with respect to the standard basis of $R_n(C_d)$ (see Proposition \ref{basetaut}(i)) is triangular with non-zero entries on the diagonal, hence it is invertible.
This implies that 
$\{w_s\}_{0\leq s \leq r(n)}$ 
is a basis of $R_n(C_d)$. 

The formula for $[\Delta_{\un{\lambda+1}}]$ follows directly from Proposition \ref{diagclass} (and the easy facts that $\sigma_0(\un{\lambda})=1$ and $\sigma_1(\un{\lambda})=d-n$), observing that fixing $s$ in the summation, the terms with  $\alpha>s$ vanish by Lemma \ref{binom} since $Q_s^{\un{\lambda+1}}(\beta)$ is a polynomial of degree $s$. 
\end{proof}


Lemma \eqref{newbasis} allows us to explicitly describe a polytope which is a bounded section of $\D_n(C_d)$. 

We first introduce 
some notation.
Fix three natural numbers $t\geq s\geq r\geq 2$. 
Let $(\bbN^s)^{\sum=t}$ be the finite subset of $\bbN^s$ 
of all 
$\un x=(x_1,\ldots, x_s)$ such that $\sum_{i=1}^s x_i=t$ and consider the map
\begin{equation*}
\begin{aligned}
\Sigma_{\leq r}: (\bbN^s)^{\sum=t} & \longrightarrow \bbN^{r-1}\subset \R^{r-1}\\
 \un x & \mapsto (\sigma_2(\un x), \ldots \sigma_r(\un x)).
 \end{aligned}
\end{equation*}
We will denote by $\Pi(t,s,r)$ the (integral) \emph{polytope} in $\bbR^{r-1}$ which is the convex hull of the image of $\Sigma_{\leq r}$. Note that the set $\P_{\leq s}(t)$ of partitions of $t$ with at most $s$ parts is naturally identified with the subset of $(\bbN^s)^{\sum=t}$  consisting of all the vectors whose entries are in non-increasing order. Since 
$\Sigma_{\leq r}$ is invariant under permutation of the entries of the vectors of $(\bbN^s)^{\sum=t}$, the polytope $\Pi(t,s,r)$ is also the convex hull of the image of the elements of $\P_{\leq s}(t)$ via 
$\Sigma_{\leq r}$. 


\begin{cor}\label{C:secdiag}
If $r(n)\geq 2$, then 
$\Pi(d-n,s(n),r(n))$ is a bounded section of $\D_n(C_d)$.\color{black}
\end{cor}
\begin{proof}
With notation as in Lemma \ref{newbasis}, for any $0\leq s \leq r(n)$ let $L_{w_s}$ be the dual of $w_s$, i.e. the linear functional on $N_n(C_d)$ which is $1$ on $w_s$ and $0$ on $w_p$ for $p\neq s$. 
From Lemma \ref{newbasis}, it follows that all the 
$\wt \delta_{\un{\lambda+1}}$ lie on the codimension two affine subspace $\{L_{w_0}=1\}\cap \{L_{w_1}=d-n\}$ of $N_n(C_d)$. Therefore, $\D_n(C_d)$ is the cone over the polytope
$$\conv\left(\wt \delta_{\un{\lambda+1}}, \un \lambda \in \P_{\leq n}(d-n)\right)\subset \aff\left(\wt \delta_{\un{\lambda+1}}, \un \lambda \in \P_{\leq n}(d-n) \right) \subseteq \{L_{w_0}=1\}\cap \{L_{w_1}=d-n\}.$$
 Lemma \ref{newbasis}, together with the fact that $\P_{\leq n}(d-n)\cong \P_{\leq s(n)}(d-n)$ (as observed above), implies that  $\conv\left(\wt \delta_{\un{\lambda+1}}, \un \lambda \in \P_{\leq n}(d-n)\right)$ can be identified with 
$\Pi(d-n,s(n),r(n))$. 
\end{proof}


Using the above corollary, the proof of Theorem \ref{T:diagcone} for $r(n)\geq 2$ follows straightforwardly from the following result describing the dimension and the vertices of the polytopes $\Pi(t,s,r)$ in terms of balanced partitions. 
Recall that, for any $1\leq j \leq s$, if we write 
\begin{equation}\label{E:expr}
t=\alpha(j)\cdot j +\rho(j) \text{ for some } \alpha(j), \rho(j)\in \bbN \text{ with the property that  } 0\leq \rho(j)\leq j,
\end{equation}
then the $j$-th  \emph{balanced partition}  is defined as
 \begin{equation}\label{balpart2}
 \un \lambda^j:=(\underbrace{\alpha(j)+1, \cdots, \alpha(j)+1}_{\rho(j)}, \underbrace{\alpha(j), \cdots, \alpha(j)}_{j-\rho(j)}, 0, \cdots, 0) \in  \P_{\leq s}(t)\subset (\bbN^s)^{\sum=t},
\end{equation}
Notice that the expression \eqref{E:expr} is unique if $j$ does not divide $t$, while if $j$ divides $t$ there are two such expressions, namely $\displaystyle t=\frac{t}{j} \cdot j+0$ and $\displaystyle t=\left(\frac{t}{j}-1\right) j+j$. However the definition \eqref{balpart2} does not depend on the chosen expression.  


\begin{prop}\label{P:vertpoly}
Fix three natural numbers $t\geq s\geq r\geq 2$. Then the polytope $\Pi(t,s,r)$ has dimension $r-1$ and its vertices are:
\begin{enumerate}[(A)]
\item $ \Sigma_{\leq r}(\un \lambda^1)=0$ and $ \Sigma_{\leq r}(\un \lambda^s)$ if $r=2$; 
\item $\Sigma_{\leq r}(\un \lambda^1)=0, \{ \Sigma_{\leq r}(\un \lambda^j)\: :\: j\:  \text{ is a $t$-break} \},  \Sigma_{\leq r}(\un \lambda^s)$ if $r=3$, where, as in Theorem \ref{T:diagcone}, a number $2\leq j \leq s-1$ is called a $t$-break if 
$$\left\lceil \frac{t}{j-1} \right\rceil>\left\lfloor \frac{t}{j+1}\right\rfloor +1 \: \left(\text{or equivalently if } \left\lceil \frac{t}{j-1} \right\rceil\neq \left\lfloor \frac{t}{j+1}\right\rfloor +1\right);$$
\item $\Sigma_{\leq r}(\un \lambda^1)=0, \ldots,   \Sigma_{\leq r}(\un \lambda^s)$ if $r\geq 4$.
\end{enumerate}
In particular 
$\Pi(t,r,r)$ is a $(r-1)$-dimensional simplex with vertices 
$\Sigma_{\leq r}(\un \lambda^1), \ldots,   \Sigma_{\leq r}(\un \lambda^r)$.
\end{prop}


We start by computing the dimension of 
$\Pi(t,s,r)$ and bounding the number of its vertices.  

\begin{lemma}\label{elemfunc}
Fix three natural numbers $t\geq s\geq r\geq 2$. Then the polytope $\Pi(t,s,r)$ has dimension  $r-1$ and its  vertices belong to the subset  $\{\Sigma_{\leq r}(\un \lambda^1)=0, \Sigma_{\leq r}(\un \lambda^2), \ldots, \Sigma_{\leq r}(\un \lambda^s)\}$. 
\end{lemma}
\begin{proof}
We will distinguish two cases, according to whether $s=r$ or $s>r$. 

\un{Case I}: $s=r$. 
For any $1\leq j \leq r$, the vector $\Sigma_{\leq r}(\un \lambda^j)$ has only the first $(j-1)$ coordinates different from zero. This shows that the points $\{\Sigma_{\leq r}(\un \lambda^1)=0, \Sigma_{\leq r}(\un \lambda^2), \ldots, \Sigma_{\leq r}(\un \lambda^r)\}$ are affinely independent; hence, their convex hull $P:=\conv(\{\Sigma_{\leq r}(\un \lambda^j)\}_{1\leq j \leq r})$ is a $(r-1)$-dimensional  simplex  whose vertices are exactly the points $\{\Sigma_{\leq r}(\un \lambda^j)\: : 1\leq j \leq r\}$.  Clearly $P$ is contained inside the convex hull of the image of $\Sigma_{\leq r}$, and we will be done if we show that it is equal to it. 

Now $P$ can be written as the intersection of $r$ half spaces $\{l_i\geq 0\}_{i=1,\ldots, r}$, where $H_i:= \{l_i=0\}$ are the supporting affine hyperplanes of $P$, i.e. those affine hyperplanes that cut out the $r$ facets of $P$.  
To show that the convex hull of the image of $\Sigma_{\leq r}$ is contained in $P$ (hence it is equal to it), it is enough to show that the image of $\Sigma_{\leq r}$ is contained in $\{l_i\geq 0\}$ for any $1\leq i \leq r$. Fix any such 
$H_i$ and note that its inverse via 
$\Sigma_{\leq r}$ is defined by a function $f_i=a_0+a_2\sigma_2+\ldots+a_r\sigma_r$, for some $a_j\in \R$.   
Since $\Sigma_{\leq r}(\un \lambda^j)\in P$ for any $1\leq j \leq r$, we get that $f_i(\un \lambda^j)\geq 0$ for any $1\leq j \leq r$. We conclude that $f_i\geq 0$ on the entire  space $(\bbN^s)^{\sum=t}$ by the next result. 

\begin{sublemma}\label{posfunct}
Consider a symmetric real valued function of the form  
\begin{equation*}
\begin{aligned}
f: (\bbN^r)^{\sum=t} & \longrightarrow \R\\
 \un x & \mapsto f(\un x)=a_0 +a_2\sigma_2(\un x)+\ldots+a_r\sigma_r(\un x),
 \end{aligned}
\end{equation*}
If $f(\un \lambda^j)\geq 0$ for every $1\leq j \leq r$ then $f$ is non-negative on the entire domain. 
\end{sublemma} 
\begin{proof}
This 
is an integral version of an old result over the real numbers due originally to Chebyshev \cite{Che}, and rediscovered (and generalized) many times since then (see \cite{KKR} and the references therein).
Our proof is a variant of the one in \cite[Lemma 2.4]{Mit}.

Observe that the elements $\un \lambda^j$, and their conjugates under the action of the symmetric group $S_{r}$ permuting the coordinates of $(\bbN^r)^{\sum=t}$, are exactly those elements of $(\bbN^r)^{\sum=t}$
such that any pair of non-zero entries is ``balanced", i.e. it is formed by two integers that are either equal or consecutive. Therefore, it is enough to show 
that any function $f$ as in the statement achieves its minimum (which exists since $(\bbN^r)^{\sum=t}$ is a finite set) on one of those elements. 

With this aim, take a minimum $\un w=(w_1, \ldots, w_r)$ of $f$. We can assume that there are at least two non-zero coefficients, say $w_1$ and $w_2$ up to permuting the indices (note that $f$ is a symmetric function), such that $|w_1-w_2|\geq 2$, for otherwise  we are trivially done. Set $
c\color{black}:=w_1+w_2$ 
(observe that $
c\color{black}\geq 3$) and consider the function
$$
\begin{aligned} 
F: \bbN\cap [0, 
c\color{black}] & \longrightarrow \R, \\
z & \mapsto f(z,  
c\color{black}-z, w_3, \ldots, w_r). 
\end{aligned}
$$
Now $F$ admits a very simple description. Indeed, since $f$ is an affine combination of elementary symmetric functions, we can write $f(x_1,\ldots, x_r)=a_1+(x_1+x_2)g(x_3,\ldots x_r)+ x_1x_2h(x_3,\ldots, x_r)$ for some functions $g$ and $h$ not depending on $x_1$ and $x_2$. 
Hence the function $F$ can be written as $F(z)=\alpha z(
c\color{black}-z)+\beta$ for some $\alpha, \beta \in \R$. We distinguish three cases:

\begin{enumerate}
\item if $\alpha=0$ then $F$ is constant;
\item if $\alpha>0$ then the only minimum points of $F$ are $0$ and $
c\color{black}$;
\item if $\alpha<0$ then $
F\color{black}$ achieves its minimum at one (or both) of the two points 
(with possibly equal values) \color{black}$\{\lfloor 
c\color{black}/2\rfloor, \lfloor 
c\color{black}/2\rfloor+1\}$    
\end{enumerate}
Note that by construction $F$ achieves its minimum at $w_1$. Since $1\leq w_1<w_1+w_2 
=c\color{black}$, case (2) cannot occur. Moreover, since $|w_1-w_2|\geq 2$, then case (3) cannot occur either. Hence, case (1) must occur, i.e. $F$ must be constant.  Therefore, the point 
$\un w':=(\lfloor 
c\color{black}/2\rfloor, 
c\color{black}-\lfloor 
c\color{black}/2\rfloor, w_3,\ldots, w_r)$ is such that $f(\un w')=f(\un w)$, which implies that 
also $\un w'$ is \color{black}a minimum for $f$. We have then constructed a new minimum of $f$ by leaving untouched the last $r-2$ coordinates and replacing the first two by a ``balanced" couple. Iterating this construction over all the pairs of non ``balanced" couple in a minimum, we finally arrive at a minimum for which every pair of non-zero elements is balanced. This proves Sublemma \ref{posfunct}.
\end{proof}

\un{Case II}: $s>r$. Observe that the map $\Sigma_{\leq r}$ is the composition of $\Sigma_{\leq s}$ with the projection $\pi:\R^{s-1}\to \R^{r-1}$ obtained by forgetting the last $s-r$ coordinates. Using the (previously proved) Case I  for the map $\Sigma_{\leq s}$, it follows that 
$$\conv(\Im \Sigma_{\leq r})=\pi(\conv(\Im \Sigma_{\leq s}))=\pi\big(\conv(\{\Sigma_{\leq s}(\un \lambda^j)\}_{1\leq j \leq s})\big)=\conv(\{\Sigma_{\leq r}(\un \lambda^j)\}_{1\leq j \leq s}).$$
It remains to observe that the first $r$ vectors $\{\Sigma_{\leq r}(\un \lambda^1)=0, \Sigma_{\leq r}(\un \lambda^2), \ldots, \Sigma_{\leq r}(\un \lambda^r)\}$ are affinely independent (same proof as in the previous case), which forces  
the polytope $\conv(\{\Sigma_{\leq r}(\un \lambda^j)\}_{1\leq j \leq s})$ to have dimension $(r-1)$ and we are done.
\end{proof}

Before giving a proof of the Proposition \ref{P:vertpoly}, we record in the following remarks two facts about elementary symmetric functions that we are going to use several times during the proof.  


\begin{remark}\label{R:schur}
It is easy to check that the balanced partitions form a chain under majorization partial order (in the sense of \cite[Chap. I, A.1]{MOA}):
$$\un \lambda^s\prec\ldots \prec \un \lambda^1.$$ 

Since, for any $2\leq k \leq n$,  the function $\sigma_k$ is strictly Schur-concave (see \cite[Chap. III, F.1]{MOA}) and the function $\frac{\sigma_k}{\sigma_{k-1}}$ is Schur-concave (see \cite[Chap. III, F.3]{MOA}), it follows that
\begin{equation}\label{E:inc-coord}
0= \sigma_k(\un \lambda^1)<\sigma_k(\un \lambda^2) \ldots <\sigma_k(\un \lambda^s),
 \end{equation}
\begin{equation}\label{E:inc-slope}
 \frac{\sigma_k(\un \lambda^{k-1})}{\sigma_{k-1}(\un \lambda^{k-1})}\leq \ldots \leq \frac{\sigma_k(\un \lambda^s)}{\sigma_{k-1}(\un \lambda^s)}.
\end{equation}
\end{remark}


\begin{remark}\label{R:Newt}
In what follows, we will have to compute the difference $\sigma_k(\un \lambda^{j_1})-\sigma_k(\un \lambda^{j_2})$ for small values of $k$, namely $k=2,3,4$. A convenient way to achieve this is to use Newton's identities 
between elementary and power sum symmetric polynomials
$$k\sigma_k=\sum_{i=1}^k (-1)^{i-1}\sigma_{k-i}p_i,$$
where 
each \color{black}$p_i(x_1,\ldots,x_s):=\sum_{j=1}^s x_j^i$ 
is the $i$-th power sum polynomial. Using the above formula, we can explicitly write 
\begin{equation}\label{E:Newtrel}
\sigma_1=p_1, \ \sigma_2=\frac{p_1^2-p_2}{2}, \ \sigma_3=\frac{p_1^3-3p_1p_2+2p_3}{6}, \ \sigma_4 = \frac{p_1^4-6p_1^2p_2+8p_1p_3+3p_2^2-6p_4}{24}.
\end{equation}


The advantage of the power sum polynomials is that they are easy to compute for the balanced partitions. Indeed, writing $t=\alpha(j) j+\rho(j) $ as in \eqref{E:expr} and setting $\t(j):=\r(j)(\a(j)+1)$, we can easily evaluate the $k$-th power sum polynomial 
on the balanced partition $\un \lambda^j$ of  \eqref{balpart2} 
\begin{equation}\label{E:p-lambda}
\begin{aligned}
& p_k(\underline{\lambda}^j) = \rho(j) (\alpha(j) +1)^k+(j-\rho(j))\alpha(j)^k 
= j \alpha(j)^k+\rho(j) [(\alpha(j)+1)^k-\alpha(j)^k]\\
& =(t-\rho(j)) \alpha(j)^{k-1} + \rho(j)[(\alpha(j)+1)^k-\alpha(j)^{
k\color{black}}]
= t\alpha(j)^{k-1}+[(\alpha(j)+1)^{k-1}-\alpha(j)^{k-1}]\tau(j).
\end{aligned}
\end{equation}
\end{remark}


We can now give a proof of Proposition \ref{P:vertpoly} distinguishing between the three cases $r=2$, $r=3$ and $r\geq 4$. The case $r=2$ is very easy. 

\begin{proof}[Proof of Proposition \ref{P:vertpoly} for $r=2$]
In this case, $\Pi(t,s,2)$ is, by Lemma \ref{elemfunc}, a $1$-dimensional polytope, i.e. a segment, and hence it has two vertices. Using \eqref{E:inc-coord} for $k=2$, we conclude that the two vertices are $ \sigma_2(\un \lambda^ 1)=\Sigma_{\leq 2}(\un \lambda^1)=0$ and $\sigma_2(\un \lambda^s)= \Sigma_{\leq 2}(\un \lambda^s)$.
\end{proof}


The proof of Proposition \ref{P:vertpoly} in the case $r=3$ is based on the following result.

\begin{lemma}\label{L:vanishdet}
Consider three natural numbers $1\leq j_1<j_2<j_3\leq s$. Then 
\begin{equation}\label{E:det>0}
D(j_1,j_2,j_3):=\det \begin{pmatrix}
\sigma_2(\un \lambda^{j_2})-\sigma_2(\un \lambda^{j_1}) & \sigma_2(\un \lambda^{j_3})-\sigma_2(\un \lambda^{j_2}) \\
\sigma_3(\un \lambda^{j_2})-\sigma_3(\un \lambda^{j_1}) & \sigma_3(\un \lambda^{j_3})-\sigma_3(\un \lambda^{j_2}) \\
\end{pmatrix}\geq 0
\end{equation}
with equality if and only if
$\displaystyle \left\lceil \frac{t}{j_1} \right\rceil=\left\lfloor \frac{t}{j_3}\right\rfloor +1$.
\end{lemma}
\begin{proof}
By the assumptions on the $j_i$'s, we can find three expressions as in \eqref{E:expr} that satisfy 
\begin{equation}\label{E:3divi}
\begin{sis}
t=\alpha_1 j_1+\rho_1 & \text{ for some } 0<\rho_1\leq j_1 , \\
t=\alpha_2 j_2+\rho_2 & \text{ for some } 0\leq\rho_2\leq j_2 , \\
t=\alpha_3 j_3+\rho_3 & \text{ for some } 0\leq \rho_3< j_3 , \\
\end{sis} 
\end{equation}
where we have set, for simplicity, $\a_h:=\alpha(j_h)$ and $\r_h:=\rho(j_h)$ for any $h=1,2,3$. For later use, we also set  $\tau_h:=\r_h(\a_h+1)$ and observe that the inequalities on the $\r_h$'s in \eqref{E:3divi}  are equivalent to the following inequalities on the $\t_h$'s:
\begin{equation}\label{E:ineqtau}
\begin{sis}
 0<\t_1=\r_1(\a_1+1)\leq j_1\a_1+\r_1=t, \\
 0\leq \t_2=\r_2(\a_2+1)\leq j_2\a_2+\r_2=t, \\
0\leq \t_3=\r_3(\a_3+1)< j_3\a_3+\r_3=t. \\
\end{sis} 
\end{equation}
In terms of the above decompositions \eqref{E:3divi}, it is easy to see that 
\begin{equation*}
\left\lceil \frac{t}{j_1} \right\rceil=\left\lfloor \frac{t}{j_3}\right\rfloor +1
\Longleftrightarrow \a_1=\a_2=\a_3. 
\end{equation*}
Therefore, we have to show that $D(j_1,j_2,j_3)\geq 0$ with equality  exactly when  $\alpha_1=\alpha_2=\alpha_3$. 

We compute $D(j_1,j_2,j_3)$.
Since $\sigma_1(\un \lambda^{j_i})=p_1(\un \lambda^{j_i})=t$ for 
$i=1,2,3$, by subtracting  to the second row the first row multiplied by $t$ and using 
\eqref{E:Newtrel}, we get that $D(j_1,j_2,j_3)$ is equal to 
\begin{equation}\label{E:det-p}
-\frac{1}{6}  \det \begin{pmatrix}
p_2(\un \lambda^{j_2})-p_2(\un \lambda^{j_1}) & p_2(\un \lambda^{j_3})-p_2(\un \lambda^{j_2}) \\
p_3(\un \lambda^{j_2})-p_3(\un \lambda^{j_1}) & p_3(\un \lambda^{j_3})-p_3(\un \lambda^{j_2}) \\
\end{pmatrix}= -\frac{1}{6}\sum_{h=1}^3 \left[p_2(\un \lambda^{j_h})p_3(\un \lambda^{j_{h+1}})- p_2(\un \lambda^{j_{h+1}})p_3(\un \lambda^{j_h})\right],
\end{equation} 
with the modulo $3$ convention $j_4:=j_1$. 
Using formulas \eqref{E:p-lambda}, we compute (for any $h=1,2,3$)
\begin{equation*}
p_2(\un \lambda^{j_h})p_3(\un \lambda^{j_{h+1}})- p_2(\un \lambda^{j_{h+1}})p_3(\un \lambda^{j_h})=t^2\a_h\a_{h+1}(\a_{h+1}-\a_h)+t\t_{h+1}\a_h(2\a_{h+1}-\a_h+1)+
\end{equation*}
$$-t\t_h\a_{h+1}(2\a_h-\a_{h+1}+1)+2\t_h\t_{h+1}(a_{h+1}-\a_h). 
$$
Substituting the above formulas for each value of $h$  into \eqref{E:det-p} and grouping together all the terms with $t^2$, $t\r_h$ and $\r_h\r_k$, we reach the  formula 
\begin{equation}\label{E:fordet}
6 D(j_1,j_2,j_3) = t^2(\a_1-\a_2)(\a_2-\a_3)(\a_1-\a_3)+ t\t_1(\a_2-\a_3)(2\a_1-\a_2-\a_3+1)+
\end{equation}
$$  -t\t_2(\a_1-\a_3)(2\a_2-\a_1-\a_3+1)+t\t_3(\a_1-\a_2)(2\a_3-\a_1-\a_2+1)+$$
$$+2\t_1\t_2(\a_1-\a_2)-2\t_1\t_3(\a_1-\a_3)+2\t_2\t_3(\a_2-\a_3).$$ 

\vspace{0.1cm}

We now use \eqref{E:fordet} to conclude. 
First, 
it is clear from 
\eqref{E:fordet} that $D(j_1,j_2,j_3)=0$ when $\a_1=\a_2=\a_3$. So we will assume that either $\a_1>\a_2$ or $\a_2>\a_3$ and prove that $D(j_1,j_2,j_3)>0$.

If we set $m:=\a_1-\a_2\geq 0$ and $n:=\a_2-\a_3\geq 0$, then 
\eqref{E:fordet} becomes
\begin{equation}\label{E:newfor}
6 D(j_1,j_2,j_3):= nm(m+n)t^2+n(2m+n+1)t\t_1+[m(m-1)-n(n+1)]t\t_2-m(m+2n-1)t\t_3+
\end{equation}
$$
+2m\t_1\t_2-2(n+m)\t_1\t_3+2n\t_2\t_3.
$$
To prove the positivity of $D(j_1,j_2,j_3)$, we will distinguish three cases:

\begin{enumerate}

\item If $m=0$ (and hence $n\geq 1$) then 
\eqref{E:newfor} becomes
$$6 D(j_1,j_2,j_3)=n(n+1)t\t_1-n(n+1)t\t_2-2n\t_1\t_3+2n\t_2\t_3=n[(n+1)t-2\t_3](\t_1-\t_2).
$$
Then we get that $D(j_1,j_2,j_3)>0$ because $\t_1>\t_2$ since $\a_1=\a_2$ and $\r_1>\r_2$ (using that $j_1<j_2$) and $t>\t_3$ by  \eqref{E:ineqtau}.

\item If $n=0$ (and hence $m\geq 1$) then 
\eqref{E:newfor} becomes
$$6 D(j_1,j_2,j_3)=m(m-1)t\t_2-m(m-1)t\t_3+2m\t_1\t_2-2m\t_1\t_3=m(\t_2-\t_3)[(m-1)t+2\t_1].
$$
Then we get that $D(j_1,j_2,j_3)>0$  because $\t_2>\t_3$ since $\a_2=\a_3$ and $\r_2>\r_3$ (using that $j_2<j_3$) and $\t_1>0$ by \eqref{E:ineqtau}.

\item If $m\geq 1$ and  $n\geq 1$ then we factorize 
\eqref{E:newfor} as follows
$$
6 D(j_1,j_2,j_3)=n(mt-\t_2)[(n+1)t-2\t_3]+m(m-1)(nt+\t_2-\t_3)t+[n(2m+n+1)t-2(n+m)\t_3]\t_1+
$$
$$
+2m\t_1\t_2.
$$
We now examine the non-negativity or positivity of each of the four terms appearing in the above formula (using the inequalities \eqref{E:ineqtau}):
\begin{itemize}
\item $n(mt-\t_2)[(n+1)t-2\t_3]\geq 0$ since $t\geq \t_2, \t_3$. 
\item $m(m-1)(nt+\t_2-\t_3)t\geq 0$ since $t\geq \t_3$.
\item $[n(2m+n+1)t-2(n+m)\t_3]\t_1>0$ since $n(2m+n+1)\geq 2(n+m)$ and $t>\t_3$.
\item $2m\t_1\t_2\geq 0$ since $\t_1,\t_2\geq 0$. 
\end{itemize}
It follows that $D(j_1j_2,j_3)$ is positive also in this case.
\end{enumerate}
\end{proof}

\begin{proof}[Proof of Proposition \ref{P:vertpoly} for $r=3$]
In this case $\Pi(t,s,3)$ is, by Lemma \ref{elemfunc}, a $2$-dimensional polytope in $\bbR^2$, i.e. a polygon. 
Observe that the inequalities \eqref{E:inc-coord} for $k=2,3$ and \eqref{E:inc-slope} for $k=3$ imply that the points $P^j:=\Sigma_{\leq 3}(\un \lambda^j)\in \bbR^2$, for $2\leq j\leq s$, have strictly increasing positive coordinates and non-decreasing slopes.
From this it follows easily that the 
two points $\{P^1=0, P^s\}$ \color{black}are always vertices of $\Pi(t,s,3)$. 

It remains to show that if 
$s\geq 3$ and $2\leq j\leq s-1$ \color{black}then the point $P^j$ is a vertex of $\Pi(t,s,3)$ if and only if $j$ is a $t$-break. 

If $j$ is not a $t$-break, \color{black} we have that
$\left\lceil \frac{t}{j-1} \right\rceil=\left\lfloor \frac{t}{j+1}\right\rfloor +1$,
hence Lemma \ref{L:vanishdet} implies that the slope of the vector $\overrightarrow{P^{j-1}P^j}$ is equal to the slope of the vector $\overrightarrow{P^jP^{j+1}}$. 
Since the 
points $\{P^{j-1}, P^j, P^{j+1}\}$ have increasing coordinates, we 
get that  $P^j$ lies in the segment $\left[ P^{j-1}, P^{j+1}\right]$, hence it is not a vertex of $\Pi(t,s,3)$.

Assume now that $j$ is a $t$-break. Since both coordinates of $P^j$ are increasing with $j$, to see that $P^j$ is a vertex of $\Pi(t,s,3)$ it is enough to show that for any $h<j<k$ the slope of the vector $\overrightarrow{P^hP^j}$ is strictly less than the slope
of $\overrightarrow{P^jP^k}$. This follows by Lemma \ref{L:vanishdet} and
$$\left\lceil \frac{t}{h} \right\rceil>\left\lfloor \frac{t}{k}\right\rfloor +1$$
which holds because $j$ is a $t$-break. 
\end{proof}

 


The proof of Proposition \ref{P:vertpoly} in the case $r\geq 4$ is based on the following result.


\begin{lemma}\label{L:face-r3}
Consider $l\geq 3$ natural numbers $1\leq j_1< \ldots <j_l\leq s$ such that 
$$\left\lceil \frac{t}{j_1} \right\rceil=\left\lfloor \frac{t}{j_l} \right\rfloor+1.$$ 
Then the points $\{\Sigma_{\leq  4}(\un \lambda^{j_1}), \ldots, \Sigma_{\leq 4}(\un \lambda^{j_l})\}$ are the vertices of a polygon 
in an affine plane of $\bbR^3$.
\end{lemma}
\begin{proof}
Set $\alpha=\lfloor \frac{t}{j_l}\rfloor$. By the assumptions on the $j_k$'s we have expressions as in \eqref{E:expr} of the form
\begin{equation}\label{E:manydiv}
t=\alpha j_k+\r_k \text{ such that } \r_k\in 
\begin{cases}
(0,j_1] & \text{if } k=1, \\
(0,j_k) & \text{if } 1<k<l, \\
[0,j_j) & \text{if } k=l.
\end{cases}
\end{equation}

Set now $Q^k:=\Sigma_{\leq 4}(\un \lambda^{j_k})=(\sigma_2(\un \lambda^{j_k}), \sigma_3(\un \lambda^{j_k}), \sigma_4(\un \lambda^{j_k}))\in \bbR^3$ and let us compute the vector $\overrightarrow{Q^{h}Q^{k}}=Q^{k}-Q^{h}=(\sigma_2(\un \lambda^{j_k})-\sigma_2(\un \lambda^{j_h}) , \sigma_3(\un \lambda^{j_k})-\sigma_3(\un \lambda^{j_h}), \sigma_4(\un \lambda^{j_k})-\sigma_4(\un \lambda^{j_h}))\in \bbR^3$ for $1\leq h<k\leq l$. Using formula \eqref{E:p-lambda} and setting $\t_k:=\r_k(\a_k+1)$ as usual, we get that 
$$
\begin{sis}
p_n(\underline{\lambda}^{j_k})-p_n(\underline{\lambda}^{j_h}) &= f_n(\alpha)(\tau_k-\tau_h) \: \text{ with } f_n(\alpha)=(\alpha+1)^{n-1}-\alpha^{n-1}\: \text{ for any } n\geq 2,\\
p_2(\underline{\lambda}^{j_k})^2-p_2(\underline{\lambda}^{j_h})^2 &=(2t\alpha+\tau_k+\tau_h)(\tau_k-\tau_h).
\end{sis}
$$
From these formulas and Newton's identities \eqref{E:Newtrel}, we obtain 
\begin{equation}\label{E:sig-diff}
\begin{sis}
\sigma_2(\underline{\lambda}^{j_k})-\sigma_2(\underline{\lambda}^{j_h}) &={1\over 2}(\tau_h-\tau_k),\\
\sigma_3(\underline{\lambda}^{j_k})-\sigma_3(\underline{\lambda}^{j_h}) &= 
             \left[{1\over 2}t-{1\over 3}f_3(\alpha)\right](\tau_h-\tau_k), \\
\sigma_4(\underline{\lambda}^{j_k})-\sigma_4(\underline{\lambda}^{j_h}) &=
\left[{1\over 4}t^2-{1\over 3}tf_3(\alpha)-{1\over 8}(2t\alpha+\tau_k+\tau_h)+{1\over 4}f_4(\alpha)\right](\tau_h-\tau_k) .
\end{sis}
\end{equation}
Hence we compute 
\begin{equation}\label{E:difvec}
\overrightarrow{Q^{h}Q^{k}}=Q^{k}-Q^{h}=(\t_h-\t_k) \left(a, b, c-\frac{\t_h+\t_k}{8} \right) \: \text{ for any }  1\leq h<k\leq l,
\end{equation}
where $a={1\over 2}$, $ b={1\over 2}t-{1\over 3}f_3(\alpha)$, $c={1\over 4}t^2-{1\over 3}tf_3(\alpha)-{1\over 4}t\alpha+ {1\over 4}f_4(\alpha)$.
From the above expression, we deduce that the points $Q^k$ lie in an affine plane of $\bbR^3$ which is a translate of the linear plane $\{bx-ay=0\}$, where we denote by $(x,y,z)$ the three coordinates of $\bbR^3$. If we take $\{x,z\}$ as the coordinates of this plane, then the slope of the vector  $\overrightarrow{Q^{h}Q^{k}}$ is equal to 
$$\text{slope}(\overrightarrow{Q^{h}Q^{k}})=\frac{1}{a}\left(c-\frac{\t_h+\t_k}{8}\right).$$
Since $\tau_1>\tau_2>\ldots>\tau_l$, we 
find that the slope of $\overrightarrow{Q^{h}Q^{k}}$ is an increasing function of $k\in [h+1,l]$. This is enough to conclude that the 
$Q^k$'s are the vertices of a polytope in their affine span. 
\end{proof}


\begin{proof}[Proof of Proposition \ref{P:vertpoly} for $r\geq 4$]
We have to show that all the points $Q_r^j:=\Sigma_{\leq r}(\un \lambda^j)$ with $1\leq j \leq s$ are vertices of the $(r-1)$-dimensional polytope $\Pi(t,s,r)$ in $\bbR^{r-1}$. Since we know that the convex hull of 
$\{Q_r^j\}_{1\leq j \leq s}$  is 
$\Pi(t,s,r)$ by Lemma \ref{elemfunc}, it remains to show  that the 
$\{Q_r^j\}_{1\leq j\leq s}$ are convex independent, i.e. no one of them is in the convex hull of the remaining ones. 

Let us first prove this assertion for  $r=4$. Assume by contradiction that there is a point $Q_4^h$ which is in the convex hull of the points $\{Q_4^j\}_{j\neq h}$, i.e.
\begin{equation}\label{E:convrel1}
Q_4^h=\sum_{j\neq h} \lambda_j Q_4^j \text{ with } \lambda_j\geq 0 \text{ and } \sum_{j\neq h} \lambda_j = 1.
\end{equation}
Consider the projection $\pi:\bbR^ 3\to \bbR^2$ onto the first two coordinates, which clearly sends $Q_4^j$ into $Q_3^j$. It follows that $Q_3^h$ is in the convex hull of the points $\{Q_3^j\}_{j\neq h}$ and more precisely that 
$$Q_3^h=\sum_{j\neq h} \lambda_j Q_3^j.$$
From  Proposition \ref{P:vertpoly} in the case $r=3$ (which was proved above), it follows that $h$ is different from $1$ and $s$ and it is not a $t$-break point. Denote by $h^-$ the biggest $t$-break point smaller than $h$ (note that $h^-\geq 2$ since $2$ is always a $t$-break point), and denote by $h^+\leq s$ the smallest $t$-break point greater than $h$ or $s$ if there is not such a break point. From the proof of Proposition \ref{P:vertpoly} it follows that $Q_3^h$ lies in the segment $[Q_3^{h^-}, Q_3^{h^+}]$ which is en edge of the polygon $\Pi(t,s,3)$, and moreover that the unique points $Q_3^j$ that lie on the edge $[Q_3^{h^-}, Q_3^{h^+}]$ are those belonging to the set $\{Q_3^j\}_{h^-\leq j \leq h^+}$.
From this it follows that $\lambda_j=0$ if $j$ does not belong to the interval $[h^-, h^+]$, and hence we can rewrite \eqref{E:convrel1} as
\begin{equation}\label{E:convrel2}
Q_4^h=\sum_{h\neq j\in [h^-,h^+]} \lambda_j Q_4^j \text{ with } \lambda_j\geq 0 \text{ and } \sum_{h\neq j\in [h^-,h^+]} \lambda_j = 1.
\end{equation}
However, by construction, the points $h^-$ and $h^+$ satisfy the equality 
$\left\lceil \frac{t}{h^-} \right\rceil=\left\lfloor \frac{t}{h^+} \right\rfloor+1$.
Hence Lemma \ref{L:face-r3} implies that 
$\{Q_4^{h^-}, \ldots, Q_4^{h^+}\}$ are convex independent and this contradicts \eqref{E:convrel2}.

Consider finally the case $r\geq 5$. The projection $\pi:\bbR^{r-1}\to \bbR^3$ into the first $3$ coordinates sends each $Q_r^j$ into $Q_4^j$. Since the points $\{Q_4^j\}_{1\leq j\leq s}$ are convex independent from what proved above, we infer that also the points $\{Q_r^j\}_{1\leq j\leq s}$ are convex independent.
\end{proof}


Then the proof of Theorem \ref{T:diagcone} can be summarized as follows.
\begin{proof}[Proof of Theorem \ref{T:diagcone}]
The assertion for $r(n)=1$ follows from Example \ref{E:r1} ensuring that if $r(n)=1$, then all the diagonals are proportional, so that $\dim \D_n(C_d)=r(n)=1$.

When instead $r(n)\geq 2$, we can endow the tautological space $R_n(C_d)$ with the basis $\{w_1,\ldots,w_{r(n)}\}$ introduced in Lemma \ref{newbasis}.
Then Corollary \ref{C:secdiag} ensures that the polytope $\Pi(d-n,s(n),r(n))$ is a bounded section of $\D_n(C_d)$, which is thus polyhedral and has dimension $r(n)$ by Proposition \ref{P:vertpoly}.
Finally, Lemma \ref{newbasis} shows that for any $1\leq j\leq s(n)$, the coordinates of the point $\Sigma_{\leq r(n)}(\un \lambda^j)$ correspond to the coordinates of the class of the $j$-th balanced diagonal $\wt \delta_{\un{\lambda^j+1}}$ with respect to the basis $\{w_1,\ldots,w_{r(n)}\}$, and hence the assertion follows from the description of the vertices of $\Pi(d-n,s(n),r(n))$ given by Proposition \ref{P:vertpoly}.
\end{proof}

It would be interesting to determine the combinatorial structure (i.e. the structure of its face lattice) of the polytope $\Pi(t,s,r)$. 
When $r=3$ it is a polygon, whence a cyclic polytope. When $r \ge 4$ there is one case where we can determine it. 

\begin{example}\label{exa:rbigs}




If $4 \le r\leq s\leq t$ 
and  $t$ is a multiple of every element $2\leq j \leq s$,  then it can be shown that $\Pi(t,s,r)$ is combinatorially equivalent to  a cyclic polytope of dimension $r-1$ and with $s$ vertices.

This follows from the ensuing result on convex geometry (the proof of which is omitted): if $3\le r \leq s\leq t$ and if 
 \begin{equation*}
 \un \lambda^j_{\Q}:=(\underbrace{\frac{t}{j}, \cdots, \frac{t}{j}}_{j}, 0, \cdots, 0) \in  \P_{\leq s}(t)\subset (\bbN^s)^{\sum=t}\: \:\text{ for } 1\leq j \leq s,
\end{equation*}
 then   the points $P^j:=(\sigma_2(\un \lambda^j_{\Q}), \ldots, \sigma_{r}(\un \lambda^j_{\Q}))$ are the vertices of a  cyclic polytope in $\R^{r-1}$.
\end{example} 

Note however that from Lemma \ref{L:face-r3} it is easy to construct examples of polytopes $\Pi(t,s,r)$ with $r\geq 4$ that are not simplicial, hence 
not combinatorially equivalent to cyclic polytopes. 

\section{Diagonal cone as a face}\label{S:diag-face}

The aim of this section is to study the extremality properties of the diagonal cone inside the (tautological) (pseudo-)effective cone.
The key result  is the following 

\begin{thm}\label{T:eta} 
Let $C$ be a smooth irreducible curve of genus $g \ge 1$. For $1\leq n \leq d-1$ consider the class 
$$\eta_{n,d} 
:=\color{black}\frac{dg}{n} x^n-\theta x^{n-1}.$$ 
Then
\begin{enumerate}[(i)]
\item \label{T:eta1} if $S$ is a $n$-dimensional diagonal in $C_d$, then $[S]\cdot \eta_{n,d}=0$;
\item \label{T:eta2} 
there exists a constant $c_{n,d}\geq 0$, which is furthermore positive if $g\geq 2$, such that for any  $n$-dimensional irreducible subvariety $S$ of $C_d$ which is  not a diagonal we have that 
$$(\eta_{n,d}-c_{n,d}x^n) \cdot [S] \geq 0.$$
In particular $\eta_{n,d}$ is a nef $(d-n)$-cycle for $g \ge 1$ and it is also big for $g \ge 2$.
\end{enumerate}
\end{thm}
\begin{proof}[Proof of Theorem \ref{T:eta}\eqref{T:eta1}]
We will prove the theorem by induction on $d$. We will distinguish two cases, the first of which gives also the base of the induction $d=2$. 

\un{CASE I:} $n=d-1\geq 1$. 

In this case, there is only one $n$-dimensional diagonal in $C_{n+1}$, namely the big diagonal $\Delta_{(2,1,\ldots,1)}$.   Let us compute the intersection of $[\Delta_{(2,1,\ldots,1)}]$ with $\eta_{n, n+1}$:
\begin{eqnarray}
& [\Delta_{(2,1,\ldots,1)}] \cdot \eta_{n,n+1}=  2(n-1)![(n+g)x-\theta] \cdot \left[\frac{(n+1)g}{n}x^n-x^{n-1}\theta\right] \quad \text{ using Example }\ref{E:r1}\nonumber\\
&= 2(n-1)!\left[\frac{(n+g)(n+1)g}{n}-(n+g)g-\frac{(n+1)g^2}{n}+g(g-1) \right]=0 \quad \text{using Lemma } \ref{L:inters}.\nonumber
\end{eqnarray}

This concludes the proof in this case.

\un{CASE II:} $d\geq n+2\geq 3$. 

We 
use the push operator $A$ of \S \ref{S:pushpull}. Using Fact \ref{F:pushpull}\eqref{F:pushpull2}, we compute (for any $1\leq m \leq e-1$):
\begin{equation}\label{A-geneta}
A(\eta_{m,e})=\frac{eg}{m}(e+1-m) x^m-\left[(e+1-(m-1)-2)x^{m-1}\theta+gx^m \right]= (e-m)\eta_{m,e+1}.
\end{equation}
In particular,  by iterating the above formula, we get 
\begin{equation}\label{A-eta}
A^{d-n-1}(\eta_{n,n+1})=(d-n-1)!\eta_{n,d}.
\end{equation}
Consider now an $n$-dimensional diagonal $\Delta_{\un a}$.  Using \eqref{A-eta} and Fact \ref{F:pushpull}\eqref{F:pushpull1}, we get that 
\begin{equation}\label{eta-diag}
\eta_{n,d}\cdot [\Delta_{\un a}]=\frac{1}{(d-n-1)!}A^{d-n-1}(\eta_{n,n+1}) \cdot [\Delta_{\un a}]= \frac{1}{(d-n-1)!} \eta_{n, n+1}\cdot B^{d-n-1}([\Delta_{\un a}]).
\end{equation}
Now, observe that  Remark \ref{R:eff-pp} implies that $B(\Delta_{\un a})$ is a non-negative linear combination of $n$-dimensional diagonals  $\{\Delta_{\un{a-\epsilon_i}}\}_{a_i\geq 2}$ in $C_{d-1}$,  where we denote by $\un{a-\epsilon_i}$ a non-increasing re-arrangement of the $n$-tuple $(a_1,\ldots, a_{i-1}, a_i-1, a_{i+1}, \ldots, a_n)$. By iterating this observation, we get that $B^{d-n-1}([\Delta_{\un a}])$ is a multiple of the big diagonal in $C_{n+1}$. Hence the vanishing of the right hand side of \eqref{eta-diag} follows from Case I, and we are done.
\end{proof}

Let us now prove Theorem  \ref{T:eta}\eqref{T:eta2} in the case of divisors, i.e. for $n=d-1$. 

\begin{prop}\label{P:eta-div}
Let $d \ge 2, g \ge 1$ and consider an irreducible divisor $S$ of $C_d$. Then
\begin{enumerate}[(i)]
\item\label{P:eta-diva}  $\eta_{d-1,d}\cdot [S] \geq 0$.
\item \label{P:eta-divb} If $S$ is different from the big diagonal $\Delta_{(2,1,\ldots,1)}$ we have 
$$\left(\eta_{d-1,d}-\frac{g}{d-1}\frac{2(g-1)}{d(d-1)+2(g-1)} x^{d-1}\right)\cdot [S] \geq 0.$$
\end{enumerate}
\end{prop} 
\begin{proof}
Observe  that part \eqref{P:eta-diva} follows from part \eqref{P:eta-divb} using that  $\eta_{d-1,d}\cdot [\Delta_{(2,1,\ldots,1)}] =0$
by Theorem \ref{T:eta}\eqref{T:eta1} and that $x^{d-1}$ is a nef class.

We will prove part \eqref{P:eta-divb} by induction on $d$. Set $\displaystyle c_d:=\frac{g}{d-1}\frac{2(g-1)}{d(d-1)+2(g-1)}$.

For $d=2$, if $S$ is irreducible and different from $\Delta_{(2)}$, then the curves $S$ and $\Delta_{(2)}$ will intersect in a finite number of points and hence $[S]\cdot [\Delta_{(2)}]\geq 0$. The result follows since 
$$\eta_{1,2}-c_2x=(2gx-\theta)-(g-1)x=(g+1)x-\theta=\frac{[\Delta_{(2)}]}{2}.$$

We now assume that \eqref{P:eta-divb} is true for $d-1$ and we prove \eqref{P:eta-divb} for $d\geq 3$. 
Using an idea from the proof \cite[Thm. 3]{Kou}, let $T$ be the  effective cycle associated to the scheme-theoretic intersection $S\cap \Delta_{(2,1,\ldots,1)}$.
Since $S$ and $\Delta_{(2,1,\ldots,1)}$ are distinct (Cartier) divisors on $C_d$, 
$S\cap \Delta_{(2,1,\ldots,1)}$  has pure codimension two and hence 
$[T]=[S]\cdot [\Delta_{(2,1,\ldots,1)}]\in N_{d-2}(C_d)$ (see \cite[Example 2.6.5]{Fult}).
Then $B(T)$ is an effective divisor in $C_{d-1}$ (that may contain the  big diagonal of $C_{d-1}$ as a component) and by induction we get in particular that $\eta_{d-2,d-1}$ is nef, hence
\begin{equation}\label{E:indcd-1}
{(d-1)g \over d-2} \;  [B(T)]\cdot x^{d-2}\geq [B(T)]\cdot  x^{d-3}\theta  \, .  
\end{equation}
Using Fact \ref{F:pushpull}, we compute 
\begin{equation}\label{E:interBT}
\begin{sis}
& [B(T)]\cdot x^{d-2}= [T]\cdot A(x^{d-2}) =2 [T]\cdot x^{d-2} = 2 [S] \cdot  [\Delta_{(2, 1, \ldots, 1)}] \cdot x^{d-2}, \\
& [B(T)]\cdot x^{d-3}\theta = [T] \cdot A(x^{d-3}\theta)= [T]\cdot (gx^{d-2}+x^{d-3}\theta ) = [S]\cdot [\Delta_{(2,1,\ldots, 1)}] \cdot  
(gx^{d-2}+x^{d-3}\theta).
\end{sis}
\end{equation}
Using that $[\Delta_{(2,1,\ldots,1)}]=2(d-2)![(d+g-1)x-\theta]$ by Example \ref{E:r1} and the relation 
$$g(g-1) x^{d-1}-2(g-1) \theta x^{d-2} + \theta^2 x^{d-3} = 0 \ \mbox{in} \ R^{d-1}(C_d)$$ 
which can be obtained either by Proposition \ref{basetaut} and Lemma \ref{L:inters} or by applying the operator $A$ to the relation $g(g-1)x^{d-1}-x^{d-3}\theta^2=0$ in $R^{d-1}(C_{d-1})$ (which holds by Lemma \ref{L:inters}),
we get 
\begin{equation}\label{E:expSD}
\begin{sis}
& \frac{1}{(d-2)!} [S] \cdot  [\Delta_{(2, 1, \ldots, 1)}] \cdot x^{d-2} =2(d+g-1)[S]\cdot x^{d-1}-2[S]\cdot x^{d-2}\theta,  \\
&  \frac{1}{2(d-2)!} [S]\cdot [\Delta_{(2,1,\ldots, 1)}] \cdot  (gx^{d-2}+x^{d-3}\theta)= \\ & = g(d+g-1)[S]\cdot x^{d-1}+(d-1)[S]\cdot x^{d-2}\theta-[S]\cdot x^{d-3} \theta^2= \\
& = g(d+2g-2) [S]\cdot x^{d-1} +(d-2g+1) [S]\cdot x^{d-1}\theta. 
\end{sis}
\end{equation}
Using \eqref{E:interBT} and \eqref{E:expSD} and setting $a=[S]\cdot x^{d-1}$ (so that $a>0$) and $b=[S]\cdot x^{d-2} \theta$, the inequality \eqref{E:indcd-1} becomes 
\[
\begin{aligned}
{(d-1)g\over d-2}\; ( 2(d+g-1)a-2b )  \; \geq  \; &    g (d+2(g-1))a + (d-2g+1)b \Longleftrightarrow   \\
2g(d-1)(d+g-1)a -2g(d-1)b \; \geq \;  &   g(d-2)(d+2(g-1))a+(d-2)(d-2g+1)b   \Longleftrightarrow \\
g [ 2(d-1)(d+g-1)-(d-2) (d+ & 2(g-1)) ]a \; \geq \;    [(d-2)(d-2g+1)+2g(d-1)]b    \Longleftrightarrow \\
g(d^2+2g-2)a \; \geq \;   (d^2+2g-2-d)b & \; \;  \Longleftrightarrow \;\;
{g(d^2+2g-2) \over d^2+2g-2-d} \; \geq \;  {b\over a} \, .
\end{aligned}
\]
We then have 
\[
\begin{aligned}
{dg\over d-1}-{b\over a}\geq & {dg\over d-1} -{g(d^2+2g-2)\over d^2+2g-2-d } = 
g\left({d\over d-1}-{d^2+2g-2\over d^2+2g-2-d }\right) =
\end{aligned}
\]
\[
\begin{aligned}
= & g\left[1+{1\over d-1}-\left(1+{d\over d^2+2g-2-d}\right)\right]=g\left({1\over d-1}-{d\over d(d-1)+2g-2}\right)\\
= &   {2g(g-1) \over (d-1)(d(d-1)+2g-2)}=c_d.  
\end{aligned}
\]

The above inequality is equivalent to 
$$\eta_{d-1,d}\cdot [S]= \frac{dg}{d-1}  x^{d-1}\cdot [S]-x^{d-2}\theta\cdot [S] \geq c_d x^{d-1}\cdot [S],
$$
and this concludes the proof.
\end{proof}


\begin{remark}\label{dleqg}
When $d\leq g$, there is a more geometric (and easier) proof of a weaker form of the above proposition: if $S$ is an irreducible divisor different from the big diagonal $\Delta_{(2,1,\ldots,1)}$ then 
$$\left(\eta_{d-1,d}-{g-d+1\over d-1} x^{d-1}\right)\cdot [S] \geq 0.$$
Indeed, we can write $\eta_{d-1,d}={g-d+1\over d-1}x^{d-1}+((g+1)x-\theta)x^{d-2}$. The second summand is, up to a multiple,  the class of the members of a $(d-2)$-dimensional covering family of the diagonal $\Delta_{(2,1,\ldots,1)}$ by curves; namely, the members of the family are  the curves $\Gamma_{D_{d-2}} = \{ 2p+D_{d-2} \}$ as $D_{d-2}$ varies in  $C_{d-2}$. To calculate the class of these curves, we take the embedding $i=i_{D_{d-2}}: C_2\to C_d$ defined by  $i (p+q):=p+q+D_{d-2}$. We have $i_* [C_2]= x^{d-2}$ and  $i^*x=x$ (for geometric reasons) and $i^*\theta =\theta $ (because $i$ commutes with the Abel-Jacobi maps). 
Therefore we get
\begin{eqnarray*}
[\Gamma_{D_{d-2}}] & =  & i_*([\Delta_{(2)}])=2 i_*((g+1)x-\theta)=2  i_*i^* ((g+1)x-\theta)=2 i_*[C_2] \cdot ((g+1)x-\theta)=  \\
& = & 2 ((g+1)x-\theta)x^{d-2}. 
\end{eqnarray*}
Since  $S \neq \Delta_{(2,1,\ldots,1)}$
we have that $[S]\cdot ((g+1)x-\theta)x^{d-2} \geq 0$, and the result follows.
\end{remark}

Before proving Theorem \ref{T:eta}\eqref{T:eta2} in the general case, let us introduce the following important concept.

\begin{defi}\label{D:lenght}
Given an $n$-dimensional irreducible subvariety $S$ in $C_d$, the \emph{length of $S$}, denoted by $l(S)$, is the length of the support of the general point  $D$ of $S$, i.e. the number of points in the support of the divisor  $D$.
\end{defi}

We note the following straightforward properties of the length function.

\begin{remark}\label{R:lenght}
Given an $n$-dimensional irreducible subvariety $S$ in $C_d$, the length of $S$ is such that  $n\leq l(S)\leq d$.  
Moreover, the length of $S$ can be characterized as the smallest integer $m$ such that $S$ is contained in a diagonal of dimension $m$. Furthermore, there is a unique diagonal $\Delta_{\un a(S)}$ of dimension $l(S)$ containing $S$: the length-$l(S)$ partition $\un a(S)$ of $d$ is specified by the multiplicity of the $l(S)$ points appearing in the general element of $S$. 

Two extremal cases are: $l(S)=\dim S=n$ which happens if and only if $S$ is a diagonal, and $l(S)=\dim C_d=d$ which happens if and only if $S$ is not contained in a proper diagonal (i.e. different from $C_d$).
\end{remark}

The crucial geometric ingredient  in the proof of the general case of Theorem \ref{T:eta}\eqref{T:eta2}   is the following lemma, describing the behavior of the pull map $B$ applied to an irreducible subvariety of $C_d$. We will use the notation introduced in Remark \ref{R:eff-pp}.

\begin{lemma}\label{L:Bcycle}
Let $1 \le n \le d-1$. Let $S$ be an irreducible subvariety of $C_{d}$ of dimension $n$ and length $m:=l(S)$. 
\begin{enumerate}[(a)]
\item \label{L:Bcycle1} For any irreducible component $V$ of  $B(S)$ we have that $l(S)-1\leq l(V)\leq l(S)$.
\item \label{L:Bcycle2} Assume that $n<m:=l(S) < d$ (or equivalently that $S$ is not a diagonal and $S \subset \Delta_{(2,1,\ldots,1)}$) and write $B^{d-m}(S)=W+\sum_{\un b}\nu_{\un b}\Delta_{\un b}$, where $W$ is an effective $n$-dimensional cycle in $C_m$ 
with no component being a diagonal
and the summation runs over all partitions $\un b$ of $m$ in $n$ parts (so that the $\{\Delta_{\un b}\}$ are exactly the $n$-dimensional diagonals of $C_m$).  Then there exists a constant $K_{n,m,d}$ (independent of $S$) such that 
$$K_{n,m,d}\cdot ([W]\cdot x^n)\geq \max\{\nu_{\un b}\}.$$
\end{enumerate}
\end{lemma}
Consider the diagram 
$$\xymatrix{
 C^m \ar[d]_{\mu_m} & C^{d} \ar[d]^{\mu_d}\ar[l]_{\hskip 0.3cm \pi}\\
C_m & C_{d}
}
$$
where $\mu_d : C^d \to C_d$ is the addition  map sending $(p_1,\ldots,p_d)$ into $\sum_{i=1}^d p_i$ (and similarly for $\mu_m$), and $\pi:C^d\to C^m$ is the projection onto the first $m$ factors.
We get a map 
\begin{equation}\label{map-Bwt}
\wt B=(\mu_m)_*\circ (\pi_*)\circ (\mu_d)^*: N_n(C_d) \longrightarrow N_n(C_m)
\end{equation}
which is easily seen to be equal to $m!B^{d-m}$. Hence we can prove the lemma with $B^{d-m}$ replaced by $\wt B$. 
Before doing that, we need the following digression into  the behavior of the diagonals under the addition map $\mu_d$ (and hence $\mu_m$) and the projection map $\pi$.  

\begin{nota}[\textbf{Diagonals in the ordinary product}]\label{digr-diag}

A \emph{partition} ${\un a}=(a_1,\ldots,a_m)$  of $d$ of length $m$ (which means that $a_1\geq \ldots \geq a_m\geq 1$ and $\sum_{i=1}^m a_i=d$) can be represented by the \emph{Young diagram} $Y(\un a)$ with $m$ rows and whose $i$-th row has exactly $a_i$ boxes. The association $\un a\mapsto Y(\un a)$ establishes a bijective correspondence between partitions of $d$ and Young diagrams with $d$ boxes. We will denote by $r(\un a)=r(Y(\un a))$ the number of parts of  $\un a$, which is equal to the number of rows of $Y(\un a)$. Given two partitions $\un a$ and $\un b$ of $d$, we say that $\un a\geq \un b$ (or that $Y(\un a)\geq Y(\un b)$) if $Y(\un b)$ can be obtained by $Y(\un a)$ via an iteration of the following operation: join two rows of $Y(\un a)$  to form a unique row and then rearrange the rows in  non-increasing cardinality. We say that $\un a >\un b$ if $\un a\geq \un b$ but $\un a\neq \un b$.  It is easy to see that $\un a\geq \un b$ if and only if 
$\Delta_{\un a}\supseteq \Delta_{\un b}$, where $\Delta_{\un a}\subset C_d$ is the $r(\un a)$-dimensional diagonal associated to $\un a$ (see \S \eqref{E:diag}).

A \emph{tableaux} $\lambda$ on $d$ numbers consists of a Young diagram $Y(\un a)$ with $d$ boxes (called the shape of $\lambda$) and an equivalence class of fillings of $Y(\un a)$ with the numbers  $\{1, \ldots, d\}$ without repetitions,
where two fillings of $Y(\un a)$  are equivalent  if one can be obtained from the other  by iterating the following two operations: reordering the numbers on a given row and exchanging two rows of the same cardinality\footnote{Note that a tableaux in our sense is not a Young tableaux (in the usual sense) because we do not require any non-decreasing property of the fillings along the  columns.}.
We will  denote by $\un{\lambda}$ the partition corresponding to the shape of $\lambda$, by $r(\lambda)$ the number of rows of the shape of $\lambda$ (so that $r(\lambda)=r(\un{\lambda})$) and we will denote by $T(\un a)$ the set of all tableaux $\lambda$ such that $\un{\lambda}=\un a$. 
The symmetric group $\bS_d$ acts on the set of all tableaux on $d$ numbers by permuting the filling: given a tableaux $\lambda$ and an element $\sigma \in \bS_d$ we will denote by $\lambda^{\sigma}$ the resulting tableaux.
Note that the action of the symmetric group $\bS_d$  preserves the shape of a tableaux, so that it induces an action on $T(\un a)$ which is easily seen to be transitive. Given two tableaux $\lambda$ and $\mu$ on $d$ numbers, we say that $\lambda\geq \mu$ if $\mu$ can be obtained from $\lambda$ via an iteration of the following operation: join two rows of $\lambda$ and their fillings to form a unique row with a specified filling and then rearrange the rows in  non-increasing cardinality. As above, we say that $\lambda>\mu$ if $\lambda\geq \mu$ but $\lambda\neq\mu$. The poset of tableaux on $d$ numbers with respect to the above partial order $\geq $ is graded with respect to the function $\lambda\mapsto r(\lambda)$ (in particular, if $\lambda\geq \mu$ then $r(\lambda)\geq r(\mu)$ with strict inequality unless $\lambda=\mu$)
and it admits meets: given two tableaux $\lambda$ and $\mu$, the meet of $\lambda$ and $\mu$ (i.e. the greatest lower bound) is the tableaux $\lambda\wedge \mu$ which is obtained by merging all the rows of $\lambda$  and  of $\mu$ (and their fillings) that have some number in common.  
Note that we have the following compatibility properties among the poset structures on the set of partitions of $d$ and on the set of tableaux on $d$ numbers: 
\begin{itemize}
\item  if $\lambda \geq \mu$ then $\un \lambda\geq \un \mu$;
\item if $\un a\geq  \un b$ and $\lambda$ is such that $\un \lambda=\un a$ then there exists a (not necessarily unique) $\mu$ such that $\un \mu=\un b$ and $\lambda\geq \mu$;
\item if $\un a\geq  \un b$ and $\mu$ is such that $\un \mu=\un b$ then there exists a (not necessarily unique)  $\lambda$ such that $\un \lambda=\un a$ and $\lambda\geq \mu$. 
\end{itemize}

To any tableaux $\lambda$ on $d$ numbers, we can associate an \emph{ordered diagonal} $\Delta(\lambda)$ in the ordinary product $C^d$ of dimension $r(\lambda)$ consisting of all the elements of $C^d$  such that the coordinates corresponding to the numbers lying on the same row of $\lambda$ are equal. Note that $\Delta(\lambda)\cong C^{r(\lambda)}$. The action of the symmetric group $\bS_d$ induces a permutation of the ordered diagonals which is compatible with the action of $\bS_d$ on the tableaux on $d$ numbers: in symbols, for any $\sigma \in \bS_d$ we have that $\sigma_*:\Delta(\lambda)\stackrel{\cong}{\longrightarrow} \Delta(\lambda^{\sigma})$, where $\sigma_*:C^d \stackrel{\cong}{\longrightarrow} C^d$ is the permutation isomorphism 
induced by $\sigma$.  The poset structure on the set of tableaux  introduced above coincides with the poset structure on the ordered diagonals coming from the inclusion: $\lambda\geq \mu$ if and only if  $\Delta(\lambda) \supseteq \Delta(\mu)$ and $\Delta(\lambda)\cap \Delta(\mu)=\Delta(\lambda\wedge \mu)$. 

With these notations, we can describe the relationship between diagonals in the symmetric product $C_d$ and ordered diagonals in the 
ordinary \color{black}product $C^d$ via the addition map $\mu_d:C^d\to C_d$. Indeed, for any  tableaux $\lambda$ on $d$ numbers and any partition $\un a$ of $d$, we have 
\begin{equation}\label{pullb-diag}
\begin{sis}
& \mu_d(\Delta(\lambda))=\Delta_{\un \lambda}, \\
& \mu_d^{-1}(\Delta_{\un a})=\bigcup_{\lambda\in T(\un a)} \Delta(\lambda).
\end{sis}
\end{equation}
Consider now the projection map $\pi:C^d\to C^m$ onto the first $m$ factors. 
Given a tableaux $\lambda$ on $d$ numbers, we will denote by $\pi(\lambda)$ the tableaux on $m$ numbers obtained by first erasing the boxes labeled by $\{m+1,\ldots, d\}$ and then rearranging the rows in non increasing cardinality. If we denote by $\bS_m\times \bS_{d-m}\subset \bS_d$ the subgroup consisting of the permutations that preserve the partition $\{1,\ldots, d\}=\{1,\ldots, m\} \coprod \{m+1,\ldots, d\}$, then  the surjective map $\lambda\mapsto \pi(\lambda)$ is $\bS_m$-equivariant and $\bS_{d-m}$-invariant. The above map admits a section: given a tableaux $\nu$ on $m$ letters, let $\pi^{-1}(\nu)$ be the tableaux on $d$ letters obtained from $\nu$ by inserting $d-m$ rows of cardinality one filled with the numbers $\{m+1,\ldots, d\}$. Clearly the maps $\pi$ and $\pi^{-1}$ preserve the partial order $\geq $, namely if 
$\lambda\geq \mu$ then $\pi(\lambda)\geq \pi(\mu)$ and $\pi^{-1}(\lambda)\geq \pi^{-1}(\mu)$.
The  ordered diagonals are preserved under  the projection map $\pi$; more precisely, we have  the following:
\begin{equation}\label{proj-diag}
\begin{sis}
& \pi(\Delta(\lambda)))=\Delta(\pi(\lambda)),\\
& \pi^{-1}(\Delta(\nu))=\Delta(\pi^{-1}(\nu)).
\end{sis}
\end{equation} 
Note the map $\pi_{| \Delta(\lambda)}: \Delta(\lambda)\twoheadrightarrow \Delta(\pi(\lambda))$ can be identified, 
after the identifications $\Delta(\lambda)\cong C^{r(\lambda)}$ and $\Delta(\pi(\lambda))\cong C^{r(\pi(\lambda))}$, with some projection map $C^{r(\lambda)}\to C^{r(\pi(\lambda))}$. 
\end{nota}

\begin{proof}[Proof of Lemma \ref{L:Bcycle}]
Part \eqref{L:Bcycle1} is immediate from the geometric description of $B(S)$  given in Remark \ref{R:eff-pp}\eqref{R:eff-ppB}.

Part \eqref{L:Bcycle2}. As explained above, it is enough to prove the result for  the map $\wt B$ of \eqref{map-Bwt}. 

Let us first describe the $n$-dimensional cycle $\wt B([S])$. According to Remark \ref{R:lenght}, there exists a unique partition $\un a(S)$ of $d$ of length $m=l(S)$ such  $S\subset \Delta_{\un a(S)}$.  This implies that the $n$-dimensional cycle $\wt B([S])$ will be supported  on the subvariety $\mu_m(\pi(\mu_d^{-1}(\Delta_{\un a(S)})))$. Let us describe more carefully this inclusion, using the notations and results of \S \ref{digr-diag}. 

First of all, using the inclusion $S\subset \Delta_{\un a(S)}$ and \eqref{pullb-diag}, we can write 
\begin{equation}\label{E:pullb-S}
\mu_d^{*}[S]=[\mu_d^{-1}(S)]=\sum_{\lambda\in T({\un a(S)})}  [S_{\lambda}],
\end{equation}
where $S_{\lambda}$ is the union of all the irreducible components of $\mu_d^{*}[S]$ (necessarily of dimension $n$) contained in   $\Delta(\lambda)$. Note that the above description \eqref{E:pullb-S} is well-defined because there cannot be irreducible components $V$ of $\mu_d^{-1}(S)$ contained in two ordered diagonals: indeed if there exist $\lambda\neq \mu\in T(\un a(S))$ such that $V\subseteq \Delta(\lambda)\cap \Delta(\mu)=\Delta(\lambda\wedge \mu)\subsetneq \Delta(\lambda), \Delta(\mu)$ then $S=\mu_d(V)\subseteq\mu_d(\Delta(\lambda\wedge \mu))=\Delta_{\un{\lambda\wedge \mu}}$ and $\Delta_{\un{\lambda\wedge \mu}}\subsetneq\Delta_{\un a(S)}$, contradicting the minimality of $\Delta_{\un a(S)}$. 
 
Observe moreover that since the map $\mu_d$ is $\bS_d$-invariant,  the permutation isomorphism $\sigma_*:C^d\stackrel{\cong}{\longrightarrow} C^d$ induced by any element $\sigma\in \bS_d$ will send $S_{\lambda}$ isomorphically onto $S_{\sigma(\lambda)}$. This defines an action of $\bS_d$ on the set $\{S_{\lambda}\}$ which is transitive since the action of $\bS_d$ on $T(\un a(S))$ is transitive.  

Next, using \eqref{pullb-diag} and \eqref{proj-diag}, we get that the push-forward   $(\mu_m)_*\pi_*([S_{\lambda}])$ is the class of an $n$-dimensional cycle with support contained in the diagonal $\mu_m((\pi(\Delta(\lambda))))=\Delta_{\un{\pi(\lambda)}}\subset C_m$, for any $S_{\lambda}$ appearing in \eqref{E:pullb-S}. We claim that the following is true:
\begin{equation}\label{E:lenght-pS} l(V)=r(\un{\pi(\lambda)}) \quad \text{ for any irreducible component }V \text{ of }(\mu_m)_*\pi_*([S_{\lambda}]).
\end{equation}

Indeed, since $V$ is contained in $\Delta_{\un{\pi(\lambda)}}$, we have that $l(V)\leq  r(\un{\pi(\lambda)})$. By contradiction, 
assume that $l(V)< r(\un{\pi(\lambda)})$.  Then Remark \ref{R:lenght} implies that there exists a partition $\un c< \un{\pi(\lambda)}$ of $m$ such that $V\subseteq \Delta_{\un c} \subsetneq \Delta_{\un{\pi(\lambda)}}$. 
By definition of $(\mu_m)_*\pi_*([S_{\lambda}])$, there exists an irreducible component $V'$  of the cycle $\pi_*([S_{\lambda}])$
such that $\mu_m(V')=V$. Now, on one hand, since $S_{\lambda}\subset \Delta(\lambda)$, we must have that $V'\subseteq \pi(\Delta(\lambda))=\Delta(\pi(\lambda))$. On the other hand, since $\mu_m(V')=V\subseteq \Delta_{\un c}$, we must have by \eqref{pullb-diag} that $V'\subseteq \Delta(\xi)$ for some $\xi\in T(\un c)$. 
By combining these two properties, we get that $V'\subseteq \Delta(\pi(\lambda))\cap \Delta(\xi)=\Delta(\nu:=\pi(\lambda)\wedge \xi)\subsetneq \Delta(\pi(\lambda))$.  
By definition of $\pi_*([S_{\lambda}])$, there exists an irreducible component $V''$ of $S_{\lambda}$ such that $\pi(V'')=V'$. 
Since $S_{\lambda}\subset \Delta(\lambda)$ and $V'\subset \Delta(\nu)$, we have that $V''\subseteq \Delta(\lambda)\cap \pi^{-1}(\Delta(\nu))=\Delta(\mu:=\lambda\wedge \pi^{-1}(\mu))$. Moreover, since $\Delta(\nu)\subsetneq \Delta(\pi(\lambda))$, we have that $\Delta(\mu)\subsetneq \Delta(\lambda)$. Now taking the image via $\mu_d$, we get that $S=\mu_d(V'')\subseteq \Delta_{\un \mu}\subsetneq \Delta_{\un \lambda}=\Delta_{\un a(S)}$ and this contradicts the minimality of $\un a(S)$. This proves \eqref{E:lenght-pS}. 

Fix now a  tableaux $\lambda_0$, with shape $\underline{a}(S)$, which contains in the first column (which has size $m$)  all the numbers $1,\ldots, m$. The tableaux $\lambda _0$ has the property that $r(\pi(\lambda_0))=r(\lambda_0)=m$ and the map $\pi_{|_{\Delta(\lambda_0)}} : \Delta (\lambda_0) \to \Delta (\pi(\lambda_0))$ is an isomorphism  (both spaces are isomorphic to $C^m$).   Therefore $T:=(\mu_m)_*(\pi_*(S_{\lambda_0}))$ is supported in some irreducible components of $W$ of length equal to $m$ by \eqref{E:lenght-pS}.  Moreover,  \eqref{E:lenght-pS} implies that the only 
cycles $S_{\lambda}$ \color{black}from the decomposition \eqref{E:pullb-S} which contribute to the sum $\Sigma_{\underline{b}} \nu_{\underline{b}}\Delta_{\underline{b}}$ are those $S_{\lambda}$ such that $r(\pi(\lambda))=n$. For any such $\lambda$, we must have that $(\mu_m)_*(\pi_*(S_{\lambda}))=\nu_{\lambda}\Delta_{\un{\pi(\lambda)}}$ for some $\nu_{\lambda}\in \bbN$. Since there are a finite number (which can be bounded only in terms of $d$) of these $\lambda$ and $x^n\cdot [W] \geq x^n\cdot [T]$ since $x$ is ample and $W=T+F$ with $F$ effective, the statement will follow from the following 

\vspace{0.1cm}

\un{Claim} : We have that $[T]\cdot x^n\geq \nu_{\lambda}$ for any $\lambda\in T(\un a(S))$ such that $r(\pi(\lambda))=n$.

\vspace{0.1cm}

To prove the Claim, observe that $\pi_*(S_{\lambda})=\nu'_{\lambda}\Delta(\pi(\lambda))$ for some $\nu'_{\lambda}\in \bbN$ and that we have 
\begin{equation}\label{E:nu1}
\nu_{\lambda}=\nu'_{\lambda}\cdot \deg (\mu_m)_{|\Delta(\pi(\lambda))}\leq \nu'_{\lambda} \cdot \deg \mu_m=\nu'_{\lambda}\cdot m!. 
\end{equation} 
In terms of the projection map   $\pi_{|_{\Delta (\lambda)}} : \Delta (\lambda )\cong C^m \to \Delta (\pi(\lambda )) 
\cong C^n$, the coefficient $\nu_{\lambda}'$ can be expressed as the following intersection number inside $\Delta(\lambda)\cong C^m$
\begin{equation}\label{E:nu2}
\nu_{\lambda}'=[S_{\lambda}] \cdot \pi_{|_{\Delta(\lambda)}}^{*}(p),
\end{equation}
for any point $p\in \Delta(\pi(\lambda))$.  Consider now a permutation $\sigma\in \bS_d$ such that $\lambda=\lambda_0^{\sigma}$ and the induced isomorphism 
$$\psi:= \sigma_* \circ (\pi_{|_{\Delta(\lambda_0 )}})^{-1}: \Delta(\pi(\lambda_0))\stackrel{(\pi_{|_{\Delta(\lambda_0)}})^{-1}}{\longrightarrow} \Delta(\lambda_0) \stackrel{\sigma_* }{\longrightarrow} \Delta(\lambda).$$ 
Note that, since $\sigma(S_{\lambda_0})=S_{\lambda}$, we have that $\psi^*([S_{\lambda}])=\pi_*([S_{\lambda_0}])$. 
Moreover, if we call $x_i$ the divisor on $\Delta(\pi(\lambda_0))\cong C^m$ which is equal to the pull-back of $x$ on $C$ via the $i$-th projection, then $((\pi_{|_{\Delta(\lambda_0)}})^{-1}) ^{*} (\sigma_*)^*(\pi_{|_{\Delta(\lambda)}})^{*}(p)=\prod_{i\in I}x_i=:x_I$, where $I$ is the set (of cardinality $n$) of rows of $\Delta(\pi(\lambda_0))$, or equivalently of $\Delta(\lambda_0)$, that do not disappear under the operations $\lambda_0\mapsto \lambda_0^{\sigma}=\lambda\mapsto \pi(\lambda)$. 
Therefore, if we apply $\psi^*$ to \eqref{E:nu2}, we can express the coefficient $\nu_{\lambda}'$ as an intersection number inside $\Delta(\pi(\lambda_0))$: 
\begin{equation}\label{E:nu3}
\nu_{\lambda}'=[\pi_*(S_{\lambda_0})] \cdot x_I.
\end{equation}
Finally, since $\mu_m^*(x^n)=m!x_I+ \left(\text{nef classes}\right)$, 
we have that 
\begin{equation}\label{E:nu4}
[\pi_*(S_{\lambda_0})] \cdot x_I\leq \frac{1}{m!} [\pi_*(S_{\lambda_0})] \cdot \mu_m^*(x^n)=\frac{[T]\cdot x^n}{m!}.
\end{equation}
Putting together \eqref{E:nu1}, \eqref{E:nu3} and \eqref{E:nu4}, the Claim follows. 
\end{proof}

\begin{proof}[Proof of Theorem \ref{T:eta}\eqref{T:eta2}]
Let $S$ be a $n$-dimensional irreducible subvariety of $C_d$ (with $1\leq n\leq d-1$) which is  not a diagonal. Denote by $m:=l(S)$ the length of $S$. Since $S$ is not a diagonal by assumption, we have that $n+1\leq m \leq d$ by Remark \ref{R:lenght}. It will be sufficient to show the existence of constants $c_{n,m,d}\geq 0$,
which are positive if $g\geq 2$, such that 
\begin{equation}\label{E:const-m}
(\eta_{n,d}-c_{n,m,d}x^n) \cdot [S]\geq 0 \quad \text{for any }S\subset C_d \quad \text{irreducible with }\:\dim S=n<m= l(S)\leq d.
\end{equation}
Indeed, once found $c_{n,m,d}$, we can set $c_{n,d}=\min\{c_{n,m,d}, n+1\leq m\leq d\}$ and we are done.

We will prove \eqref{E:const-m} by induction on $d$. We will distinguish three cases, the first of which gives also the base of the induction, i.e. $d=2$. 

\un{CASE I:} $1\leq n=d-1$ ($\Rightarrow m=d$).

This follows from Proposition \ref{P:eta-div}.

\vspace{0.1cm}
 
\un{CASE II:} $3\leq n+2\leq m=d$. 

We will use the $(d-n-1)$-th iteration of the push-pull operators of \S \ref{S:pushpull} in order to reduce to Case I.
Using Fact \ref{F:pushpull}\eqref{F:pushpull2}, we compute 
\begin{equation}\label{A-x}
A^{d-n-1}(x^{n})=(d-n)! x^n,
\end{equation}
where the $x^n$ on the left lies in $R^n(C_{n+1})$ and the $x^n$ on right lives in $R^n(C_d)$. Therefore, setting $c_{n,d,d}:=(d-n)c_{n,n+1,n+1}$, equations \eqref{A-eta} and \eqref{A-x} give that 
\begin{align*}\label{A-eta-pert}
A^{d-n-1}(\eta_{n, n+1}-c_{n, n+1, n+1} x^n) & =(d-n-1)!(\eta_{n,d}-(d-n)c_{n, n+1,n+1}x^n) \\
& =(d-n-1)! (\eta_{n,d}-c_{n,d,d}x^n).
\end{align*}
Using this and Fact \ref{F:pushpull}\eqref{F:pushpull1}, we get 
\begin{equation}\label{E:red-div}
\begin{aligned}
 \phantom{a}[S] \cdot (\eta_{n,d}-c_{n,d,d}x^n) & =\frac{1}{(d-n-1)!} [S] \cdot A^{d-n-1}(\eta_{n,n+1}-c_{n,n+1,n+1} x^n) = \\
 & =\frac{1}{(d-n-1)!} [B^{d-n-1}(S)] \cdot (\eta_{n,n+1}-c_{n,n+1,n+1} x^n).
\end{aligned}
\end{equation}
Applying Lemma \ref{L:Bcycle}\eqref{L:Bcycle1} iteratively, we get that each irreducible component of the $n$-dimensional cycle $B^{d-n-1}(S)\subset C_{n+1}$ has length $\ge n+1$, i.e. it is not a diagonal. Hence Case I implies that the right hand side of \eqref{E:red-div} is non-negative and we are done. 
 
\vspace{0.1cm}

\un{Case III:} $3\leq n+2\leq m+1\leq d$.

We will use the $(d-m)$-th iteration of the  push-pull operators of \S \ref{S:pushpull} in order to reduce to $C_{m}$ and then apply the induction hypothesis (note that $m<d$).   By iterating $(d-m)$-times formula \eqref{A-geneta}  and iterating $(d-m)$-times the  formula for $A(x^n)$ given in Fact \ref{F:pushpull}\eqref{F:pushpull2}, we get for any $b\in \R$ and using that $n<m$:
\begin{equation}\label{A-d-m}
\begin{aligned}
A^{d-m}(\eta_{n,m}-bx^n) & = (m+1-n)\ldots (d-1-n)[(m-n)\eta_{n,d}-b(d-n)x^n] \\
& = \frac{(d-1-n)!}{(m-1-n)!} \left[\eta_{n,d}-\frac{d-n}{m-n} bx^n\right].
\end{aligned}
\end{equation}
Now, write $B^{d-m}(S)=W+\sum_{\un b} \nu_{\un b} \Delta_{\un b}$ as in  Lemma \ref{L:Bcycle}\eqref{L:Bcycle2} and compute for any $a\geq 0$:
\begin{eqnarray}\label{E:red-m1}
\\
& &\frac{(d-1-n)!}{(m-1-n)!} [S]\cdot  (\eta_{n,d}-ax^{n})=[S]\cdot A^{d-m}\left(\eta_{n,m}-\frac{m-n}{d-n}a x^n\right) \ \text{by } \eqref{A-d-m}  \nonumber \\
& = &  [B^{d-m}(S)]\cdot \left(\eta_{n,m}-\frac{m-n}{d-n}a x^n\right) \nonumber \\
& = & ([W]+\sum_{\un b}\nu_{\un b}[\Delta_{\un b}]) \left(\eta_{n,m}- \frac{m-n}{d-n}a x^n\right) \ \text{using Fact }\ref{F:pushpull}\eqref{F:pushpull1} \nonumber \\
& =& [W]\cdot \eta_{n,m} -\frac{m-n}{d-n} a \left([W]\cdot x^n +\sum_{\un b}\nu_{\un b} [\Delta_{\un b}]\cdot x^n\right) \ 
\text{as }\color{black}\eta_{n,m}\cdot [\Delta_{\un b}]=0 \text{ by Theorem } \ref{T:eta}\eqref{T:eta1}\nonumber \\
& \geq & [W]\cdot \eta_{n,m} -\frac{m-n}{d-n} a \left([W]\cdot x^n +\max\{\nu_{\un b}\} \sum_{\un b} [\Delta_{\un b}]\cdot x^n\right) \nonumber \\
& \geq & [W]\cdot \eta_{n,m} -\frac{m-n}{d-n} a \left(1+K_{n,m,d} \sum_{\un b} [\Delta_{\un b}]\cdot x^n\right) 
(\color{black}[W]\cdot x^n)
 \ \text{by Lemma \ref{L:Bcycle}\nonumber}.
\end{eqnarray}

Now, by the induction hypothesis, we can assume that we have already found constants $\{c_{n,k,m}\geq 0\}_{n<k\leq m}$, which are positive if $g\geq 2$, in such a way that \eqref{E:const-m}  holds true. Hence, if we take $a=c_{n,m,d}\geq 0$, and positive if $g\geq 2$,   so that  
$$\frac{m-n}{d-n}c_{n,m,d}\left(1+K_{n,m,d} \sum_{\un b} [\Delta_{\un b}]\cdot x^n\right)=\min\{c_{n,k,m}, n<k\leq m\},$$
then the validity of \eqref{E:const-m} for any $n<k\leq m$ implies that the last term in \eqref{E:red-m1} is non-negative and the first statement is proved.

Let us now prove the last statement 
of Theorem \ref{T:eta}\eqref{T:eta2}\color{black}. The fact that $\eta_{n,d}$ is nef  follows immediately from \eqref{T:eta1} and the existence of $c_{n,d}\geq 0$. In order to prove that $\eta_{n,d}$ is big for $g\geq 2$, consider the partition  $\un b=(b_1,\ldots,b_{d-n})$  with $b_1=n+1, b_2=\ldots =b_{d-n}=1$ and let $\Delta_{\un b}$ be the corresponding diagonal. Now Proposition \ref{diagclass} gives 
$$[\Delta_{\un b}] = (n+1)(d-n-1)!\big((d-n+gn)x^n - n x^{n-1}\theta\big)$$
and therefore 
\begin{equation}
\label{big}
\eta_{n,d} = \frac{(d-n)(g-1)}{n} x^n + \frac{1}{(n+1)(d-n-1)!} [\Delta_{\un b}].
\end{equation}
As $x$ is ample, it follows by \cite[Cor. 2.12]{fl1} that $x^n$ is a big $(d-n)$-cycle, therefore so is $\eta_{n,d}$, being the sum of a big and an effective $(d-n)$-cycle. 
\end{proof}

\begin{remark}
It follows by Theorem \ref{T:eta}(i) and \eqref{big} that, for $g \ge 2$, diagonal cycles are not nef. In fact let $\Delta_{\un a}$ be any $n$-dimensional diagonal cycle. Then 
$$0 = [\Delta_{\un a}] \cdot \eta_{n,d} = \frac{(d-n)(g-1)}{n} [\Delta_{\un a}] \cdot x^n + \frac{1}{(n+1)(d-n-1)!} [\Delta_{\un a}] \cdot [\Delta_{\un b}]$$
whence $[\Delta_{\un a}] \cdot [\Delta_{\un b}] < 0$.
\end{remark}

Theorem \ref{T:eta} gives 
some very nice convex-geometric consequences for the $n$-dimensional diagonal cone $\D_n(C_d)$ inside the cones of $n$-dimensional (tautological or not) effective (and pseudoeffective) cycles on $C_d$ 
(for the definition of perfect face and edge see \ref{E:perf-face} in Appendix \ref{B}).

\begin{cor}\label{diag-poly}
For any $1\leq n\leq d-1$ and $g\geq 1$ we have:
\begin{itemize}
\item[(a)] $\Pseff_n(C_d)$ is locally finitely generated at every $\alpha \in \{\eta_{n,d}-c_{n,d}x^n < 0\}$, in particular at every non-zero $\alpha$ in the $n$-dimensional diagonal cone;
\item[(b)] the $n$-dimensional diagonal cone is $\{\eta_{n,d}=0\} \cap \Pseff_n(C_d)$;
\item[(c)] every extremal ray of $\Pseff_n(C_d)$ that is contained in $\{\eta_{n,d} - c_{n,d}x^n < 0\}$ 
is generated by a diagonal cycle;
\item[(d)] every face of the $n$-dimensional diagonal cone is a perfect face of $\Pseff_n(C_d)$ (and hence also of $\Eff_n(C_d)$);
\item[(e)] if $R_n(C_d) = N_n(C_d)$ (in particular for a very general curve) then $\cone(\eta_{n,d})$ is an edge of $\Nef^n(C_d)$.
\end{itemize} 

The same conclusions (a)-(d) hold for $\Psefft_n(C_d)$ and $\Efft_n(C_d)$. Moreover $\cone(\eta_{n,d})$ is an edge of $\Neft^n(C_d)$.
\end{cor}


\begin{proof}[Proof of Corollary \ref{diag-poly}]
Let us first prove the case $g=1$. In this case $N_n(C_d)=R_n(C_d)$ has dimension two because $C_d$ is a projective bundle over $\Pic^d(C)\cong C$ by Fact \ref{projbund}; therefore, $\Pseff_n(C_d)= \Psefft_n(C_d)$  has two extremal rays. By Example \ref{E:r1} and Theorem \ref{T:eta}\eqref{T:eta1}, all the diagonal cycles lie on the ray $(\eta_{n,d})^\perp \cap \Pseff_n(C_d)$. Moreover, since $\eta_{n,d}$ is nef by  Theorem \ref{T:eta}\eqref{T:eta2}, we get that 
the ray $(\eta_{n,d})^\perp \cap \Pseff_n(C_d)$ is an extremal ray of $\Pseff_n(C_d)= \Psefft_n(C_d)$ and $\cone(\eta_{n,d})$ is an extremal ray of $\Nef^n(C_d)= \Neft^{
n\color{black}}(C_d)$. The case $g=1$ is proved. 

 Let us now prove the case $g\geq 2$.  The first part is a straightforward consequence of Corollary \ref{edges2} and Theorems \ref{T:eta} and \ref{T:diagcone}. 
For the case $\Psefft_n(C_d)$,  set
\begin{displaymath}
Y_1=\left\{\left. \sum\limits_{j=1}^s a_j [W_j] \right| \begin{array}{l} a_j \ge 0, \,W_j \text{ is an irreducible } n\text{-dimensional subvariety}\\ \text{of } C_d \text{ that is not a diagonal, and } \sum_j a_j [W_j] \in R_n(C_d) \end{array}\right\}
\end{displaymath}
and apply Proposition \ref{edges1} with $V = R_n(C_d), Y = Y_1 \cup  \{n-\mbox{dimensional diagonals}\}$ and the same functionals.
Now $K(Y) = \Psefft_n(C_d)$ follows as in the first part of the proof of Proposition \ref{edges1}, that is that any $\alpha \in \Psefft_n(C_d)$ can be written as $u + w$ with $u$ in the $n$-dimensional diagonal cone and $w \in K(Y_1)$. 
\end{proof}

%

Let us point out also that from Theorem \ref{T:eta} it follows that the property of being a positive linear combination of diagonal cycles is a numerical property.  

\begin{cor}\label{diag-rig}
For any $1\leq n\leq d-1$ and $g\geq 1$, consider two effective $n$-cycles such that $[S]=[T]\in N_n(C_d)$. If $S$ is a positive linear combinations of diagonal cycles then also $T$ is a positive linear combination of diagonal cycles.

In particular, the diagonal cycles that span an extremal ray of $\D_n(C_d)$ (see Theorem \ref{T:diagcone})
are numerically rigid, i.e. if $S$ is any such diagonal cycle and $T$ is an effective $n$-cycle such that $[T]=[S]$ then $T=S$. 
\end{cor}
\begin{proof}
Theorem \ref{T:eta} implies that the positive linear combinations of diagonals cycles are the unique effective cycles whose class lies on the hyperplane $\{\eta_{n,d}=0\}$. Since this latter property is a numerical property, the result follows.
\end{proof}


\appendix
\section{Convex geometry of cones} 
\label{B}

We collect here all the definitions and results about the convex geometry of cones that we need throughout the paper.  A good reference is \cite{Roc}.

Let $V$ be a real finite dimensional vector space with the Euclidean topology. A (convex) \textbf{cone} $K$ inside $V$ is a non-empty subset $K$ of $V$ such that if $x,y\in K$  and $\alpha, \beta\in \R^{>0}$ then $\alpha x+\beta y\in K$. We say that a cone $K$ \emph{contains} $0$ if $0\in K$, 
or equivalently if $\alpha x+\beta y\in K$ for any $x,y\in K$ and $\alpha,\beta\in \R_{\geq 0}$. Given a cone $K$, the set $K\cup \{0\}$ is a cone that contains $0$. 

Given a cone $K\subseteq V$, the  smallest linear subspace containing $K$ is called the \emph{linear hull} of $K$ and it is denoted by $\langle K \rangle$. The dimension of $\langle K \rangle$ is called the dimension of the cone $K$ and it is denoted by $\dim K$. A cone is said to be \emph{full} if $\langle K \rangle =V$, or equivalently if $\dim K=\dim V$. 
There is a largest subspace contained in $K\cup \{0\}$, namely $(K \cup \{0\}) \cap (- (K\cup \{0\}))$, which is called the \emph{lineality space} and it is denoted by $\lin(K)$. The dimension of $\lin(K)$ is called the \emph{lineality} of the cone $K$ 
and it is denoted by $\ldim(K)$. A cone is said to be \emph{salient} (or sometimes called \emph{pointed})
if $\lin(K)=\{0\}$, or equivalently if it does not contain lines through the origin. 

To any cone $K$, we can associate a closed cone, namely the \emph{closure}  $\ov K$ of $K$, and a relatively open cone, namely the \emph{relative interior} $\ri(K)$ of $K$ inside $\langle K\rangle$, with the properties that 
$$\ri(K) \subseteq K \subseteq \ov{K}.$$
The above inclusions induce equalities on the linear hulls but only (possibility strict) inclusions on the lineality spaces. 
Hence, we have that
$$
\begin{sis}
& K \: \text{is full }\Longleftrightarrow \ri(K) \:  \text{is full }\Longleftrightarrow \ov K \: \text{is full,}\\
& \ov K \: \text{is salient } \Longrightarrow K \: \text{is salient} \Longrightarrow \ri(K) \: \text{is salient.}
\end{sis}
$$
The \emph{relative boundary} of $K$  is $\partial K:=K\setminus \ri(K)$.  

Given a subset $S$ of $V$, we will denote by $\aff(S)$ the \emph{affine hull} of $S$, that is the smallest affine subspace containing $S$, by $\conv(S)$ the \emph{convex hull} of $S$, that is the smallest convex subset containing $S$ and by
$\cone(S)$ the \emph{conical hull} of $S$, that is the smallest cone containing $S\cup \{0\}$, which is equal to the intersection of all the cones that contain $S\cup \{0\}$. Note that $\cone(S)$ is a full cone containing $0$ inside the linear hull $\langle S \rangle=\langle \cone(S) \rangle$. Also, if $S$ is a finite set, then $\conv(S)$ is a full dimensional polytope inside $\aff(S)=\aff(\conv(S))$. 

\vspace{0.1cm}

Given a cone $K\subseteq V$, the \textbf{dual cone} (or polar cone)  of $K$ is the cone in the dual vector space $V^{\vee}$ defined by 
$$K^{\vee}:=\{l \in V^{\vee}\: : \: l(x)\geq 0 \quad \text{for all }x\in K\}\subseteq V^{\vee}.$$
Let us summarize the main properties of the duality operator of cones in the following


\begin{fact}\label{F:dual-cone} Let $K \subseteq V$ be a cone.
\noindent
\begin{enumerate}[(i)]
\item \label{F:dual-cone1}  The dual cone $K^{\vee}$ is always closed, and $K^{\vee}=(\ov K)^{\vee}$. Moreover, taking the double dual is equivalent to taking the closure, i.e. $(K^{\vee})^{\vee}=\ov{K}$. In particular, upon fixing an identification $V\cong V^{\vee}$, the duality operator $K\mapsto K^{\vee}$ is an involution on the set of closed convex cones in $V$.
\item \label{F:dual-cone2}  $\langle K \rangle = (\lin (K^{\vee}))^{\perp}$ and $\langle K^{\vee} \rangle = (\lin (\ov{K}))^{\perp}$. In particular, 
$$\codim K=\ldim (K^{\vee}) \: \text{and }\ldim (\ov{K})=\codim K^{\vee}.$$
\end{enumerate}
It follows that a closed cone $K$ is salient (resp. full) if and only if 
$K^{\vee}$ is full (resp. salient).
\end{fact}
\begin{proof}
For 
\eqref{F:dual-cone1}, see \cite[Thm. 14.1 and page 121]{Roc}. For 
\eqref{F:dual-cone2}, see \cite[Thm. 14.6]{Roc}.
\end{proof}

In the following, we will be interested in cones $K$ such that $\ov K$ is salient. For any such cone $K$, if we take $l\in \ri(K^{\vee})$ (which is non-empty since $\ov K$ is salient) then the intersection $\{l=1\} \cap K$ is a bounded convex set (called a \emph{bounded section}  of $K$) and, if $K \neq \{0\}$, $K\cup \{0\}$ is the cone over $\{l=1\}\cap K $, i.e.
$$K\cup \{0\}=\R_{\ge 0}\cdot (\{l=1\}\cap K):=\{\lambda\cdot x\: : \lambda\in \R_{\geq 0} \text{ and }x\in \{l=1\}\cap K\}.
$$
Note that, moreover, $K\cup \{0\}$ is closed if and only if some (or equivalently any) bounded section is compact, i.e. it is a convex body in $\{l=1\}$.  Conversely, if $S\subset V$
is a bounded convex set, then the abstract cone over $S$, i.e.
$$C(S):=\{\lambda\cdot (1, x)\in \R\oplus V\: : \lambda \geq 0, x\in S\}\subset \R\oplus V,$$
is a cone (containing $0$) with salient closure, and moreover $C(S)$ is closed if and only if $S$ is a convex body. 
Via these two constructions, the study of cones (resp. closed cones) with salient closure can be reduced to the study 
of bounded convex sets (resp. convex bodies).  
Moreover, using the above constructions, a more abstract characterization of cones with salient closure can be obtained: a cone $K$ has salient closure $\ov K$ if and only 
if \color{black}$K$ does not contain lines (not necessarily through the origin).


\vspace{0.1cm}


Given a cone $K$, 
a  \textbf{face}  of $K$ is a subcone 
$F\subseteq K$ such that whenever $x,y\in K$ are such that $x+y\in F$ then $x,y\in F$.  
Note that $K$ is always a face, while if $0 \in K$, then $\{0\}$ is a face of $K$ if and only if $K$ is salient. A non-trivial face of $K$ is a face that is different from  $\{0\}$ and $K$, which are the unique  faces of dimension, respectively, $0$ and $\dim K$. Faces of dimension one are either lines through the origin (which can occur only if $K$ is not salient) or  \emph{extremal rays} of $K$, i.e. rays  of the form $\{0\}\neq \cone(v)\subseteq K$ with the property  that if $v=w_1+w_2$ for some $w_1, w_2\in K$ then $w_1, w_2\in \cone (v)$. A face of codimension one is called a \emph{facet}.
The faces of a cone $K$ form a poset $\F(K)$ with respect to the inclusion relation. We denote by $\F_{r}(K)$ the 
set of \color{black}faces of $K$ of dimension $r$. The (arbitrary) intersection of faces of a cone $K$ is still a face (so that the poset $\F(K)$ admits a meet given by the intersection)  and faces are transitive, i.e. if $F$ is a face of $K$ and $G$ is a face of $F$ then $G$ is a face of $K$. 

\begin{fact}\label{F:faces}
Let $K$ be a cone 
\begin{enumerate}[(i)]
\item \label{F:faces1}
There is a partition of $K$ into locally closed subsets 
$$K=\coprod_{F \in \F(K)} \ri(F),$$
and the cones $\{\ri(F)\}_{F \in \F(K)}$ are exactly the maximally relatively open subcones of $K$.
\item \label{F:faces2}
If $K=\cone(S)$ for a certain set $S$, then every face $F$ of $K$ is equal to 
$$F=\cone(S\cap F).$$
\end{enumerate}
\end{fact}
\begin{proof}
For part \eqref{F:faces1}, see \cite[Thm. 18.2]{Roc}. For part \eqref{F:faces2}, see \cite[Thm. 18.3]{Roc}. 
\end{proof}


From now on, we will restrict our attention to faces of salient closed cones $K \neq \{0\}$, which are easier to deal with.  As explained above, a bounded section of $K$ is a convex body $B$ such that the abstract cone $C(B)$ over $B$ is isomorphic to $K$. Therefore, the non-zero faces of $K$ are exactly the cones over the faces of $B$. Hence, it is easy to translate properties of faces of 
$B$ (like the ones contained in \cite[Chap. 2]{Sch}) into properties about the faces of $K$. 


\vspace{0.1cm}

There is a duality operation between faces of a closed cone $K\subseteq V$ and the faces of its dual cone $K^{\vee}\subseteq V^{\vee}$. Given a face $F$ of $K$, the \textbf{dual  face} (or normal cone) of $F$ is 
\begin{equation}\label{E:dual-fac}
\wh{F}:=\{l \in K^{\vee}\::   l_{|F} = 0\}. 
\end{equation} 
We get a face duality operation (for any salient full closed cone $K$)
\begin{equation}\label{E:dualoper}
\begin{aligned}
D_K:\F(K)& \longrightarrow \F(K^{\vee}),\\
F & \mapsto \wh{F}.
\end{aligned}
\end{equation}

We are now going to see that the faces that behave particularly well with respect to face duality are the following one: a face $F$ of $K$ is said to be \textbf{exposed} if either $F$ is trivial \footnote{This is a convention that will be very useful in what follows and it is different from the standard definition of exposed face, according to which $\{0\}$ is exposed but $K$ is not.} or $F=K\cap H$ for some  \emph{supporting hyperplane} $H$, i.e. a linear hyperplane $H=\{l=0\}$ such that $K$ is contained in the half-space $H^+=\{l\geq 0\}$. We denote by $\E(K)\subseteq \F(K)$ the subset of exposed  faces of $K$. 
Exposed faces of dimension one are called \emph{exposed extremal rays}. 
Since it is easily checked that the (arbitrary) intersection of exposed faces is an exposed face, every face different from $K$ is contained in a minimal exposed face, namely $D_{K^{\vee}}(D_K(F))$ (see \cite{Bar}, \cite[p. 75]{Sch}). Hence the natural inclusion $\E(K)\subseteq \F(K)$ admits a retraction 
$$\epsilon_K: \F(K)\to \E(K)$$
sending  a face $F \in \F(K)$ into the smallest exposed face containing $F$. 

\begin{fact}\label{F:exp-dual}
Let $K\subseteq V$ be a closed cone. Then for any face $F$ of $K$, we have that 
$$D_K(F)\in \E(K^{\vee}), \quad D_K(F)=D_K(\epsilon_K(F)), \quad D_{K^{\vee}}(D_K(F))=\epsilon_K(F).$$ 
\end{fact}
\begin{proof}
Follows from \cite[Thm. 2.1.4]{Sch}. 
\end{proof}




It follows directly from the definition \eqref{E:dual-fac} that for any face $F$ of $K$ we have the inequality
\begin{equation}\label{E:ineq-faces}
\dim F+\dim \wh F\leq \dim V.
\end{equation}
Faces 
achieving equality in \eqref{E:ineq-faces} are called \textbf{perfect (or regular)}, see \cite[Note 4 at p. 90]{Sch}.


More explicitly, a face $F$ of $K$ is perfect if and only if, setting $c = \codim_V F$,
there exist linear hyperplanes $H_i =\{l_i=0\}_{1\leq i \leq c}$ such that\footnote{When $c=0$, we use the standard convention that an empty intersection of hyperplanes is the whole space. In particular, $K$ is always a perfect face.}
\begin{equation}\label{E:perf-face}
\begin{sis}
& K\subseteq H_i^+=\{l_i \geq 0\}\quad \text{for any }\: 1\leq i \leq c, \\
& \langle F \rangle=\cap_{i=1}^c H_i.
\end{sis}
\end{equation}


Note that $K$ and $\lin{\ov K}$ are perfect faces of $K$ and $\ov K$ respectively. Also if $K$ is a cone containing $0$, then $\{0\}$ is a face if and only $K$ is salient and in that case $\{0\}$ is perfect.

We will denote by $\P(K)$ the set of all perfect faces of $K$ and by $\P_r(K)$ the set of all perfect faces of $K$ of dimension $r$. 
Perfect faces of dimension one are called {\bf perfect extremal rays} or {\bf edges} as in \cite[\S 6]{o1}.

\begin{fact}\label{F:perf-dual}
Let $K\subset V$ be a salient full closed cone of dimension $n:=\dim K$.
Perfect faces are exposed and the face duality operator \eqref{E:dual-fac} induces a bijection
$$D_K:\P_r(K)\stackrel{\cong}{\longrightarrow} \P_{n-r}(K^{\vee}),$$
for any $0\leq r \leq n$, whose inverse is $D_{K^{\vee}}$.
\end{fact}

\begin{proof}
Let us first show that perfect faces are exposed. Let $F$ be a perfect  face of $K$. 
Then there exist linear hyperplanes $H_i =\{l_i=0\}_{1\leq i \leq c}$, where $c$ is the codimension of $F$ in $V$, satisfying \eqref{E:perf-face}. 
Let $a_i > 0, 1 \le i \le c$, let $l = \sum_{i=1}^c a_il_i$ and let $H = \{l=0\}$ and $H^+ = \{l\geq 0\}$. Clearly 
$$
\begin{sis}
& K\subseteq \cap_{i=1}^c H_i^+\subseteq H^+, \\
&  F\subset \langle F\rangle= \cap_{i=1}^c H_i\subseteq H.
\end{sis}
$$
It remains to prove that $K\cap H\subseteq F$. Take an element $x \in K \cap H$. Then  $0 = l(x) = \sum_{i=1}^c a_il_i(x)$ and $l_i(x) \ge 0$, for every $1 \le i \le c$, whence $l_i(x) = 0$ for $1 \le i \le c$. Hence
$x \in K \cap \left(\cap_{i=1}^c H_i \right) = K \cap \langle F \rangle$. Since $ \langle F \rangle= F - F$,  there are $y,z \in F$ such that $x = y - z$. Therefore $y = x + z\in F$ with $x, z \in K$, which by the definition of face, implies that  $x \in F$. Thus $F$ is exposed.

If $F$ is perfect then, using that $F$ is exposed by what was just shown and using Fact \ref{F:exp-dual}, we get that 
$$\dim \wh{\wh{F}}+ \dim \wh F= \dim F+\dim \wh F=\dim V,$$
which shows that $\wh F$ is a perfect face of $K^{\vee}$.  Therefore, the bijection $D_K:\E(K)\stackrel{\cong}{\longrightarrow} \E(K^{\vee})$ restricts to a bijection between  $\P(K)$ and $\P(K^{\vee})$ which, by definition, sends a face of dimension $r$ into a face of dimension $n-r$.
\end{proof}


\begin{remark}\label{perf}
Let $K\subset V$ be a salient full closed cone. It follows from the definition of a perfect face  together with Fact \ref{F:perf-dual} that if $F$ is a perfect face of $K$  then $F$ is a full cone in $\langle F\rangle$ and $\wh{F}$ is a full cone in $\langle \wh{F}\rangle=\langle F\rangle^{\perp}$. Conversely, it is easy to see that if $L\subseteq V$ is a subspace such that $K\cap L$ is a full cone in $L$ and $K^{\vee}\cap L^{\perp}$ is a full cone in $L^{\perp}$ then $F:=K\cap L$ is a perfect face of $K$ with dual face being perfect and equal to $K^{\vee}\cap L^{\perp}$.   
\end{remark}

\begin{remark}\label{R:faces}
\noindent 
\begin{enumerate}[(i)]
\item 
For a face $F$ of a full salient closed cone $K$, the chain of implications:
$$ F \ \text{is perfect} \Longrightarrow F \ \text{is exposed}   \Longrightarrow F \ \text{is  face}
$$
are all strict, except in codimension one (i.e. for facets) where they coincide.
\begin{figure}
\begin{tikzpicture}[scale=2]
\path (-1cm,-1cm) coordinate (B) ;
\path (-2cm,-1cm) coordinate (A) ;
\path (-2cm,1cm) coordinate (E) ;
\path (-1cm,1cm) coordinate (D) ;

\filldraw[color=black!30!white] (A) node [fill, color=black, shape=circle, scale=0.3, label=below:$\color{black}A$]{} -- (B) node [fill, color=black, shape=circle, scale=0.3, label=below:$\color{black}B$]{}
         arc (-90:30:1cm) node [fill, color=black, shape=circle, scale=0.3, label=30:$\color{black}C$]{} arc (30:90:1cm) -- (D)  node [fill, color=black, shape=circle, scale=0.3, label=above:$\color{black}D$]{} -- (E) node [fill, color=black, shape=circle, scale=0.3, label=above:$\color{black}E$]{}  -- cycle; 

\draw (A)  -- (B)  arc (-90:90:1cm)  -- (D)  -- (E)  -- cycle;

\node[label=$K \cap \{ l \uguale 1 \}$] at (-1cm,-0.25) {};

\path (3.5cm,0cm) coordinate (AEv) ;
\path (2.5cm,1cm) coordinate (ABv) ;
\path (2.5cm,-1cm) coordinate (DEv) ;

\filldraw[color=black!30!white] (AEv) node [fill, color=black, shape=circle, scale=0.3, label=right:$\color{black}\widehat{AE}$]{} -- (ABv) node [fill, color=black, shape=circle, scale=0.3, label=above:$\color{black}\widehat{AB}$]{}
         arc (90:210:1cm) node [fill, color=black, shape=circle, scale=0.3, label=210:$\color{black}\widehat{C}$]{} arc (210:270:1cm) -- (DEv) node [fill, color=black, shape=circle, scale=0.3, label=below:$\color{black}\widehat{DE}$]{} --  cycle; 

\draw (AEv)  -- (ABv) arc (90:270:1cm) -- (DEv) --  cycle; 

\node[label=$K^{\vee} \cap \{ l' \uguale 1 \}$] at (2.5cm,-0.25) {};
\node[label=-45:$\widehat{E}$] at (3cm,-0.5cm) {};
\node[label=45:$\widehat{A}$] at (3cm,0.5cm) {};

\end{tikzpicture}
\caption {
\color{black} Sections taken with $l \in \ri(K^{\vee})$ and $l' \in \ri(K)$. The 1-dimensional faces $A,E$ and the facets $AE, AB,DE$ are perfect. The 1-dimensional face $C$ is exposed, but not perfect. The 1-dimensional faces $B$ and $D$ are not exposed.}
\label{due}
\end{figure}
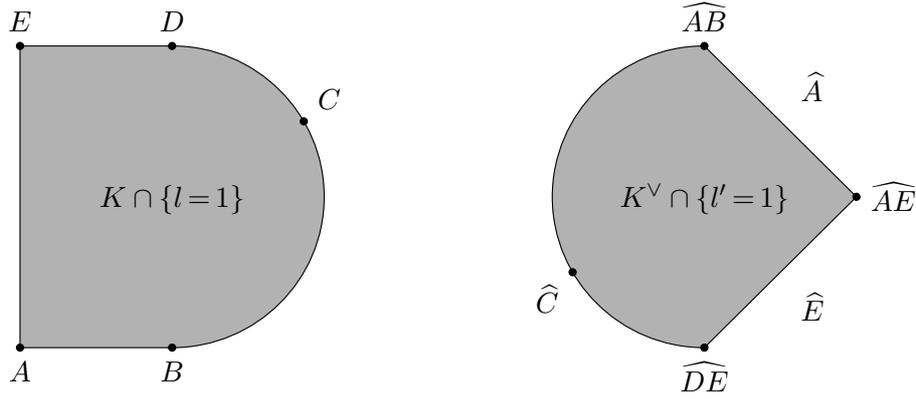

\noindent In Figure \ref{due} we show a $2$-dimensional convex body (section of a $3$-dimensional closed salient full cone) 
illustrating the 
three \color{black}different types of faces. 
\item Exposed and perfect faces are not transitive, i.e. if $F$ is an exposed (resp. perfect) face of $K$ and $G$ is a face of $F$, then $G$ is not necessarily an exposed face of $K$, although it will obviously be a face (see again Figure \ref{due}).
\end{enumerate}
\end{remark}

Faces also provide one way of representing a closed cone, the other being through supporting hyperplanes. 

\begin{fact}\label{F:repres}
Let $K$ be a closed cone. 
\begin{enumerate}
\item \label{F:repres1} [External representation]
$K$ is the intersection of its supporting half-spaces;
\item \label{F:repres2} [Internal representations]
If $K$ is salient then $K$ is the convex hull of its extremal rays.
\end{enumerate}
\end{fact}
\begin{proof}
For a proof, see \cite[Thm. 11.5]{Roc}, \cite[Thm. 18.5]{Roc} and \cite[Thm. 18.7]{Roc}. 
\end{proof}

\vspace{0.1cm}

The nicest cones are the polyhedral cones:  a cone $K\subseteq V$ is called \textbf{polyhedral} if one of the following equivalent condition holds (see \cite[Thm. 19.1]{Roc})
\begin{enumerate}[(i)]
\item $K$ it is the conic hull of a  finite set of points of $V$;
\item $K$ is the intersection of finitely many closed half-spaces passing through the origin:
\item $K$ is closed and it has finitely many faces. 
\end{enumerate} 
The dual of a polyhedral cone is polyhedral (see \cite[Cor. 19.2.2]{Roc}) and every face of a polyhedral cone is perfect (see \cite[Sec. 2.4]{Sch}).

Let $K \subseteq V$ be a closed cone and let $v \in \partial K$. We say that $K$ is  \textbf{locally polyhedral} at $v$ if there exist a neighborhood $U$ of $v$ in $V$ and a polyhedral cone $P \subseteq V$ such that $K \cap U = P \cap U$. We will say that $K$ is \textbf{locally polyhedral} if it is such at every point $v \in \partial K$. We say that $K$ is \textbf{locally finitely generated} at $v$ if there exist a subset $C \subset K$ and $v_1,\ldots,v_s \in K$ such that $C$ is closed in $V$, $\alpha x \in C$ for every $x \in C$ and $\alpha > 0$, $v \not\in C$ and $K$ is generated by $C$ and $\{v_1,\ldots,v_s \}$. 


\begin{remark}
Let $K \subseteq V$ be a cone and let $v \in \partial K$. It is easily seen that if $K$ is locally polyhedral at $v$ then it is locally finitely generated at $v$.
\end{remark}

We now prove the main result of this appendix. This is crucial in the proof of Corollary \ref{diag-poly}.
\begin{prop}
\label{edges1}
Let $V$ be an $\R$-vector space of dimension $\rho$. 
Let $\{\delta_1, \ldots, \delta_s\} \subseteq Y \subset V$ be two subsets and assume that $\dim(\cone(\delta_1,\ldots, \delta_s)) = r \ge 1$. Let $K(Y)$ be the closed convex cone generated by $Y$ and suppose we have two linear maps
$l, \phi : V \to \R$ with the following properties:
\begin{itemize}
\item[(i)] $\phi \in \ri(K(Y)^{\vee})$, that is $\phi(x) > 0$ for every $0 \neq x \in K(Y)$;
\item[(ii)] $l \in K(Y)^{\vee}$, that is $l(x) \ge 0$ for every $x \in K(Y)$;
\item[(iii)] $l(\delta_i) = 0$ for $1 \le i \le s$;
\item[(iv)] $(l - \phi)(y) \ge 0$ for every $y \in Y \setminus \{\delta_1,\ldots, \delta_s\}$.
\end{itemize}
Then
\begin{itemize}
\item[(a)] $K(Y)$ is locally finitely generated at every $v \in \{l - \phi < 0 \}$, in particular at every $v \in \cone(\delta_1,\ldots, \delta_s), v \ne 0$;
\item[(b)] $\cone(\delta_1,\ldots, \delta_s) = \{l=0\} \cap K(Y)$;
\item[(c)] every extremal ray of $K(Y)$ that is contained in $\{l - \phi < 0\}$ is $\R_{\ge 0}\delta_i$ for some $i$;
\item[(d)] every face of $\cone(\delta_1,\ldots, \delta_s)$ is a perfect face of $K(Y)$;
\item[(e)] if $\rho=r+1$ then $\cone(l) = D_{K(Y)}(\cone(\delta_1,\ldots, \delta_s))$ is an edge of $K(Y)^{\vee}$.
\end{itemize}
\end{prop}


To prove Proposition \ref{edges1} we need the following simple fact (whose proof is left to the reader).
\begin{lemma}
\label{comp}
Let $K \subset V$ be a closed convex cone and let $\phi \in \ri(K^{\vee})$. 

For any $b > 0$ the subsets
\[ K^{\le b} = \{x \in K : \phi(x) \le b\} \ \mbox{and} \ \ K^{= b} = \{x \in K : \phi(x) = b\} \]
are compact.
\end{lemma}

\renewcommand{\proofname}{Proof of Proposition \ref{edges1}}
\begin{proof}
We set 
$\varphi = l -\phi$. 
Let $v \in K(Y)$. We have that $v = \lim\limits_{m \to +\infty} v_m$ with
\[ v_m = \sum_{i=1}^{p_m} a_{i,m} y_{i,m} \ \mbox{and} \ a_{i,m} \ge 0, \ y_{i,m} \in Y. \]
Reordering the 
vectors $y_{i,m}$ \color{black}we can assume that $v_m = u_m + w_m$ with $u_m \in \cone(\delta_1, ..., \delta_s)$ and $w_m$ a conical combination of $y_{i,m} \in Y \setminus \{\delta_1,\ldots, \delta_s\}$, so that 
$\varphi(w_m) \ge 0$ by (iv).
Now $\phi(v) = \lim\limits_{m \to +\infty} \phi(v_m)$ whence $\{\phi(v_m), m \in \NN \}$ is bounded. Since 
\[ \phi(v_m) = \phi(u_m) + \phi(w_m) \ge \phi(u_m) \] 
we get that also $\{\phi(u_m), m \in \NN \}$ is bounded. By Lemma \ref{comp}, upon passing to a subsequence, we can assume that the $u_m$ converge to some $u \in \cone(\delta_1,\ldots, \delta_s)$ (since $\cone(\delta_1,\ldots, \delta_s)$ is closed) and therefore that also the $w_m$ converge to some $w \in K(Y)$ with 
$\varphi(w) \ge 0$.

Hence $v = u + w$. This proves that 
\begin{equation}
\label{gen}
K(Y) \ \mbox{is generated by} \ K(Y) \cap \{\varphi \ge 0\} \ \mbox{and} \ \delta_1,\ldots, \delta_s. 
\end{equation}
Now (a) follows by setting 
$C = K(Y) \cap \{\varphi \ge 0\}$. To see (b), one inclusion being obvious by (iii), assume that $l(v)=0$. Then also $l(w)=0$ and therefore 
$0 \le \varphi(w) = -\phi(w)$, giving $\phi(w)=0$, whence $w = 0$ and $v = u  \in \cone(\delta_1,\ldots, \delta_s)$, so that we get (b). To see (c) suppose that $\R_{\ge 0} v$ is an extremal ray of $K(Y)$ such that 
$\varphi(v) < 0$.
It follows by extremality that $w = c v$ for some $c \ge 0$. But 
$0 \le \varphi(w) = c \varphi(v) \le 0$, giving $c = 0$. Then $w = 0$ and $v = u \in \cone(\delta_1,\ldots, \delta_s)$, that is $v = \sum_{i=1}^s a_i \delta_i$ with $a_i \ge 0$. But $v \neq 0$, whence there is an $i$ such that $a_i > 0$ and by extremality we get that $a_i \delta_i \in \R_{\ge 0} v$, thus giving (c).

Let us now prove (d). Since $C=K(Y) \cap \{\varphi \ge 0\}$ is a closed convex cone of $V$ contained in $\{\varphi \ge 0\}$ such that $\phi \in \ri(C^{\vee})$, we claim that we can pick  a polyhedral cone $Q$ of $V$ such that 
\begin{equation}\label{polyQ}
C \subseteq Q\subseteq  \{\varphi \ge 0\} \quad \text{ and } \quad   \phi \in \ri(Q^{\vee}).
\end{equation}
Indeed, since $\phi \in \ri(C^{\vee})$, Lemma \ref{comp} implies that the section $C\cap \{\phi=1\}$ is a  convex compact subset of $\{\phi=1\}$, which is moreover contained in the half-space $\{\varphi \geq 0\} \cap \{\phi=1\}$. Therefore, we can pick a polytope $B\subset \{\varphi \geq 0\} \cap \{\phi=1\}$ such that $C\cap \{\phi=1\}\subseteq B$. Now the cone $Q\subset V$ over $B$
satisfies \eqref{polyQ}.

Set ${\mathcal D} = \cone(\delta_1,\ldots, \delta_s)$. We now claim that 
\begin{equation}\label{big-poly}
P:={\mathcal D}+Q \subset V \text{ is a polyhedral cone such that } K(Y)\subseteq P \text{ and } {\mathcal D} \text{ is a face of } P.  
\end{equation}
Indeed, since ${\mathcal D}$ and $Q$ are polyhedral cones, then so is $P$. Moreover, $P$ contains $K(Y)$ since $K(Y)$ is generated by ${\mathcal D}$ and $C$ by \eqref{gen} and $Q$ contains $C$ by \eqref{polyQ}. Finally, let us show that $\{l=0\}$ is a supporting hyperplane for ${\mathcal D}$ inside $P$. First of all, $P$ is contained in $\{\phi+\varphi=l\geq 0\}$ since $l$ vanishes on ${\mathcal D}$ by (iii) and $Q$ is contained in $\{\phi+\varphi=l\geq 0\}$ by \eqref{polyQ}. Moreover, using \eqref{polyQ}, a point $p=k+q\in {\mathcal D}+Q=P$ with $k \in \mathcal{D}$ and $q\in Q$ belongs to $\{l=0\}$ if and only if 
$$0=l(p)=l(k)+l(q)=l(q) \Leftrightarrow \phi(q)=\varphi(q)=0 \Leftrightarrow q=0.$$
In other words, $\{l=0\}\cap P={\mathcal D}$, which concludes the proof of \eqref{big-poly}. 

Using  \eqref{big-poly}, we now conclude the proof of part (d). Indeed, take any face $F$ of ${\mathcal D}$ and let $c:= \codim_V F$. 
Now $F$ is also a face of $P$ since ${\mathcal D}$ is a face of $P$ by \eqref{big-poly}. Since $P$ is polyhedral by \eqref{big-poly}, 
the face $F$ is perfect in $P$, so there exists linear functionals $\{l_1,\ldots, l_c\}$ on $V$  such that 
\begin{equation}\label{E:int-li}
\langle F\rangle=\bigcap_{i=1}^c \{l_i=0\} \quad \text{ and } l_i\in P^{\vee} \text{ for any } 1\leq i \leq c. 
\end{equation}
Since $K(Y)\subseteq P$, we have that $P^{\vee}\subseteq K(Y)^{\vee}$; hence $l_i\in K(Y)^{\vee}$  for any  $1\leq i \leq c$. Then  \eqref{E:int-li} implies that $F$ is also a perfect face of $K(Y)$, i.e. (d) holds true.

To see (e) just notice that, since $\rho = r+1$, we have that $\langle {\mathcal D} \rangle = \{l = 0\}$ and therefore $\langle {\mathcal D} \rangle^{\perp} = \langle l \rangle$ and $D_{K(Y)}({\mathcal D}) = \hat {\mathcal D} = \cone(l)$. Since ${\mathcal D}$ is a perfect face of $K(Y)$ it follows by definition that $\cone(l)$ is a perfect face, that is an edge, of $K(Y)^{\vee}$.
\end{proof}
\renewcommand{\proofname}{Proof}


We note the following consequence of Proposition \ref{edges1} for cones of pseudoeffective cycles.

\begin{cor}
\label{edges2}
Let $X$ be a projective variety of dimension $d$ and let $n$ be an integer such that $1 \le n \le d -1$. Let $I$ be the set of classes in $N_n(X)$ of irreducible $n$-dimensional subvarieties of $X$. Let $\delta_1, \ldots, \delta_s \in I$ and suppose that $1 \le \dim(\cone(\delta_1,\ldots, \delta_s)) < \dim N_n(X)$. Moreover suppose that we have an ample $\R$-divisor $A$ and $\eta \in \Nef^n(X)$ such that:
\begin{itemize}
\item[(i)] $\eta \cdot \delta_i = 0$ for $1 \le i \le s$;
\item[(ii)] $(\eta - A^n) \cdot y \ge 0$ for every $y \in I \setminus \{\delta_1, \ldots, \delta_s\}$.
\end{itemize}
Then
\begin{itemize}
\item[(a)] $\Pseff_n(X)$ is locally finitely generated at every $\alpha \in \{\eta - A^n < 0\}$, in particular at every $\alpha \in \cone(\delta_1,\ldots, \delta_s), \alpha \ne 0$;
\item[(b)] $\cone(\delta_1,\ldots, \delta_s) = \{
\eta\color{black}=0\} \cap \Pseff_n(X)$;
\item[(c)] every extremal ray of $\Pseff_n(X)$ that is contained in $\{\eta - A^n < 0\}$ is $\R_{\ge 0}\delta_i$ for some $i$;
\item[(d)] every face of $\cone(\delta_1,\ldots, \delta_s)$ is a perfect face of $\Pseff_n(X)$;
\item[(e)] if $\dim(\cone(\delta_1,\ldots, \delta_s)) = \dim N_n(X) - 1$ then $\cone(\eta)$ is an edge of $\Nef^n(X)$.
\end{itemize} 
\end{cor}
\begin{proof}
Define, for $\alpha \in N_n(X)$, $l(\alpha) = \eta \cdot \alpha$ and $\phi(\alpha) = A^n \cdot \alpha$. The fact that  $\phi \in \ri(\Pseff_n(X)^{\vee})$ 
is in \cite[Cor. 3.15]{fl1}, \cite[Prop. 3.7]{fl2}. Now $K(I) = \Pseff_n(X)$ so we can apply Proposition \ref{edges1}.
\end{proof}


\section{Tate classes on Jacobians of very general curves}
\label{A}
\centerline{\scalebox{.8}{by BEN MOONEN}}

\vskip .3cm

The purpose of this appendix is to explain some facts about Tate classes on the Jacobian of a very general curve. These facts are well-known to the experts and we claim no originality.

\subsection{}\label{TateCl}
Let $k$ be a field and let $k \subset \kbar$ be an algebraic closure. If $k_0 \subset k$ is a subfield, let $\Gamma_{k_0} = \Aut(\kbar_0/k_0) = \Gal(k_0^{\sep}/k_0)$. We say that a field is of finite type if it is finitely generated over its prime field.

Let $X$ be a smooth projective scheme over~$k$, and let $\ell$ be a prime number with $\ell \neq \mathrm{char}(k)$. We want to define when a cohomology class in $\mathrm{H}^{2n}(X_{\kbar},\Ql(n))$ is a Tate class. For this, choose a model $X_0/k_0$ of~$X$ over a subfield $k_0 \subset k$ of finite type. We have a natural isomorphism $\mathrm{H}^{2n}(X_{0,\kbar_0},\Ql(n)) \isomarrow \mathrm{H}^{2n}(X_{\kbar},\Ql(n))$ via which we obtain an action of $\Gamma_{k_0}$ on $\mathrm{H}^{2n}(X_{\kbar},\Ql(n))$. Then we say that a cohomology class $c \in \mathrm{H}^{2n}(X_{\kbar},\Ql(n))$ is a Tate class if $c$ has an open stabilizer in $\Gamma_{k_0}$. It is known (see for instance~\cite{Mo}, Section~1) that this notion is independent of how we choose the model $X_0/k_0$.

The even cohomology ring $\mathrm{H}^\ev(X) = \oplus_{n\geq 0}\; \mathrm{H}^{2n}(X_{\kbar},\Ql(n))$ is a finite dimensional commutative $\Ql$-algebra. The Tate classes form a $\Ql$-subalgebra that we denote by $\mathrm{Tate}(X) \subset \mathrm{H}^\ev(X)$.

\begin{remark}
We have a cycle class map $\mathrm{CH}(X_{\kbar}) \otimes \Ql \to \mathrm{H}^\ev(X)$. Its image is contained in $\mathrm{Tate}(X)$. The Tate conjecture predicts that the image is all of $\mathrm{Tate}(X)$.
\end{remark}

\subsection{}\label{rhol}
With $k$ and~$\ell$ as before, let $X$ be an abelian variety of dimension $g>0$ over~$k$. Let $V_\ell(X) = \Ql \otimes_{\Zl} T_\ell(X)$, where $T_\ell(X)$ is the Tate-$\ell$-module of~$X$. If there is no risk of confusion, we abbreviate $V_\ell(X)$ to~$V_\ell$. The choice of a polarization~$\theta$ of~$X$ gives rise to an alternating bilinear form $\phi \colon V_\ell \times V_\ell \to \Ql(1)$.

Let $(X_0,\theta_0)$ be a model of the polarized abelian variety $(X,\theta)$ over a subfield $k_0 \subset k$ of finite type. This gives rise to an identification $V_\ell(X_0) \isomarrow V_\ell$, via which we obtain a continuous action of the Galois group $\Gamma_{k_0}$ on~$V_\ell$. The pairing~$\phi$ is equivariant with respect to this action, with $\Gamma_{k_0}$ acting on~$\Ql(1)$ through the $\ell$-adic cyclotomic character $\chi_\ell\colon \Gamma_{k_0} \to \Zl^\times$. Hence the Galois representation on~$V_\ell$ is given by a homomorphism
\begin{equation}\label{rholX0}
\rho_\ell = \rho_{X_0,\ell} \colon \Gamma_{k_0} \to \CSp(V_\ell,\phi)\, ,
\end{equation}
and if $\nu \colon \CSp(V_\ell,\phi) \to \mG_{\mult,\Ql}$ is the multiplier character then $\nu \circ \rho_\ell  = \chi_\ell$. For further general discussion of such Galois representations we recommend \cite{Chi}, \cite{SerreMcGill}, \cite{SerreKyoto}.

The cohomology ring $\mathrm{H}(X) = \oplus_{n\geq 0}\; \mathrm{H}^n(X_{\kbar},\Ql)$ is isomorphic, as graded $\Ql$-algebra with $\Gamma_{k_0}$-action, to the exterior algebra of~$V_\ell^*$. The form~$\phi$ may be 
viewed as a cohomology class $[\theta] \in \mathrm{H}^{2}(X_{\kbar},\Ql(1)) \cong \bigl(\wedge^2 V_\ell^*\bigr)(1)$. To explain the notation we note that the polarization~$\theta$ corresponds to an algebraic equivalence class of line bundles on~$X$, and $[\theta]$ is simply the associated class in cohomology. In particular, $[\theta]$ is in the image of the cycle class map, and is therefore a Tate class. (Note that in the even cohomology ring $\mathrm{H}^\ev(X)$ we do include Tate twists, so $\mathrm{H}^\ev(X) \cong \oplus_{n\geq 0}\, \bigl(\wedge^{2n}\, V_\ell^*\bigr)(n)$. This is important 
since $\Gamma_{k_0}$ acts non-trivially on $\Ql(n)$ for $n \neq 0$.)

\begin{defi}
Let us say that the polarized abelian variety $(X,\theta)$ has property~$(\Theta)$ if the sub-algebra $\mathrm{Tate}(X) \subset \mathrm{H}^\ev(X)$ is generated, as a $\Ql$-algebra, by the class~$[\theta]$.
\end{defi}

\begin{remarks}\label{Rems}
\null \hskip 5cm
\begin{enumerate}[(1)]
\item If $(X,\theta)$ has property~$(\Theta)$ then $X_{\kbar}$ has Picard number~$1$ and any other polarization class on~$X$ is a scalar multiple of~$[\theta]$. We say that an abelian variety~$X$ has property~$(\Theta)$ if $(X,\theta)$ has property~$(\Theta)$ for some (equivalently: any) polarization~$\theta$.

\item If $X_0$ is a model of $X$ over a subfield $k_0 \subset k$ then it is immediate from how we defined Tate classes that $X$ has property~$(\Theta)$ if and only if $X_0$ has property~$(\Theta)$.

\item If $X$ has property~$(\Theta)$ then clearly the cycle class map $\mathrm{CH}(X_{\kbar}) \otimes \Ql \to \mathrm{Tate}(X)$ is surjective, confirming Tate's conjecture.

\item Every elliptic curve has property~$(\Theta)$. An abelian variety of dimension $2$ or~$3$ with Picard number~$1$ over~$\kbar$ has property~$(\Theta)$.
\end{enumerate}
\end{remarks}

\begin{defi}\label{BDef}
We say that the abelian variety~$X$ has property~(B) (for ``Big Galois image'') if for some polarization~$\theta$ and for some model $(X_0,\theta_0)$ over a subfield $k_0\subset k$ of finite type, the image of the Galois representation~$\rho_{X_0,\ell}$ is Zariski-dense in the algebraic group $\CSp(V_\ell,\phi)$.
\end{defi}

It can be shown that this property is independent of which model $(X_0,\theta_0)$ is chosen. Further, to deduce property~(B) it suffices to know that the Zariski closure of the image of~$\rho_{X_0,\ell}$ contains the group $\Sp(V_\ell,\phi)$. Indeed, with $\nu$ the multiplier character as before, $\nu \circ \rho_{X_0,\ell}$ is the $\ell$-adic cyclotomic character and $k_0$ contains only finitely many $\ell$-power roots of unity; therefore the image of $\nu \circ \rho_{X_0,\ell}$ is Zariski dense in~$\mG_{\mult,\Ql}$.

\begin{prop}\label{B=>Theta}
In the situation of \emph{\ref{rhol}}, if $X$ has property~\emph{(B)} then it has property~\emph{$(\Theta)$}.
\end{prop}

\begin{proof}
Let $G = \Sp(V_\ell,\phi)$. It suffices to show that there are no $G$-invariants in the even cohomology ring of~$X$ other than the polynomials in the class~$[\theta]$. As $G$ acts trivially on~$\Ql(n)$, the even cohomology is isomorphic, as a representation of~$G$, to the even exterior algebra of~$V_\ell^*$. The claim then follows from a result of classical invariant theory: the sub-algebra of $\Sp(V_\ell,\phi)$-invariants in $\wedge\, V_\ell^*$ is generated by the class of~$\phi$ in $\wedge^2\, V_\ell^*$, which in our notation is just 
$[\theta]$.
\end{proof}

\begin{prop}\label{verygen2}
Let $S$ be a geometrically integral scheme of finite type over a field~$F$ that is of finite type. Let $X$ be an abelian scheme over~$S$, and assume the generic fibre~$X_\eta$ has property~\emph{$(\Theta)$}. Let $F \subset k$ be a field extension. Then for a very general point $s \in S(k)$ the fibre~$X_s$ has property~\emph{$(\Theta)$}.
\end{prop}

To be clear about the meaning of ``very general point'': the assertion is that there exist countably many closed subschemes $Z_i \subsetneq S_k$ such that for every $s \in S(k)$ outside the union of the~$Z_i$, the polarized abelian variety $(X_s,\theta_s)$ has property~$(\Theta)$. Note that this gives nothing if $S(k)$ 
is countable, which happens for instance when $k$ is finitely generated over the prime field.

\begin{proof}
There are countably many closed reduced subschemes $Z \subsetneq S_k$ that are defined over~$F$. Suppose $s \in S(k)$ does not lie on any of these subschemes. This means that under the morphism $S_k \to S$ the point~$s$ maps to the generic point of~$S$. If $K$ is the function field of~$S$, this gives us a field homomorphism $K \to k$ via which $(X_s,\theta_s)$ becomes isomorphic to $(X_\eta,\theta_\eta) \otimes_K k$. In other words: $X_\eta$ is a model of~$X_s$. By Remark~\ref{Rems}(2) and the assumption that $X_\eta$ has property~$(\Theta)$, it follows that $X_s$ has property~$(\Theta)$.
\end{proof}

\subsection{}\label{X/S}
Let $S$ be a geometrically irreducible scheme that is smooth and of finite type over a field~$F$. Let $\eta$ be the generic point of~$S$ and $K = \kappa(\eta)$ the function field. Choose an algebraic closure $K \subset \Kbar$, and let $\bar{\eta}$ be the corresponding geometric point of~$S$. Let $\ell$ be a prime number different from~$\mathrm{char}(F)$.

If $(X,\theta)$ is a principally polarized abelian scheme over~$S$, the modules $T_\ell(X_s)$, for $s$ a point of~$S$, are the geometric fibres of a smooth $\Zl$-sheaf~$\mathscr{T}_\ell$ on~$S$ on which we have an alternating bilinear form $\phi \colon \mathscr{T}_\ell \times \mathscr{T}_\ell \to \Zl(1)_S$. We take $\bar{\eta}$ as base point and write $T_\ell = T_\ell(X_\eta)$, which is the fibre of~$\mathscr{T}_\ell$ at~$\bar{\eta}$. We continue to write~$\phi$ for the polarization form on~$T_\ell$. The sheaf~$\mathscr{T}_\ell$ is then given by a representation
\[
\rho_{X/S,\ell} \colon \pi_1(S,\bar{\eta}) \to \CSp(T_\ell,\phi)
\]
of the \'etale fundamental group. The morphism of pointed schemes $\eta \colon (\Spec(K),\bar{\eta}) \to (S,\bar{\eta})$ induces a surjective homomorphism $\eta_* \colon \Gal(\Ksep/K) = \pi_1\bigl(\Spec(K),\bar{\eta}\bigr) \to \pi_1(S,\bar{\eta})$, and the composition $\rho_{X/S,\ell} \circ \eta_*$ is the Galois representation~$\rho_{X_\eta,\ell}$. In particular, the image of~$\rho_{X_\eta,\ell}$ is the same as the image of~$\rho_{X/S,\ell}$.

If one aims to show that the generic fibre~$X_\eta$ has property~$(\Theta)$ one may try to use a monodromy argument. For this, consider the ($\ell$-adic) geometric monodromy group $\mathscr{G}_{\mono} \subset \CSp(T_\ell,\phi)$, which is defined to be the image under~$\rho_{X/S,\ell}$ of the geometric fundamental group $\pi_1(S_{\Fbar},\bar{\eta})$. As the latter is a subgroup of $\pi_1(S,\bar{\eta})$, the image of~$\rho_{X_\eta,\ell}$ contains~$\mathscr{G}_{\mono}$. Because $\pi_1(S_{\Fbar},\bar{\eta})$ acts trivially on~$\Zl(1)$, the polarization form~$\phi$ is invariant (and not only semi-invariant) under the action of~$\mathscr{G}_{\mono}$, and hence $\mathscr{G}_{\mono} \subset \Sp(T_\ell,\phi)$. If $\mathscr{G}_{\mono}$ is Zariski-dense in $\Sp(V_\ell,\phi)$, where as before $V_\ell = T_\ell \otimes \Ql$, then by the remark after Definition~\ref{BDef} together with Proposition~\ref{B=>Theta} it follows that $X_\eta$ has property~$(\Theta)$.

\begin{example}\label{Mg}
Fix an integer $g \geq 2$. For $n$ an integer with $n\geq 3$, let $\mathscr{M}_g^{(n)}$ denote the moduli scheme over $\mZ[1/n,\zeta_n]$ (with $\zeta_n$ a primitive $n$th root of unity) of curves of genus~$g$ equipped with a symplectic level~$n$ structure. It is proven in~\cite{DeMu}, Section~5, that the geometric fibres of $\mathscr{M}_g^{(n)} \to \Spec\bigl(\mZ[1/n,\zeta_n]\bigr)$ are irreducible. To see this for the geometric fibre in characteristic~$0$ one may work over~$\mC$ and use topological methods. For the geometric fibres in positive characteristic the result is then a consequence of the existence of a smooth compactification of~$\mathscr{M}_g^{(n)}$, as this implies that all geometric fibres have the same number of components. A purely algebraic proof of this irreducibility result can be found in~\cite{Ek}.

Let $k$ be a field, $F_0$ its prime field, and $p \geq 0$ the characteristic. Let $\ell$ be a prime number, $\ell \neq p$. 
To avoid working with stacks, choose an 
integer $m\geq 3$ that is not divisible by~$p\ell$, let $F$ be a finite extension of~$F_0$ that contains a root of unity~$\zeta$ of order~$m$, and let $S = \mathscr{M}_{g,F}^{(m)}$.

We apply the discussion in~\ref{X/S} with $(X,\theta)$ the universal Jacobian scheme over~$S$. Let
$$
\bar{\rho}_i \colon \pi_1(S_{\Fbar},\bar{\eta}) \to \Sp(T_\ell/\ell^i T_\ell,\phi)
$$
be the representation of the geometric fundamental group on $T_\ell/\ell^i T_\ell \cong X_{\bar{\eta}}[\ell^i]\bigl(\Kbar\bigr)$, which is the same as the representation induced by~$\rho_{X/S,\ell}$. The corresponding Galois cover of~$S_{\Fbar}$ is
$$
\mathscr{M}_{g,\Fbar}^{(\ell^i m)} \to S_{\Fbar} = \mathscr{M}_{g,\Fbar}^{(m)}\, .
$$
The irreducibility of $\mathscr{M}_{g,\Fbar}^{(\ell^i m)}$ is equivalent to the fact that the representation~$\bar{\rho}_i$ is surjective. Taking the limit over all~$i$, we find that the geometric monodromy group~$\mathscr{G}_{\mono}$ of~$X/S$ is the full symplectic group $\Sp(T_\ell,\phi)$, and as discussed in~\ref{X/S} this implies that the generic Jacobian~$X_\eta$ has property~$(\Theta)$. Together with Proposition~\ref{verygen2} this gives the conclusion that for an arbitrary field~$k$ the Jacobian of a very general curve of genus~$g$ over~$k$ has property~$(\Theta)$.
\end{example}
\begin{remarks}
\null \hskip 5cm
\begin{enumerate}[(1)]
\item If we want to work over ``small'' fields, such as number fields, the results discussed so far do not give anything. However, for an abelian variety~$X$ over a field~$k$ of characteristic~$0$ with $\End(X_{\kbar}) = \mZ$, it can be shown by other methods that $X$ has property~(B) if a certain numerical condition on its dimension is satisfied. The first such results were obtained by Serre; an improved version by Pink is given in \cite{Pink}, Theorem~(5.14).

\item For $g\geq 2$ and $\mathrm{char}(k) \neq 2$, it is also true that the very general hyperelliptic curve of genus~$g$ over~$k$ has property~$(\Theta)$. The argument is the same as in~\ref{Mg}, using that for $\ell \notin \{2,p\}$ the geometric monodromy group is again the full symplectic group $\Sp(T_\ell,\phi)$ by \cite{AchPri}, Corollary~3.5. (For $\ell = 2$ a slight modification of the argument is needed.)
\end{enumerate}
\end{remarks}

\end{document}